\documentclass{article}
\usepackage{graphicx} 
\usepackage{amsmath, amssymb}
\usepackage{amsthm}
\usepackage{float}
\usepackage{graphicx}
\usepackage{subcaption}
\usepackage{bm}
\usepackage{enumerate}
\usepackage{comment}
\usepackage{tikz}

\usepackage[pagewise]{lineno}

\usepackage[backend=biber, style=numeric, doi=true, giveninits=true]{biblatex}
\AtEveryBibitem{
  \clearfield{number}
}

\newtheorem{theorem}{Theorem}[section]
\newtheorem{corollary}[theorem]{Corollary}
\newtheorem{proposition}[theorem]{Proposition}
 
\newtheorem{remark}[theorem]{Remark}

\newcommand{\TVtw}{\operatorname{TV}_{\widetilde{w}}}
\newcommand{\TVw}{\operatorname{TV}_w}
\newcommand{\tw}{\widetilde{w}}
\newcommand{\vnu}{{\bm{\nu}}}

\newcommand{\methodTwo}{\textcolor{black}{directionally weighted\ }}

\DeclareMathOperator*{\argmax}{arg\,max}
\DeclareMathOperator*{\argmin}{arg\,min}

\title{Directionally Weighted Total Variation for Inverse Problems}

\author{Ole L{\o}seth Elvetun\thanks{Faculty of Science and Technology, Norwegian University of Life Sciences, P.O. Box 5003, NO-1432 {\AA}s, Norway. Email: ole.elvetun@nmbu.no.} and Bj{\o}rn Fredrik Nielsen\thanks{Faculty of Science and Technology, Norwegian University of Life Sciences, P.O. Box 5003, NO-1432 {\AA}s, Norway. Email: bjorn.f.nielsen@nmbu.no.}}

\begin{document}

\maketitle
\begin{abstract}
We study weighted total variation (TV) regularization for inverse problems in which the forward operator has a large null space, a setting in which standard (unweighted) TV is known to produce systematic reconstruction artifacts such as spatial bias. To address this issue, we consider spatially varying weights derived from a sensitivity analysis of the forward operator expressed via an associated Green’s function. These weights rebalance the TV penalty and thereby compensate for the inhomogeneous sensitivity of the forward operator.

We introduce and analyze two related models: an isotropic weighted TV functional, where the weight reflects worst-case directional sensitivity, and a \methodTwo refinement in which the weighting depends on the actual jump direction of the solution, yielding a total variation of Finsler type. In addition to the standard variational formulation, we also study corresponding basis-pursuit problems. For the latter, we derive optimality conditions and prove exact recovery results under suitable structural assumptions.

Numerical experiments for inverse source problems and related applications demonstrate that the proposed weighted formulations significantly reduce artifacts present in unweighted TV reconstructions, yielding improved localization and size recovery. These results highlight the crucial role of spatial weighting in TV regularization for inverse problems with large null spaces.
\end{abstract}

\section{Introduction}

In this paper we are concerned with inverse problems for which the forward operator $K$ has a significant null space. More precisely, we assume that $K: L^2(\Omega) \rightarrow L^2(\partial\Omega)$ is a linear operator and consider the weighted problem
\begin{equation} \label{def:variational_form}
    \min_{f \in BV(\Omega)} \left\{ \frac{1}{2} \| Kf - d \|_{L^2(\partial\Omega)}^2  + \alpha \underbrace{\int_{\Omega} w(y) \, |Df| (y)}_{= \TVw(f)} + \beta \int_{\partial\Omega} w_\partial (y) |f(y)| \, ds(y) \right\}.
\end{equation}
Here, $d \in L^2(\partial\Omega)$ is the measured data, $\Omega \subset \mathbb{R}^2$ is a bounded Lipschitz domain, $\alpha > 0$ is a regularization parameter, $Df$ is the vector-valued Radon measure representation of the "gradient" of $f$, and $w$ is the weight function
\begin{equation} \label{def:w}
w(y)=\sup_{\bm d,\; |\bm d|_2 =1}
\left\|
K\!\left(
\nabla G(\cdot;y)\cdot \bm d
\right)
\right\|_{L^p(\partial\Omega)},
\end{equation}
where $G$ is a suitable Green's function which we define below. The presence of the boundary term, i.e., the term involving the parameter $\beta$, is motivated by the following observation: Without this term, it becomes "cheap" to obtain a good data fit with a source which is (approximately) constant close to $\partial \Omega$ since the total variation of a constant function is zero. Hence, a boundary-bias occurs. We will return to this issue below and explain how it is linked to the choice of boundary conditions to use in the definition of the Green's function $G$.  

Operators that map interior distributions to boundary observations, i.e., $K \colon L^2(\Omega) \to L^2(\partial \Omega)$, arise naturally in inverse source problems for elliptic partial differential equations, see~\cite{isakov2006inverse,elbadia2000inverse} and the references therein. Concrete examples include source identification in groundwater modelling, the inverse problem of electrocardiography~\cite{macleod1998recent}, and certain regimes of electrical impedance tomography~\cite{borcea2002electrical}. A defining feature of such operators is that the dimension of the null space is typically infinite: many distinct interior sources produce identical boundary data. Stable reconstruction therefore requires regularization, and the regularization functional plays a particularly active role since it must determine the component of the solution that lies in $\mathcal{N}(K)$.

When $K$ has a substantial null space, the conventional choice $w \equiv 1$ in \eqref{def:variational_form} is known to introduce systematic artifacts in the reconstructed solution. Since each unit of total variation produces a far larger reduction of the data-fidelity term in regions where $K$ is highly sensitive than in regions where $K$ is nearly blind, an unweighted TV penalty effectively becomes cheaper, relative to the fidelity term, in the sensitive regions. As a consequence, recovered features tend to be displaced toward those regions, and sources of equal magnitude may be reconstructed with different intensities - if at all - depending on their location relative to $\partial \Omega$. These observations motivate spatially varying weights $w$ that rebalance the TV penalty according to the non-uniform sensitivity of $K$.

Spatially weighted TV functionals have been considered by several authors and in several contexts. Anisotropic and direction-aware TV functionals are studied in~\cite{grasmair2010anisotropic,bayram2012directional}, and structure-guided variants tailored to multi-channel and multi-modal data appear in~\cite{ehrhardt2016vectorial,ehrhardt2016mri}. Specifically for inverse problems with significant null spaces, weighted Tikhonov and sparsity formulations have been investigated previously, and the anisotropic (i.e., $|Df|_1$) weighted TV counterpart to problem~\eqref{def:variational_form} was introduced and analyzed in~\cite{burger2025weightedtv}. The present paper builds on that work in two main respects: (i) it replaces the axis-aligned weights in the anisotropic setting with taking the supremum over all possible directions, and (ii) it introduces and analyzes the \methodTwo refinement~$\TVtw$, in which the worst-case supremum over directions is replaced by the actual jump direction of the candidate solution, see \eqref{def:variational_form_pre} and \eqref{def:tw} below.

The form of the weight \eqref{def:w} can be motivated as follows. Using the representation of $K$ in terms of $G$, the action of $K$ on a localized perturbation of $f$ near a point $y$, in direction $\bm d$, is to leading order given by $K(\nabla G(\cdot;y)\cdot \bm d)$. The weight $w(y)$ thus measures the maximal sensitivity of the data to a unit perturbation of the gradient of $f$ at $y$, with the supremum taken over all admissible directions. The role of $w$ in \eqref{def:variational_form} is thus to remove spatial bias: where $K$ is weakly sensitive, $w(y)$ is small and the TV penalty becomes inexpensive, so the minimizer is allowed to develop jumps and other structure there. Conversely, in regions of high sensitivity each unit of total variation already produces a large change in $Kf$, so a larger value of $w(y)$ is needed to prevent the data-fidelity term from biasing reconstructions toward those regions. In this way, the weight rebalances the cost of total variation against the heterogeneous sensitivity of $K$ across $\Omega$, allowing the reconstruction to follow the data wherever the data are informative and the prior wherever they are not.

Our investigation will also suggest the following alternative to \eqref{def:variational_form}, involving the polar decomposition $Df = \vnu_f |Df|$, $|\vnu_f|_2 = 1$,
\begin{equation} \label{def:variational_form_pre}
    \min_{f \in BV(\Omega)} \left\{ \frac{1}{2} \| Kf - d \|_{L^2(\partial)}^2  + \alpha \underbrace{\int_{\Omega}
\tw(y,f) \, |Df|(y)}_{= \TVtw(f)} + \beta \int_{\partial\Omega} w_\partial (y) |f(y)| \, ds(y)\right\},
\end{equation}
where the weight function in this case reads
\begin{equation} \label{def:tw}
    \tw(y,f) = \left\|
K\!\left(
\nabla G(\cdot;y)\cdot \vnu_f(y)
\right)
\right\|_{L^p(\partial\Omega)}.
\end{equation}
Below we will simply write $\vnu$, instead of $\vnu_f$, whenever it is clear from the context with which function $f$, $\vnu$ is associated. 
The crucial difference between \eqref{def:w} and \eqref{def:tw} is that the supremum over $\bm d$ in \eqref{def:w} is replaced by the actual jump direction $\vnu(y)$ of the candidate $f$. While \eqref{def:w} treats every direction equally and thereby penalizes the worst-case sensitivity at $y$, the formulation in \eqref{def:tw} is sensitive to the direction of the level curves: variations of $f$ in different directions are penalized differently, in accordance with how strongly $K$ responds to perturbations in those directions. Setting $\phi(y, \vnu) := \|K(\nabla G(\cdot;y) \cdot \vnu)\|_{L^p(\partial\Omega)}$, the map $\vnu \mapsto \phi(y,\vnu)$ is the composition of a linear map with the $L^p$ norm and is therefore positively $1$-homogeneous and sublinear on $\mathbb{R}^2$ for every $y$. Consequently,
\begin{equation*}
   \TVtw(f) \;=\; \int_\Omega \phi(y, \vnu(y)) \, d|Df|(y)
\end{equation*}
fits into the framework of anisotropic, position-dependent total variation functionals, i.e., Finsler TV \cite{amar1994notion},
\begin{equation} \label{def:phi_TV}
   f \mapsto \int_\Omega \phi(y, dDf(y)),
\end{equation}
in which $\phi: \overline{\Omega} \times \mathbb{R}^N \to [0,\infty)$ is positively $1$-homogeneous in its second argument; despite the appearance of $\vnu$ in the integrand, $\TVtw(f)$ is just the polar-decomposition representation of \eqref{def:phi_TV} for our specific choice of $\phi$. Functionals of this type are convex and lower semicontinuous on $BV(\Omega)$ under mild assumptions on $\phi$, and have been studied since the 1990s in the context of relaxation and integral representation~\cite{bouchitte1998global,amar2008relaxation,decicco2007relaxation}, as well as in connection with weighted-anisotropic TV regularization for image reconstruction~\cite{jalalzai2014discontinuities}. The class \eqref{def:phi_TV} continues to attract attention: the gradient flow associated with general convex linear-growth integrands $f(x,\mathbb{A}u)$ and the corresponding BV-relaxation are studied in~\cite{meyer2026total}, and a general lower semicontinuity and existence theory for anisotropic TV functionals with position-dependent integrand is developed in~\cite{ficola2024lower}. The closely related directional and structure-guided TV functionals appear in~\cite{bayram2012directional,ehrhardt2016vectorial}. 

Finally, by construction one has the pointwise inequality $\TVtw(f) \le \TVw(f)$, with equality at points where the jump direction $\vnu(y)$ realizes the supremum in~\eqref{def:w}.

In addition to \eqref{def:variational_form} and \eqref{def:variational_form_pre}, we also consider their basis-pursuit counterparts, in which the data fidelity term is replaced by a constraint:
\begin{equation*}
    \min_{f \in BV(\Omega)} \TVw(f) \quad \text{subject to} \quad K f = d,
\end{equation*}
and analogously for $\TVtw$. 

\paragraph{Contributions and outline.} The purpose of this paper is to 
(i) derive the weights $w$ and $\tw$ from a sensitivity argument involving the Green's function $G$ associated with $K$, (ii) establish some recoverability properties of the basis pursuit counterparts to \eqref{def:variational_form} and \eqref{def:variational_form_pre}, and (iii) illustrate, by means of numerical experiments, that the weighted formulations correct several of the artifacts produced by standard isotropic TV when $K$ has a significant null space. 

The remainder of the paper is organized as follows. Section~\ref{sec:motivating} provides a motivating example where unweighted total variation fails to recover an interior source. Section~\ref{sec:preliminaries} collects the necessary background on the Green's function setting, and the detailed motivation and analysis of $w$ and $\tw$. In Section~\ref{sec:analysis} we present the optimality conditions and establish some recoverability criteria for the basis-pursuit counterparts of \eqref{def:variational_form} and \eqref{def:variational_form_pre}. Section~\ref{sec:numerics} reports numerical experiments illustrating the theoretical results, and Section~\ref{sec:conclusion} contains concluding remarks.

\section{Motivating example} \label{sec:motivating}

Here we present a numerical example showing that standard unweighted isotropic TV fails to yield adequate results for identifying an internal source from boundary measurements. Throughout this section we let $\Omega = (0,1)^2$ and take $K$ to be the operator that maps a source $f \in L^2(\Omega)$ to the boundary trace $u|_{\partial\Omega}$ of the solution $u \in H^1(\Omega)$ of the screened-Poisson problem
\begin{equation}\label{eq:screenedPoisson}
    -\Delta u + u = f \quad \text{in } \Omega, \qquad \partial_{\bm{n}} u = 0 \quad \text{on } \partial \Omega,
\end{equation}
for which the analysis below applies. Boundary data $d = K f^\dagger$ are generated synthetically from a known/true piecewise constant source $f^\dagger$, and we then solve \eqref{def:variational_form_pre} with $\tw \equiv 1$. The true source is the indicator function of a disk centered at $(0.5, 0.6)$ with radius $0.3$, see Figure~\ref{fig:unweighted}(a).

The reconstructions obtained with unweighted TV are shown for $\beta = 0$ (i.e., without boundary penalty) in Figure~\ref{fig:unweighted}(b) and for $\beta = \alpha$ (i.e., with boundary penalty) in Figure~\ref{fig:unweighted}(c). Two effects are immediately visible. First, the recovered support is pulled toward the boundary $\partial \Omega$: the reconstructed mass is concentrated where it produces the largest boundary signal per unit total variation, rather than at the true location of the source. This is particularly outspoken when $\beta = 0$. Second, the magnitude of the reconstruction is markedly smaller than that of $f^\dagger$. Both effects originate from the spatial inhomogeneity of $K$: in the unweighted formulation, every unit of $|Df|$ is charged the same price regardless of where it is placed, even though a unit of $|Df|$ near $\partial\Omega$ explains far more of $d$ than a unit of $|Df|$ deep in the interior. The minimizer therefore prefers configurations near the boundary, and accepts a global reduction in amplitude in order to lower the regularization cost. 

\begin{figure}[H]
    \centering
    \begin{subfigure}[t]{0.32\linewidth}
        \centering
        \includegraphics[trim=0 0 0 19, clip, width=\linewidth]{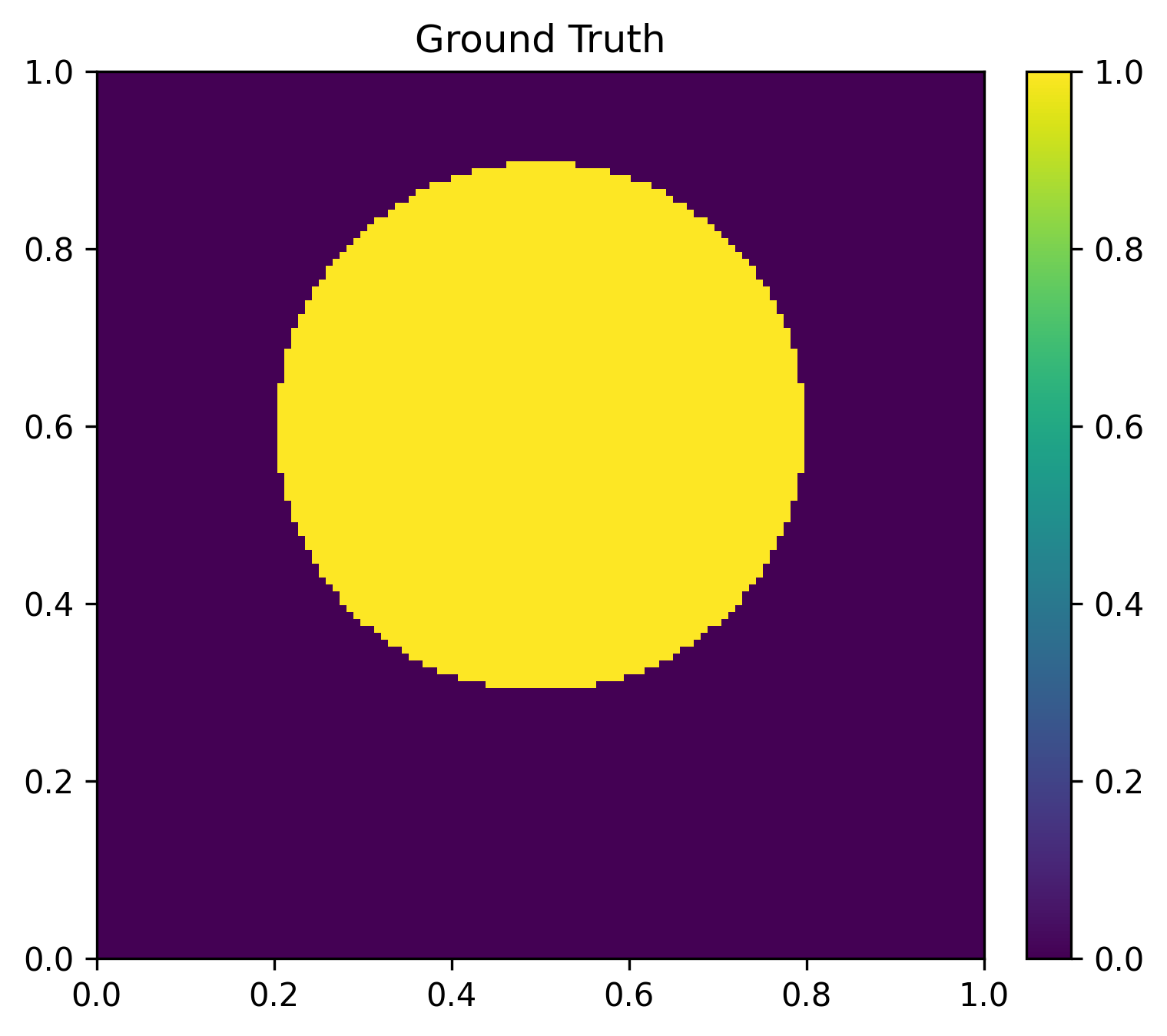}
        \caption{True source}
    \end{subfigure}
    \begin{subfigure}[t]{0.32\linewidth}
        \centering
        \includegraphics[trim=0 0 0 19, clip, width=\linewidth]{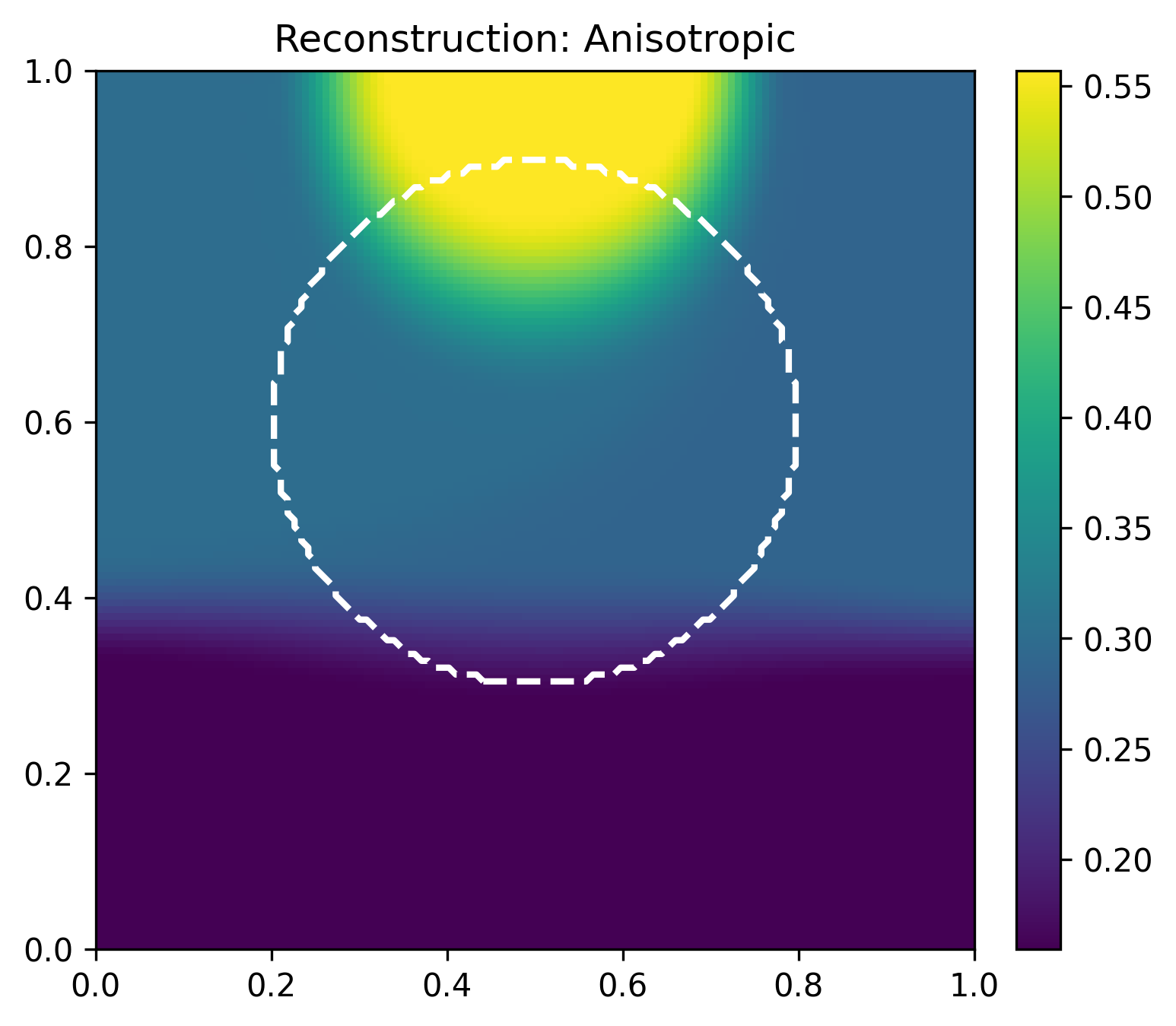}
        \caption{Unweighted; $\beta=0$}
    \end{subfigure}
    \begin{subfigure}[t]{0.32\linewidth}
        \centering
        \includegraphics[trim=0 0 0 19, clip, width=\linewidth]{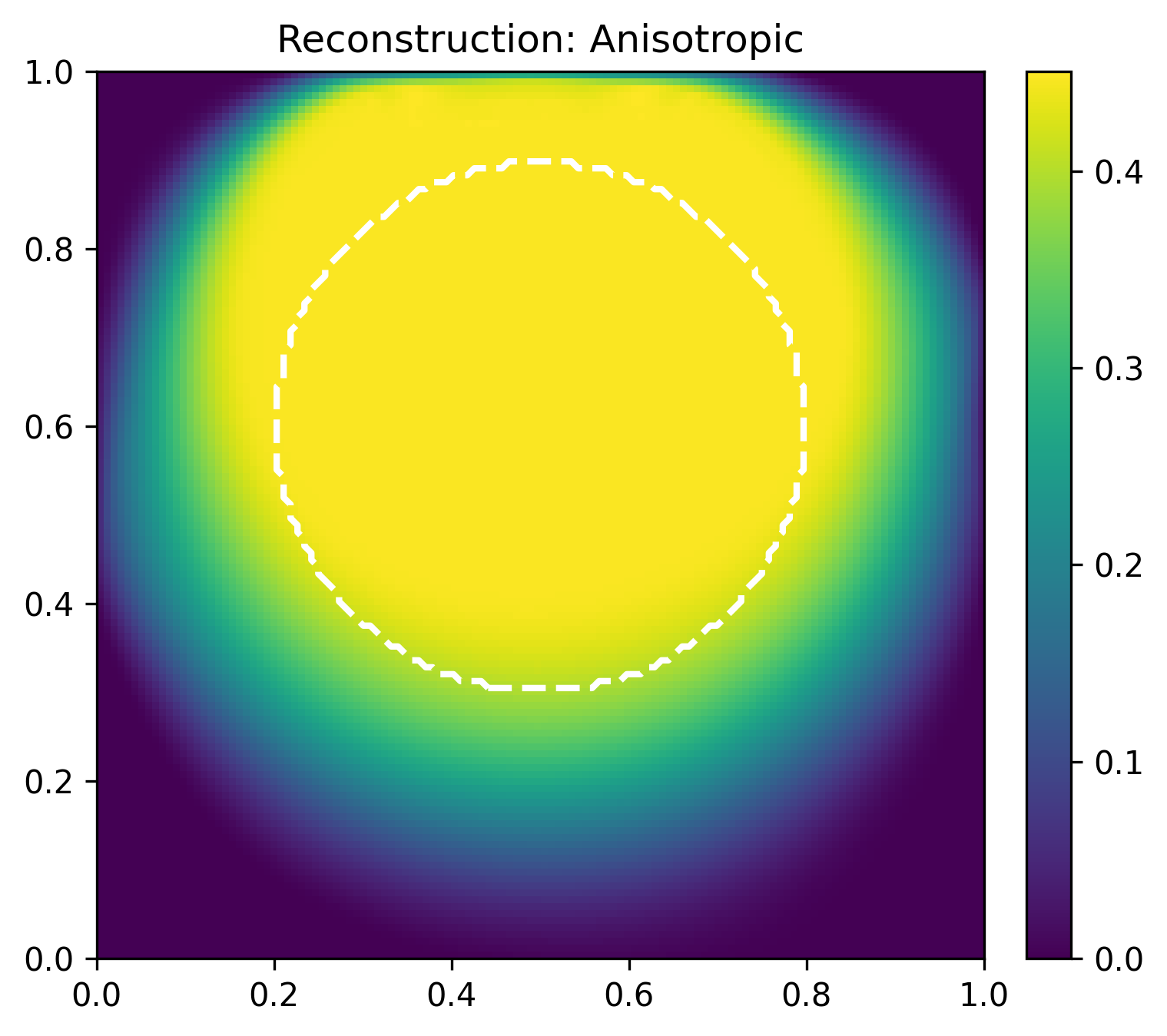}
        \caption{Unweighted; $\beta = \alpha$}
    \end{subfigure}
    \caption{Comparison of the true source and unweighted inverse recoveries (i.e., $\tw = 1$ in \eqref{def:variational_form_pre}). Panels (b) and (c) show results computed without and with boundary penalty, respectively. 
    }
    \label{fig:unweighted}
\end{figure}

This example sets the stage for the rest of the paper. It shows that, when the forward operator has a significant null space, the choice of regularization functional must account for the spatial sensitivity of $|Df|$, which serves as the main motivation for introducing the weight functions $w$ and $\tw$. 

\section{Preliminaries}\label{sec:preliminaries}
We will now derive how the source $f$ can be represented in terms of the Green's function of the Laplace operator and the gradient $\nabla f$. These standard results then motivate the definition of our weights and can be used to establish some basic inequalities. 

One interesting point is that the choice of boundary conditions for the Green's function influences the weights and the choice could - and perhaps should - be based on whether one wants to penalize jumps at the boundary or not, i.e., whether $\beta = 0$ or $\beta > 0$ in \eqref{def:variational_form} and \eqref{def:variational_form_pre}. We will return to this after introducing the framework.

\subsubsection*{Neumann representation}
Let $G$ denote the Green's function associated with a standard self-adjoint second order elliptic operator subject to homogeneous Neumann boundary conditions: 
\begin{equation} \label{def:Greens_function}
\begin{split}
    &-\Delta_y G(y;x)=\delta_x(y)-\frac{1}{|\Omega|} \quad y \in \Omega, \\
    &\nabla_y G(y;x) \cdot \bm{n}(y) = 0 \quad y \in \partial \Omega, 
\end{split}    
\end{equation}
and note that we may express the source $f$ in the form, using integration by parts, 
\begin{align}
    \nonumber
    f(x) &= \int_\Omega \delta_x(y) f(y) \, dy \\
    \nonumber
    &=\int_{\Omega} \nabla G(y;x) \cdot  Df(y) 
    + \frac{1}{|\Omega|} \int_\Omega f(y) \, dy \\
    \nonumber
    &=\int_{\Omega} \nabla G(y;x) \cdot \vnu(y) \, |Df|(y)   
    + \frac{1}{|\Omega|} \int_\Omega f(y) \, dy\\
    \label{eq:source_representation}
    &=\int_{\Omega} \nabla G(x;y) \cdot \vnu(y) \, |Df|(y)   
    + \frac{1}{|\Omega|} \int_\Omega f(y) \, dy,   
\end{align}
where we have employed the polar decomposition $Df = \vnu |Df|$, $| \vnu |_2 = 1$, and the symmetry of the Green's function $G$. We also observe that this representation is independent of the free constant of the solution of \eqref{def:Greens_function}. 

Let us, for the sake of simplicity, assume that $K1=0$, i.e., that the null space of $K$ contains the constant functions. If not, we could consider the composite function $Q = P \circ K$, where $P$ projects any $y \in R(K)$ onto the orthogonal complement of the image under $K$ of the constants and then define the weights in \eqref{def:w} and \eqref{def:tw} using $Q$ rather than $K$. See \cite{burger2025weightedtv} for details. 

From \eqref{eq:source_representation} we find that, keeping in mind that $K$ is a linear operator,  
\begin{align*}
Kf
&= K \int_{\Omega} \nabla G(\cdot;y) \cdot \vnu(y) \, |Df|(y) \\[0.5em]
&= \int_{\Omega} 
K\!\left(
\nabla G(\cdot;y)\cdot \vnu(y)
\right) \, |Df|(y) ,
\end{align*}
and we get the following bound for the $L^p$-norm of the image of $f$ under $K$: 
\begin{align}
\nonumber
\|Kf\|_{L^p(\partial\Omega)}
&= \left(\int_{\partial\Omega} \left| \int_\Omega K\!\left(
\nabla G(\cdot;y)\cdot \vnu(y)
\right) \, |Df|(y) \right|^p ds(z) \right)^{1/p} \\[0.5em]
\nonumber
&\le \int_{\Omega} \left(
\int_{\partial\Omega}
\left|
K\!\left(
\nabla G(\cdot;y)\cdot \vnu(y)
\right)
\right|^p ds(z) \right)^{1/p} \, |Df|(y) \\[0.5em]
\nonumber
&= \int_{\Omega}
\left\|
K\!\left(
\nabla G(\cdot;y)\cdot \vnu(y)
\right)
\right\|_{L^p(\partial\Omega)} \, |Df|(y) \\[0.5em]
\label{eq:upper_bound_pre}
&= \TVtw(f), 
\end{align}
provided that $\tw$ is defined according to \eqref{def:tw}. 

It follows from the definitions \eqref{def:w} and \eqref{def:tw} that 
\begin{equation*}
   \tw(y,f) \leq w(y)  \quad \forall y \in \Omega, \, \forall f \in BV(\Omega), 
\end{equation*}
and hence that 
\begin{equation*}
    \TVtw(f) \leq \TVw(f) \quad \forall f \in BV(\Omega).     
\end{equation*}
We can therefore in view of \eqref{eq:upper_bound_pre} conclude that also 
\begin{align}
\label{eq:upper_bound}
\|Kf\|_{L^p(\partial\Omega)}
&\leq \TVw(f) \quad \forall f \in BV(\Omega). 
\end{align}

\subsubsection*{Dirichlet representation}

If we instead consider the Green's function for the Laplace operator with homogeneous Dirichlet boundary conditions, we get, in contrast to \eqref{eq:source_representation}, an alternative representation of the source $f$, namely
\begin{equation}\label{label:source_rep_dirichlet}
    f(x) = \int_\Omega \nabla G(x;y) \cdot \vnu(y)|Df|(y) - \int_{\partial\Omega} \partial_{\bm n}G(x;y) \ f(y) \, ds(y). 
\end{equation}

Assuming $f$ to have zero trace, we get exactly the same upper bounds as derived in \eqref{eq:upper_bound_pre} and \eqref{eq:upper_bound}. If not, we get, from an argument analogous to \eqref{eq:upper_bound_pre}, that
\begin{equation*}
    \|Kf\|_{L^p(\partial\Omega)} \leq \TVtw(f) + \int_{\partial\Omega} w_\partial |f| \, ds \leq \TVw(f) + \int_{\partial\Omega} w_\partial |f| \, ds,
\end{equation*}
where $$w_\partial(y) = \left(\int_{\partial\Omega} |K (\partial_{\bm n} G(\cdot;y))|^p ds(z)\right)^{1/p}.$$

\begin{remark}
    Note that the weight for the boundary term arises naturally when applying Dirichlet boundary conditions to the Green's function. This is not the case for Neumann conditions. This motivates the use of the former when one wants to apply boundary penalty, i.e., when $\beta > 0$.
\end{remark}

\section{Analysis}\label{sec:analysis}
In this section we will first consider some instructive cases where we can guarantee exact recovery of a given source, before we view the problems through the lens of more classical optimality conditions.

\subsection{Exact recovery}
We will now present an exact recovery result for the zero regularization counterpart to \eqref{def:variational_form_pre}: 
\begin{theorem} \label{thm:exact_recovery_pre}
    Let $p = 1$ and $\beta = 0$. Assume that $K1=0$ and let $Df^* = \vnu^* |Df^*|$ be the polar decomposition of the BV-derivative of the true source $f^*$. If $f^*$ satisfies, for every $z \in \partial\Omega$, 
\begin{equation} \label{eq:assumption1a}
\operatorname{sgn}\!\left\{
(K\!\left[\nabla G(\cdot;y)\cdot \vnu^*(y)
\right])(z)
\right\}
\geq 0
\qquad \forall y \in \Omega 
\end{equation}
or 
\begin{equation} \label{eq:assumption1b}
\operatorname{sgn}\!\left\{
(K\!\left[\nabla G(\cdot;y)\cdot \vnu^*(y)
\right])(z)
\right\}
\leq 0
\qquad \forall y \in \Omega , 
\end{equation}
then 
\begin{equation} \label{def:basis_pursuit_pre}
    f^* \in \arg\min_{g \in BV(\Omega)} \TVtw(g) \textnormal{ subject to } Kg = K f^*.
\end{equation}
\end{theorem}
\begin{proof}
Assumptions \eqref{eq:assumption1a} and \eqref{eq:assumption1b} imply that we can replace, when $f=f^*$, the inequality in \eqref{eq:upper_bound_pre} with equality, i.e., 
   \begin{equation*}
       \TVtw(f^*) = \| K f^* \|_{L^1(\partial \Omega)}. 
   \end{equation*}
Furthermore, for any $g$ satisfying $Kg=K f^*$ we can invoke \eqref{eq:upper_bound_pre} with $f=g$ and conclude that 
   \begin{align*}
       \TVtw(f^*) = \| K f^* \|_{L^1(\partial \Omega)} = \| K g \|_{L^1(\partial \Omega)} \le \TVtw(g).  
   \end{align*}
\end{proof}

The following quantity plays an important role in our analysis  of \eqref{def:variational_form}, 
\begin{equation} \label{def:optimal_direction}
\bm q(y)
\in \argmax_{\bm d,\;\|\bm d\|=1}
\left\|
K\!\left(\nabla G(\cdot;y)\cdot \bm d
\right)
\right\|_{L^1(\partial \Omega)} \quad y \in \Omega .
\end{equation}
We remark that $\bm q(y)$ is not unique and that Theorem \ref{thm:exact_recovery} below holds for any $\bm q(y)$ satisfying \eqref{def:optimal_direction}, provided that also \eqref{eq:assumption1a}-\eqref{eq:assumption2} hold. 
Roughly speaking, $\bm{q}(y)$ is a dominant direction of $\nabla G(\cdot;y)$ under the image of $K$.  
Using this concept, our exact recovery result for the basis pursuit problem associated with \eqref{def:variational_form} reads: 
\begin{theorem} \label{thm:exact_recovery}
Let $p = 1$ and $\beta = 0$. Assume that $K1=0$ and that $f^*$, where $Df^* = \vnu^* |Df^*|$, satisfies \eqref{eq:assumption1a}-\eqref{eq:assumption1b}. If  
\begin{equation} \label{eq:assumption2}
\vnu^*(y) = \bm q(y)
\qquad |Df^*|-a.e.,
\end{equation}
then 
\begin{equation} \label{def:basis_pursuit}
    f^* \in \argmin_{g \in BV(\Omega)} \TVw(g) \textnormal{ subject to } Kg = K f^*.
\end{equation}
\end{theorem}
\begin{proof}
As in the proof of Theorem \ref{thm:exact_recovery_pre}, we find that 
\begin{equation*}
    \TVtw(f^*) = \| K f^* \|_{L^1(\partial \Omega)}. 
\end{equation*}
Assumption \eqref{eq:assumption2} implies that 
\begin{align*}
    \tw(y,f^*) &= \left\|
K\!\left(
\nabla G(\cdot;y)\cdot \vnu^*(y)
\right)
\right\|_{L^1(\partial \Omega)} \\
&= \left\|
K\!\left(
\nabla G                    (\cdot;y)\cdot \bm{q}(y)
\right)
\right\|_{L^1(\partial \Omega)} \\ 
&= w(y).
\end{align*}

For any $g$ satisfying $Kg=K f^*$ we thus find that  
   \begin{align*}
       \TVw(f^*) = \TVtw(f^*) = \| K f^* \|_{L^1(\partial \Omega)} = \| K g \|_{L^1(\partial \Omega)} \le \TVw(g),  
   \end{align*}
where the inequality follows from \eqref{eq:upper_bound}. 
\end{proof}

\subsubsection*{Example 1}
Let us consider a 1D problem with $\Omega=(0,1), \ p = 1, \ \beta = 0$ and assume that $f^*$ is strictly increasing on $\Omega$. Then $|Df^*|(y) = (Df^*)(y)$ for all $y \in \Omega$. Consequently, $$\vnu^*(y)= 1 \in \argmax_{d,\;\|d\|=1}
\left\|
K\!\left(G'(\cdot;y)\cdot d
\right)
\right\|_{L^1(\partial \Omega)} \quad y \in \Omega ,$$
and it follows that \eqref{eq:assumption2} holds with $q(y)=1$, $y\in \Omega$. We conclude that \eqref{def:basis_pursuit} holds for strictly increasing functions provided $K$ and $G$ are such that \eqref{eq:assumption1a} or \eqref{eq:assumption1b} is satisfied with $\vnu^*(y)=1$. The same argument can, of course, be made for strictly decreasing functions.   

\subsubsection*{Example 2}
We consider a true source in 1D with a single jump: 
\begin{equation}
    f^*(x) 
    = \left\{ 
    \begin{array}{cc}
       r-1,  & x<r, \\
       r,  &   x>r,
    \end{array}\right. . \label{eq:heaviside} 
\end{equation}
where $r \in (0,1)$ is fixed.
Also, $\Omega=(0,1), \ p = 1$ and $\beta = 0$. Then, cf. \eqref{eq:upper_bound}, 
\begin{align*}
\nonumber
\|Kf^*\|_{L^1(\partial\Omega)}
&= \int_{\partial\Omega} \left| \int_\Omega K\!\left(
\nabla G(\cdot;y)\cdot \vnu^*(y)
\right) \, |Df^*|(y) \right| dz \\[0.5em]
&= \int_{\partial\Omega} \left| \int_\Omega K\!\left(G'(\cdot;y)\cdot \nu^*(y)
\right) \, d \delta_r (y) \right| dz \\
&= \int_{\partial\Omega} \left| K\!\left(G'(\cdot;r)\cdot 1
\right) \right| dz \\
&=w(r). 
\end{align*}
Moreover, 
\begin{align*}
    TV_w(f^*) &= \int_{\Omega} w(y) \, |Df^*| (y) \\
    &= \int_{\Omega} w(y) \, d \delta_r (y) \\
    &= w(r). 
\end{align*}
Consequently, if $Kg=Kf^*$, then 
\begin{align*}
    TV_w(f^*)=w(r)=\|Kf^*\|_{L^1(\partial \Omega)}=\|Kg\|_{L^1(\partial \Omega)} \leq TV_w(g), 
\end{align*}
where we have used \eqref{eq:upper_bound}, and it follows that $f^*$ obeys \eqref{def:basis_pursuit}. In this example $\vnu^*(y)=0$ for $y \neq r$ and $\vnu^*(r)=1$ and therefore  \eqref{eq:assumption1a}, \eqref{eq:assumption1b} and \eqref{eq:assumption2} all hold. 

\subsection{Optimality conditions}
The optimality conditions associated with the noise-free limit version of \eqref{def:variational_form_pre} are presented in the proposition below. To this end, we define the polar $F^\circ$ of a convex, lower semi-continuous and 1-homogeneous function $F$ by
\begin{equation*}
    F^\circ(y, \Psi(y)) = \sup_{\theta \neq 0} \left\{\frac{\theta \cdot \Psi(y)}{F(y,\theta)}: \theta \in \mathbb{R}^2 \right\}
\end{equation*}
\begin{proposition} \label{prop:opt_gen}
 The function $f^* \in BV(\Omega)$ is a minimizer of 
 \begin{equation} \label{def:basis_pursuit_Lagrange}
     \min_{f \in BV(\Omega)} \int_\Omega \tw(y,\vnu(y))|Df|(y) + \beta \int_{\partial\Omega} w_\partial |f| \, d\sigma \quad \textnormal{s.t.} \quad Kf = Kf^*
 \end{equation}
 if and only if there exists a Lagrange multiplier $\lambda \in L^q(\partial\Omega)$, $\frac{1}{p}+\frac{1}{q} = 1$ and a vector field $\bm{z} \in L^\infty(\Omega)$ with $\nabla \cdot \bm{z} \in L^2(\Omega)$ such that the following conditions hold:
 \begin{enumerate}[(i)]
     \item \textbf{Stationarity condition:} $$K^* \lambda = -\nabla \cdot \bm{z} \quad \textnormal{in} \ \Omega,$$ 
     \item \textbf{Dual constraints:}
       \begin{enumerate}[1.] 
         \item \textbf{Interior:} $\ \ \ \tw^\circ(y,\bm{z}(y)) \leq 1, \quad y \in \Omega$,
         \item \textbf{Boundary:} $|\bm{z}(y) \cdot \bm{n}| \leq \beta w_\partial(y), \quad  y \in \partial\Omega .$
        \end{enumerate}
     \item \textbf{Exact saturation / Complementary slackness:}
        \begin{enumerate}[1.]
         \item \textbf{Interior alignment:} $$\bm{z} \cdot \vnu_{f^*} = \tw(y, \vnu_{f^*}), \quad |D f^*|-a.e.$$
         \item \textbf{Boundary alignment:} $(\bm{z} \cdot \bm{n})f^* = - \beta w_\partial|f^*| \quad \textnormal{on} \ \partial\Omega.$
        \end{enumerate}
 \end{enumerate}
\end{proposition}

\begin{proof}
    The result is an instance of Fenchel–Rockafellar duality for the weighted total variation functional; the dual vector field $\bm z$ and the saturation conditions follow the  predual formulation of TV \cite{chambolle2010introduction}, and condition (i) is the source condition in the sense of \cite{burger2004convergence}. The derivation of the proof is therefore omitted, but we present how these conditions can be used to verify that $f^*$ is indeed a minimizer. First, note the general Fenchel inequality
    \begin{equation*}
        \bm{\psi}(y) \cdot \vnu(y) \leq F(y,\vnu(y)) \quad \textnormal{whenever} \quad F^\circ(y,\bm{\psi}(y)) \leq 1.
    \end{equation*}

    Now, let $g \in BV(\Omega)$ be any feasible solution, i.e., $Kg = Kf^*$. Using the Fenchel inequality and all of the conditions above, we get
    \begin{align*}
        \int_\Omega \tw(y,\vnu_g(y))\,|Dg|(y) + \beta \int_{\partial\Omega} w_\partial|g| \, ds &\geq
        \int_\Omega (\bm{z} \cdot \vnu_g)\,|Dg|(y) + \int_{\partial\Omega} |\bm{z}\cdot\bm{n}|\,|g| \, ds \\ &\geq 
        \int_\Omega \bm{z}\cdot Dg(y) - \int_{\partial\Omega}(\bm{z}\cdot \bm{n})g \, ds \\ &=
        -\int_\Omega (\nabla\cdot \bm{z}) g \, dy + \int_{\partial\Omega} (\bm{z}\cdot \bm{n})g \, ds - \int_{\partial\Omega}(\bm{z}\cdot \bm{n})g \, ds \\ &=
        \langle K^*\lambda, g\rangle = \langle \lambda, Kg\rangle = \langle \lambda, Kf^*\rangle = \langle K^*\lambda, f^*\rangle \\ &=
        -\int_\Omega (\nabla \cdot \bm{z})f^* \, dy \\ &=
        \int_\Omega \bm{z} \cdot Df^*(y) - \int_{\partial\Omega} (\bm{z}\cdot \bm{n})f^* \, ds \\ &= 
        \int_\Omega (\bm{z} \cdot \vnu_{f^*})\,|Df^*|(y) - \int_{\partial\Omega} (\bm{z}\cdot \bm{n})f^* \, ds \\ &=
        \int_\Omega \tw(y,\vnu_{f^*})\,|Df^*|(y) + \beta \int_{\partial\Omega} w_\partial|f^*| \, ds
    \end{align*}
\end{proof}

Note that if $\beta = 0$ the dual constraint for the boundary implies that  $\bm{z} \cdot \bm{n} = 0$ along $\partial \Omega$.

\begin{remark}
    What might seem to be the most challenging part of the optimality conditions to fulfill is the Euler-Lagrange equation. However, for any $\lambda \in L^q(\partial\Omega)$ we have (in a distributional sense via $\varphi \in C^1_c(\Omega)$) the representation 
    \begin{align*}
        \langle K^*\lambda, \varphi \rangle &= \langle \lambda, K\varphi \rangle 
        \\ &= \left\langle \lambda, \int_\Omega K(\nabla G(\cdot;y)) \cdot D\varphi(y)  \right\rangle \\ &= \int_\Omega\left\langle \lambda, K(\nabla G(\cdot;y)) \right\rangle
        \cdot D\varphi(y) \\ &= -\int_\Omega \nabla \cdot \langle \lambda, K(\nabla G(\cdot;y)) \rangle \, \varphi(y)\, dy \\ &= - \int_\Omega (\nabla \cdot\bm{z}_\lambda(y)) \, \varphi(y) \, dy,
    \end{align*}
    where we have defined 
    \begin{equation} \label{eq:z_formula}
        \bm{z}_\lambda(y) = \langle \lambda,K(\nabla G(\cdot;y)) \rangle.
    \end{equation} 
    Taking the Euclidean inner product of \eqref{eq:z_formula} with $\theta \in \mathbb{R}^2$ and applying H\"older's inequality give
    \begin{align*}
        \bm{z}_\lambda(y)\cdot\theta &= \langle \lambda, K(\nabla G(\cdot;y)) \rangle \cdot \theta \\ &= \langle \lambda, K(\nabla G(\cdot;y) \cdot \theta) \rangle \\ &\leq \|\lambda\|_{q} \|K(\nabla G(\cdot,y) \cdot \theta) \|_{p} \\ &= \|\lambda\|_{q}\tw(y,\theta).
    \end{align*}
    Taking supremum over $\theta \neq 0$ yields
    \begin{equation*}
        \tw^\circ(y,\bm{z}(y)) \leq \|\lambda\|_q \quad \forall \, y \in \Omega.
    \end{equation*}
    In particular, if $\|\lambda\|_q\le 1$, then the associated field $\bm{z}_\lambda$ satisfies the interior dual constraint
    $\tw^\circ(y,\bm{z}_\lambda(y))\le 1$ uniformly in $y$.
    Thus, for the chosen weight $\tw$, interior dual feasibility reduces to a global norm bound on the Lagrange multiplier rather than a pointwise requirement. 
\end{remark}

From these considerations, we get the simplified optimality conditions
\begin{corollary}\label{col:simp_cond}
 Let $\tw$ be defined as in \eqref{def:tw}. Then, the function $f^* \in BV(\Omega)$ is a minimizer of \eqref{def:basis_pursuit_Lagrange}
 if there exists a Lagrange multiplier $\lambda \in L^q(\partial\Omega)$ such that:
 \begin{enumerate}[(i)]
    \item \textbf{Global radial condition:} $$\|\lambda\|_q \leq 1$$
    \item \textbf{Exact saturation / Complementary slackness:}
         $$\bm{z}_\lambda \cdot \vnu_{f^*} = \tw(y, \vnu_{f^*}), \quad |Df^*|-a.e.,$$
    \item \textbf{Boundary trace:} $|\bm z_\lambda(x) \cdot \bm n| \leq \beta w_\partial(x),  \quad x \in \partial\Omega.$ 
     \item \textbf{Boundary alignment:} $(\bm{z} \cdot \bm{n})f^* = - \beta w_\partial|f^*| \quad \textnormal{on} \ \partial\Omega.$
 \end{enumerate}
Here, $\bm{z}_\lambda$ is the vector field defined in \eqref{eq:z_formula}.
\end{corollary}

\begin{corollary}
    Let $\tw$ be defined as in \eqref{def:tw} and fix any "location-direction" pair $(y_0, \theta_0).$ Assume $p = q = 2$. Then the dual candidate
    \begin{equation*}
        \lambda_{(y_0,\theta_0)} := \frac{K(\nabla G(\cdot, y_0) \cdot \theta_0)}{\|K(\nabla G(\cdot, y_0) \cdot \theta_0)\|_2},
    \end{equation*}
    satisfies the global radial condition (i) of Corollary \ref{col:simp_cond}. Furthermore, if $\theta_0 = \vnu_{f^*}(y_0)$, it also fulfils exact saturation (ii) at $y_0$.
\end{corollary}
\begin{proof}
    Clearly, $\|\lambda_{(y_0,\theta_0)}\|_2 = 1$. Setting the given choice of $\lambda$ into \eqref{eq:z_formula}, it follows that the complementary slackness condition is satisfied at $y_0$.
\end{proof}

This can be interpreted as follows: The TV atom $\chi_R$ associated with a simple set $R$ (cf. \cite{fleming1957functions, AmbrosioCasellesMasnouMorel2001}) is generated by a continuum of infinitesimal oriented boundary elements $(y_0,\theta_0)$ of its perimeter. Under the weight $\tw$ these elements are treated without spatial bias in the sense that every $(y_0,\theta_0)$ is equally dual-feasible through its own unit-length canonical certificate/Lagrange multiplier $\lambda_{(y_0,\theta_0)}$.


Admittedly, to certify a full region, i.e., to recover $f^* = \chi_{R}$, we would need a saturating dual certificate in a $|Df^*|$-a.e. sense. 
The existence of such a certificate is far from trivial and relies on the domain $\Omega$, the shape $R$ and the forward operator $K$. In the unweighted setting, exact recovery of this type has been established for certain shapes under specific forward operators, both for isotropic TV \cite{DeCastroDuvalPetit2024} and in a  anisotropic TV setting \cite{HollerWirth2024}.

Nevertheless, we will attempt to further highlight why the weights are necessary to certify extended shapes at different depths throughout the domain $\Omega$. To this end, assume the true source to be the characteristic function $\chi_E$, where $E \subset \subset \Omega$ has a smooth boundary $\partial E$.

Note that on characteristic functions the weighted total variation reduces to the weighted perimeter, 
$$\TVtw(\chi_E)=\int_{\partial E} \tw(y,\bm n_E)ds.$$ 
Consequently, if $\chi_E$ solves \eqref{def:basis_pursuit_Lagrange}, it is a stationary point of the Lagrangian
$$\mathcal L(E,\lambda)=\int_{\partial E}\tw(y,\bm n_E)\,ds+\langle\lambda,K\chi_E-d\rangle,$$
where $\lambda$ is an associated Lagrange multiplier. 

To derive the associated Euler–Lagrange equation, consider a smooth normal deformation
\begin{equation*}
    \partial E_t = \{\, y + t\,\varphi(y)\,\bm n_E(y) : y \in \partial E \,\},
\end{equation*}
with $\varphi$ an arbitrary smooth scalar velocity. The stationarity condition of the Lagrangian reads
\begin{equation*}
    \frac{d}{dt}\, \mathcal{L}(E_t,\lambda)\bigg|_{t=0} = 0 \quad \forall\, \varphi .
\end{equation*}
By the theory of shape derivatives \cite{delfour2011shapes}, the first variation of the constraint term is
\begin{equation*}
    \frac{d}{dt}\, \langle \lambda,\, K\chi_{E_t}\rangle \bigg|_{t=0}
    = \int_{\partial E} K^*\lambda \;\varphi \; ds,
\end{equation*}
while the first variation of the weighted perimeter (in which the normal $\bm n_{E_t}$ also varies with $t$) is
\begin{equation*}
    \frac{d}{dt} \int_{\partial E_t} \tw(y, \bm n_{E_t})\, ds \bigg|_{t=0}
    = \int_{\partial E} \kappa_{\tw}\; \varphi \; ds,
\end{equation*}
where $\kappa_{\tw}$ is the weighted anisotropic curvature \cite{dogan2012first}. Stationarity for all $\varphi$ therefore gives
\begin{equation*}
    \int_{\partial E} \big(\kappa_{\tw} + K^*\lambda\big)\,\varphi \; ds = 0
    \quad \forall\, \varphi,
\end{equation*}
and hence
\begin{equation}\label{eq:euler_lagrange_shape}
    \kappa_{\tw} + K^*\lambda = 0 \quad \text{on } \partial E .
\end{equation}
We emphasize that \eqref{eq:euler_lagrange_shape} is a \emph{necessary}
(stationarity) condition that must hold simultaneously at every point of
$\partial E$; it does not by itself guarantee that $\chi_E$ is a global
minimizer.

For simplicity, assume now that $\tw(y,\vnu) = w(y)$, i.e.\ that the weight is isotropic and that $E = B_r(x_0)$, i.e., a ball of radius $r$ centered at $x_0 \in \Omega$. The weighted curvature then simplifies to
\begin{equation*}
    \kappa_{\tw} = \frac{w}{r} + \nabla w \cdot \bm n,
\end{equation*}
so that \eqref{eq:euler_lagrange_shape} becomes
\begin{equation}\label{eq:curvature_cond}
    K^*\lambda = -\frac{w}{r} - \nabla w \cdot \bm n \quad \text{on } \partial B_r(x_0) .
\end{equation}

Let us now discuss the realism of satisfying \eqref{eq:curvature_cond} for different choices of weight functions $w$. To this end, let us define $\rho(y) := \operatorname{dist}(y,\partial\Omega)$ as the distance from $y$ to the boundary, which we refer to as the depth of $y$. 
We also assume, which is the case for many inverse problems, that $K^*\lambda$ decays exponentially into the interior:
\begin{equation*}
    |K^*\lambda(y)| \;\lesssim\; C(\rho(y))\, e^{-\tau \rho(y)},
\end{equation*}
where $C(\cdot)$ has at most polynomial dependence on $\rho$.

Note that \eqref{eq:curvature_cond} must hold at every point of $\partial B_r$ for the same multiplier $\lambda$. At a single point, the equation can always be satisfied; the difficulty is simultaneity along the curve $\partial B_r$. With unweighted TV, i.e., $w \equiv 1$, the right hand side becomes the constant $-\frac{1}{r}$, which cannot match the (assumed) exponentially decay of $K^*\lambda$ for curves spanning over multiple depths. A weight varying only polynomially in $\rho$ also fails for the same reason. Hence, to allow an approximate recovery of interfaces that span a range of depths, it appears that the weight function must have the same exponential dependence on $\rho$ as $K^*\lambda$. This is achieved in the present paper by defining $w$ in terms of images under $K$ of the gradient of appropriate Green's functions. 

\section{Numerical results}\label{sec:numerics}

\subsubsection*{Discretization}

All experiments are carried out on the unit square $\Omega = (0,1) \times (0,1)$, discretized by a uniform grid with spacing $h = 1/128$. The forward operator $K: L^2(\Omega) \to L^2(\partial\Omega)$ is the map $f \mapsto u|_{\partial\Omega}$, where $u$ solves the screened Poisson equation \eqref{eq:screenedPoisson}. The PDE is discretized with standard piecewise-linear finite elements on the triangulation associated with the grid, and the resulting discretized forward operator is denoted by $\mathsf{K} = \mathsf{T}\mathsf{A}^{-1} \in \mathbb{R}^{M \times N}$, where $\mathsf{A}$ and $\mathsf{T}$ are the matrices associated with the screened Poisson equation and the trace operator, respectively. 

To describe how the weights are generated, let $\mathsf K = \mathsf U\mathsf S\mathsf V^\top$ be the (truncated) singular value decomposition of the forward matrix, with singular values $s_1,\dots,s_r$ and right singular vectors $\bm v_1,\dots,\bm v_r\in\mathbb R^N$.
Furthermore, we denote by $\mathsf{L}$ the discretization of the Laplace operator $-\Delta$, used for the representation of the Green's functions, cf. \eqref{label:source_rep_dirichlet}.  
The directional weight at a point $y$ is the magnitude of the dipole sensitivity,
\begin{equation}
\label{eq:sensitivity}
   s_{\bm\nu}(y) = \bigl\| \mathsf K\,\nabla_{\bm\nu} G(\cdot\,;y) \bigr\|_2,
   \qquad G(\cdot\,;y) = \mathsf L^{-1}\delta_y ,
\end{equation}
i.e., the discrete counterpart of $\tw(y,\bm\nu)
= \|K(\nabla G(\cdot\,;y)\cdot\bm\nu)\|_{L^2(\partial \Omega)}$.

Evaluating \eqref{eq:sensitivity} directly would require one dipole solve per cell, i.e.\ $N$ solves.  Index the grid cells by $e$, with $y_e$ the corresponding point and $\delta_e := \delta_{y_e}$ the discrete source at that cell. The directional derivative $\nabla_{\bm\nu}$ acts on the source location $y$, while $\mathsf{L}^{-1}$ acts on the field variable; thus linearity of $\mathsf{L}^{-1}$ gives $\nabla_{\bm\nu}\mathsf{L}^{-1}\delta_e = \mathsf{L}^{-1}\nabla_{\bm\nu}\delta_e$. Using in addition the symmetry $\mathsf L = \mathsf L^\top$ to write $\bm v_k^\top\mathsf{L}^{-1} = \bm z_k^\top$ with $\bm z_k = \mathsf{L}^{-1}\bm v_k$ gives
\begin{equation*}
    \bm v_k^\top \mathsf{L}^{-1}\nabla_{\bm\nu}\delta_e = (\mathsf{L}^{-1}\bm v_k)^\top\nabla_{\bm\nu}\delta_e = \bm z_k^\top \nabla_{\bm\nu}\delta_e = (\nabla_{\bm\nu}^\top\bm z_k)_e = -(\nabla_{\bm\nu}\bm z_k)_e,
\end{equation*}
where the last equality follows from integration by parts. Expanding \eqref{eq:sensitivity} in the singular basis of $\mathsf{K}$ and using the expression above yields
\begin{equation}
\label{eq:reciprocity}
   \bigl\| \mathsf K\,\nabla_{\bm\nu} G(\cdot\,;y_e)\bigr\|_2^2
   = \sum_{k=1}^{r} s_k^2\,
     \bigl(\bm v_k^\top \mathsf L^{-1}\,\nabla_{\bm\nu}\delta_e\bigr)^2
   = \sum_{k=1}^{r} s_k^2\,\bigl(\nabla_{\bm\nu}\,\bm z_k\bigr)_e^2.
\end{equation}
Thus pushing each detectable mode/singular vector once through $\mathsf L^{-1}$ and reading off its discrete gradient recovers the sensitivity at \emph{every} cell and in \emph{every} direction simultaneously, replacing $N$ per-cell solves by $r$ mode solves ($r\ll N$).

Stacking the gradients of the smoothed modes column-wise gives
\[
   {\mathsf W}_x = \mathsf D_x\,[\,\bm z_1\;\cdots\;\bm z_r\,]\,\mathsf S_r,
   \qquad
   {\mathsf W}_y = \mathsf D_y\,[\,\bm z_1\;\cdots\;\bm z_r\,]\,\mathsf S_r,
   \qquad \mathsf S_r = \operatorname{diag}(s_1,\dots,s_r),
\]
with $\mathsf D_x, \mathsf D_y$ the discrete partial-derivative operators.
For each evaluation cell $e$ (paired with one $x$- and one $y$-edge) the corresponding rows form the local weight
\[
   {\mathsf W}^{(e)} =
   \begin{bmatrix} ({\mathsf W}_x)_{e,:} \\[2pt] ({\mathsf W}_y)_{e,:} \end{bmatrix}^{\!\top}
   \in \mathbb R^{r\times 2},
\]
whose $2\times 2$ Gram matrix is precisely the metric whose quadratic form in $\bm\nu$ equals $s_{\bm\nu}(y_e)^2$ in \eqref{eq:sensitivity}.

With $(\nabla_h \bm f)_e = \bigl((\mathsf D_x \bm f)_e,(\mathsf D_y \bm f)_e\bigr)^\top$, the discrete directionally weighted total variation ($p=2$) is the group ($\ell_{2,1}$) norm of the weighted gradient,
\begin{equation*}
   \mathrm{TV}_{\tw}(\bm f)
   = \sum_{e=1}^{N_{\mathrm{eval}}} \bigl\| {\mathsf W}^{(e)}(\nabla_h \bm f)_e \bigr\|_2 .
\end{equation*}

From this, it is also straightforward to compute the discretized isotropic weights corresponding to \eqref{def:w} by setting
\begin{equation*}
\mathrm{TV}_{w}(\bm f)
  = \sum_{e=1}^{N_{\mathrm{eval}}} \|\mathsf W^{(e)}\|_2 \,\bigl\|(\nabla_h \bm f)_e\bigr\|_2 .
\end{equation*}

Equation \eqref{eq:reciprocity} shows that each mode's contribution $(\nabla_{\bm\nu}\bm z_k)_e^2$ involves the factor $s_k^2$. In our experiments, removing this factor - i.e. setting $s_k = 1$ for all $k$ in \eqref{eq:reciprocity} - improves recovery in practise, and all numerical results below use this choice. For a more in-depth discussion of employing such a filter for weight generation, see \cite{ENS25}. Note that our analysis is presented in terms of an abstract forward operator $K$. Hence, such a filter can also be incorporated in the infinite dimensional setting by considering it to be part of the action of $K$. 

\subsubsection*{Test problems}
 
We consider two test problems, which differ in the geometry of the unknown. The first is an \emph{inverse source problem}: the goal is to recover a localized source $f^\dagger$ supported in the interior of $\Omega$ from boundary measurements of the corresponding solution $u$. The ground truth has compact support strictly inside $\Omega$, so feature creation on $\partial \Omega$ is undesirable. The second is a \emph{gravimetry-inspired problem}: the unknown $f^\dagger$ models a piecewise constant density and the discontinuity of interest is the interface between two materials of different density, which need not stay away from the boundary.
 
This distinction directly informs the choice of Green's function and, with it, the form of the regularizer. As discussed in Section~\ref{sec:preliminaries}, the Dirichlet representation~\eqref{label:source_rep_dirichlet} naturally produces a boundary term with weight $w_\partial$ and is therefore the appropriate setting whenever the regularization should discourage features on $\partial\Omega$, i.e., $\beta > 0$. The Neumann representation~\eqref{eq:source_representation}, by contrast, contains no boundary term, and is the appropriate setting when boundary features are admissible, i.e., $\beta = 0$. We therefore use the Dirichlet-based weights with $\beta > 0$ for the inverse source problem and the Neumann-based weights with $\beta = 0$ for the gravimetry problem.

If not explicitly stated otherwise, we employed $p=2$ in the definition of the weight functions $w$ and $\tw$, cf. \eqref{def:w} and \eqref{def:tw}. 

\subsection*{Inverse Source Problem}

The first simulations concern the recovery of spherical shapes of different radius. We observe in Figure \ref{fig:circles} that a source with small radius is quite well recovered when located close to the boundary, and that relatively larger sources are also recoverable deeper into the domain. Let us mention that when we ran experiments with deep sources with small radii, the spatial extent of the reconstructed source was too large and the magnitude too small.

These observations can be linked to assumptions \eqref{eq:assumption1a} and \eqref{eq:assumption1b}. Admittedly, it is unrealistic that \eqref{eq:assumption1a} or \eqref{eq:assumption1b} is perfectly satisfied in a practical setting involving elliptic PDEs, but they can nevertheless shed some light on recoverability: Panels (a) and (b) of Figure \ref{fig:traces} show $(K\!\left[\nabla G(\cdot;x)\cdot \vnu^*(x)\right])(z_R)$ and $(K\!\left[\nabla G(\cdot;x)\cdot \vnu^*(x) \right])(z_L)$ as functions of the polar angle of $x$ located at a small circle (a) and a large circle (b). Here, $z_R$ (red) and $z_L$ (blue) are boundary points. In neither of the cases, $(K\!\left[\nabla G(\cdot;x)\cdot \vnu^*(x)\right])(z_R) \geq 0$ for all $x$ at the circle or $(K\!\left[\nabla G(\cdot;x)\cdot \vnu^*(x) \right])(z_L) \geq 0$ for all $x$ at the circle, but for the large source this is closer to being satisfied.

\begin{figure}[H]
    \centering
    \begin{subfigure}[t]{0.3\linewidth}        
        \centering
        \includegraphics[trim=0 0 0 19, clip, width=\linewidth]{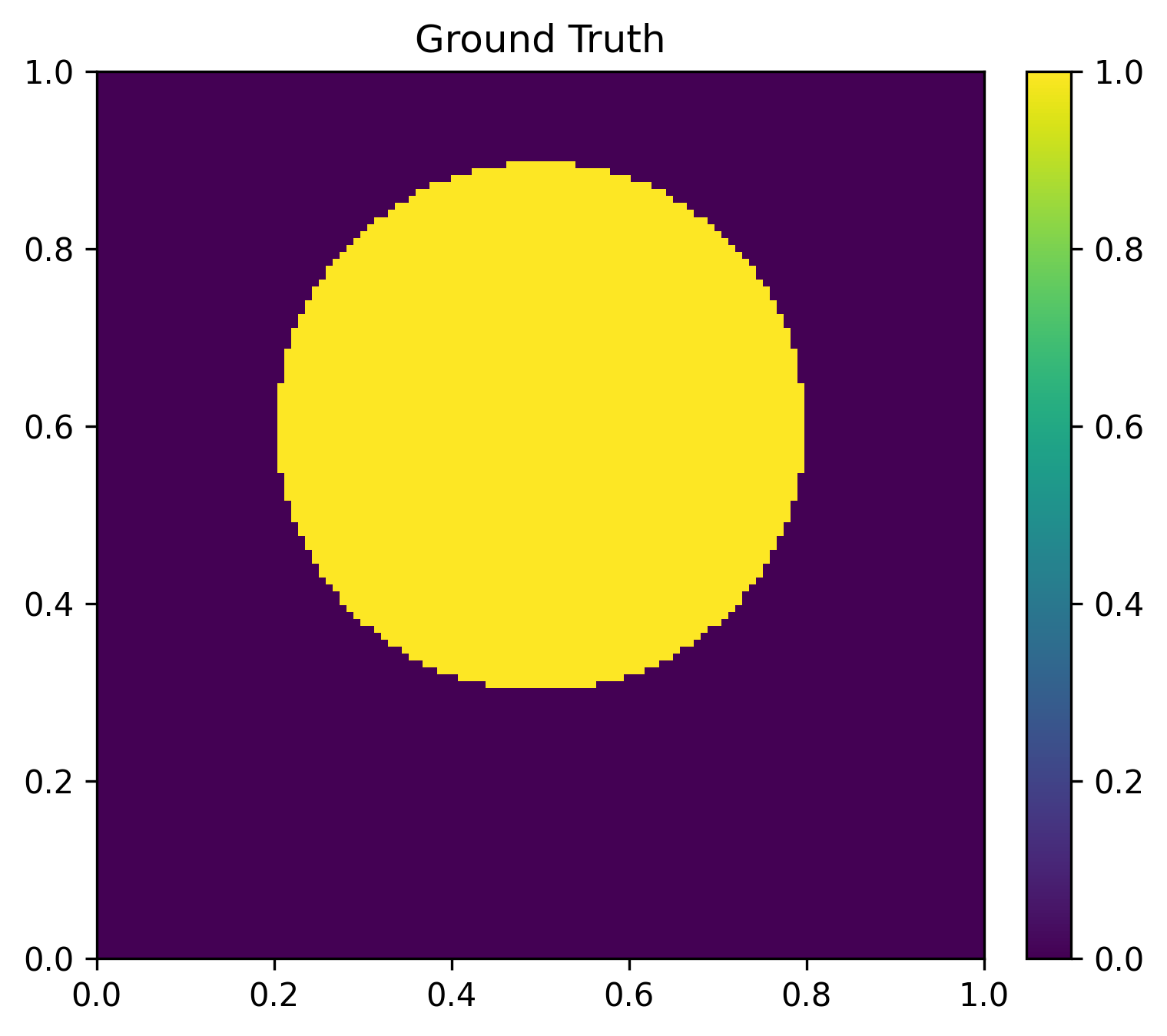}
        \caption{True source}
    \end{subfigure}
    \begin{subfigure}[t]{0.3\linewidth}        
        \centering
        \includegraphics[trim=0 0 0 19, clip, width=\linewidth]{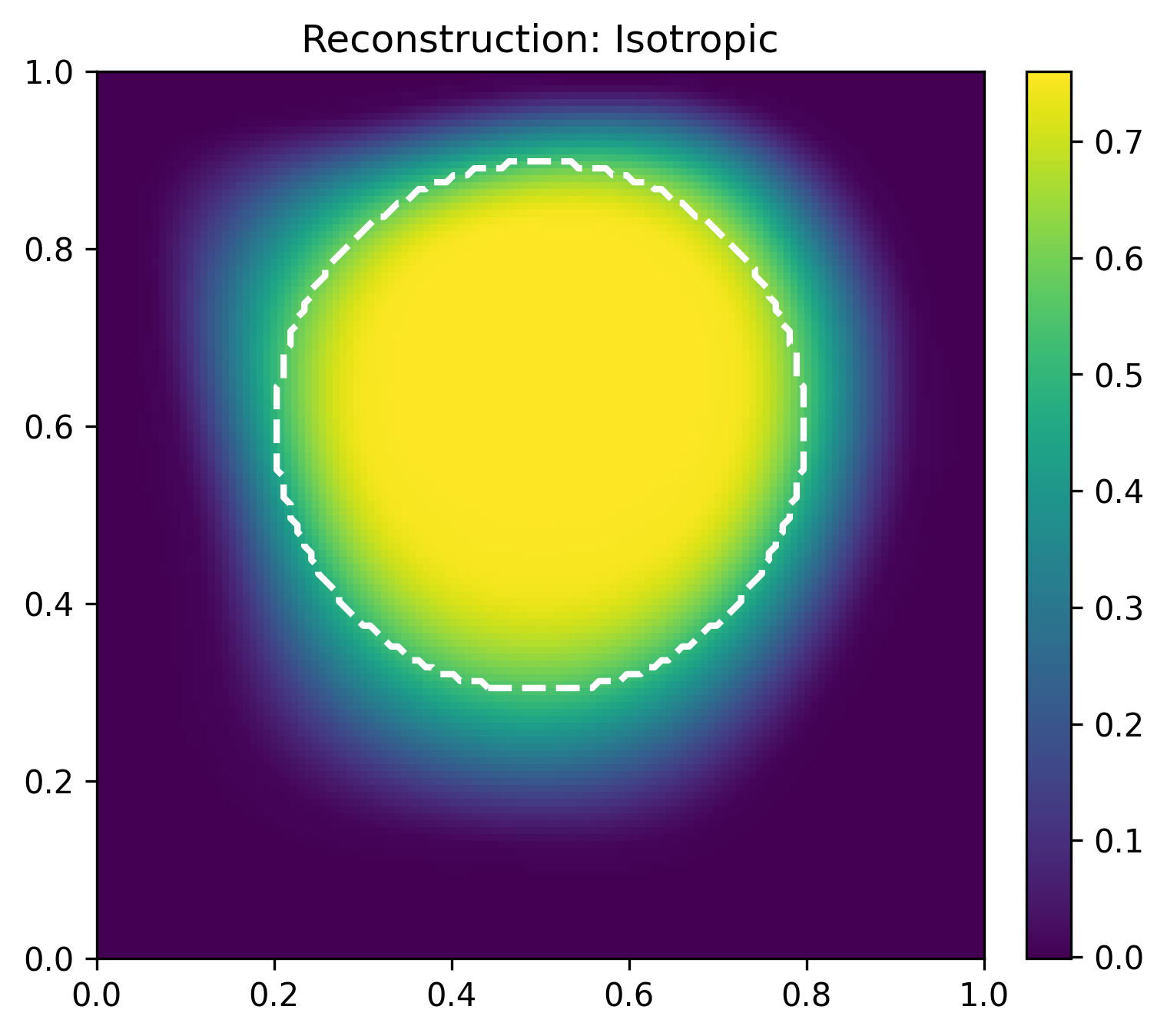}
        \caption{Recovery using $\TVw$}
    \end{subfigure}
    \begin{subfigure}[t]{0.3\linewidth}        
        \centering
        \includegraphics[trim=0 0 0 19, clip, width=\linewidth]{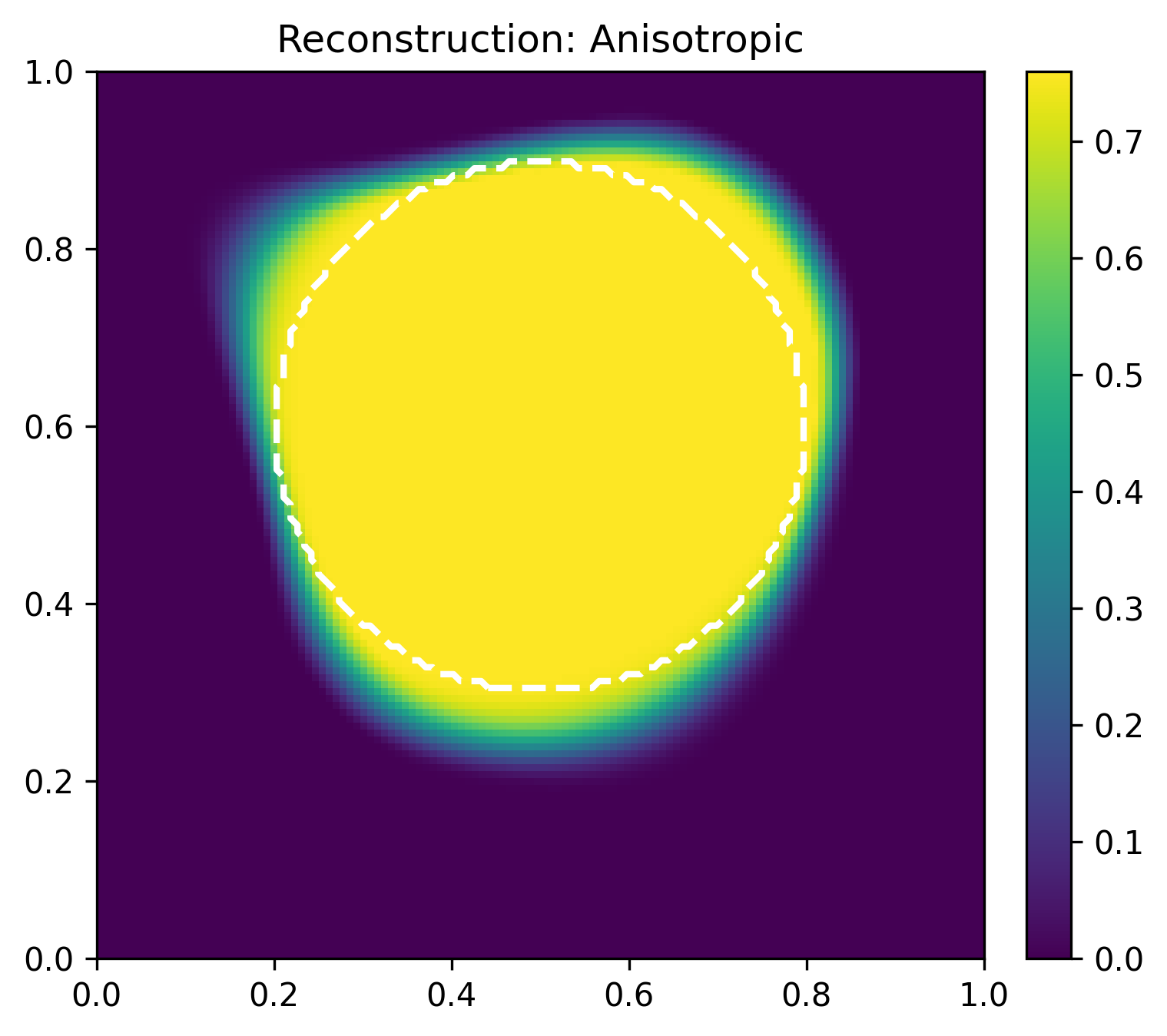}
        \caption{Recovery using $\TVtw$}
    \end{subfigure}
    \begin{subfigure}[t]{0.3\linewidth}        
        \centering
        \includegraphics[trim=0 0 0 19, clip, width=\linewidth]{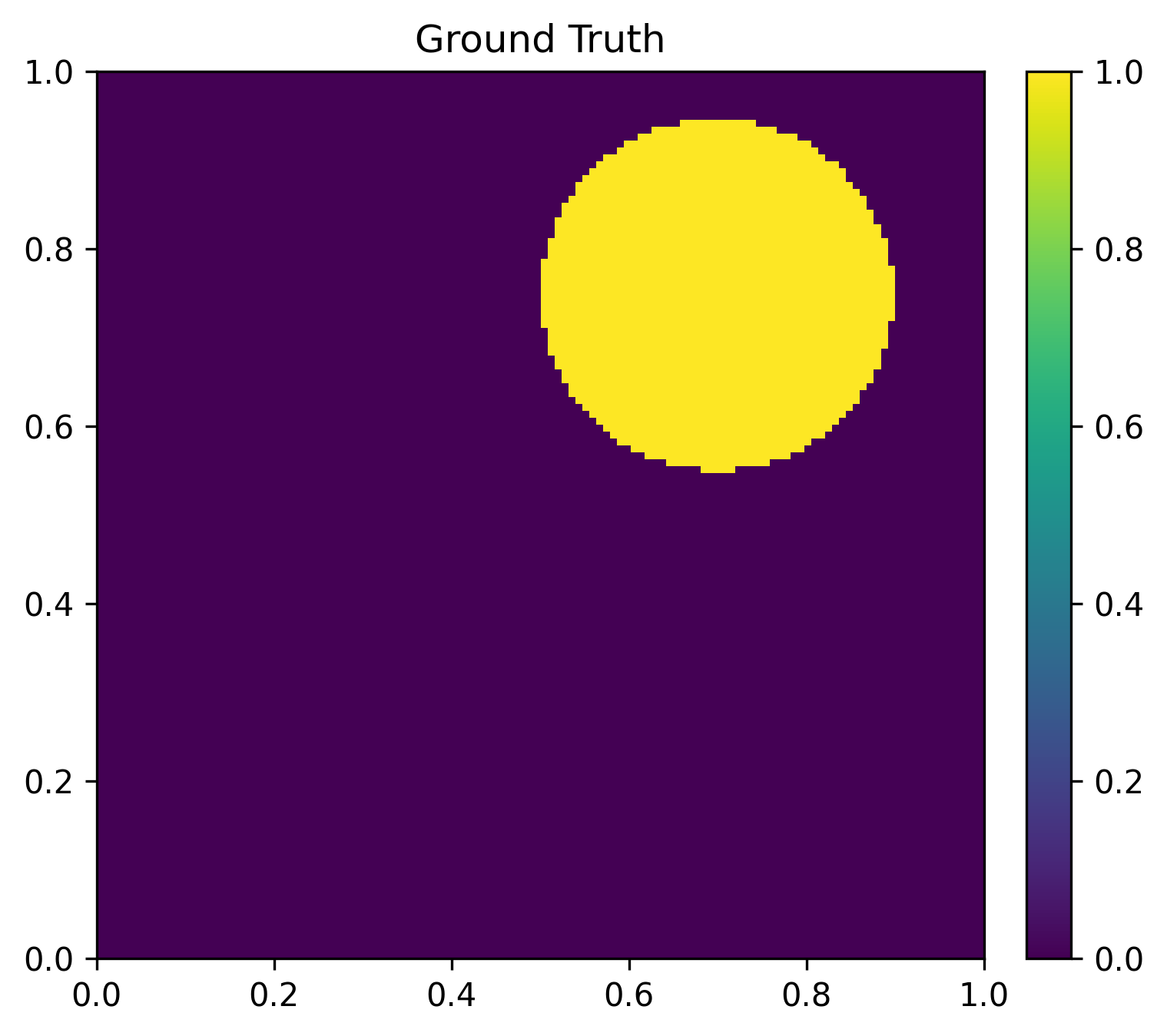}
        \caption{True source}
    \end{subfigure}
        \begin{subfigure}[t]{0.3\linewidth}        
        \centering
        \includegraphics[trim=0 0 0 19, clip, width=\linewidth]{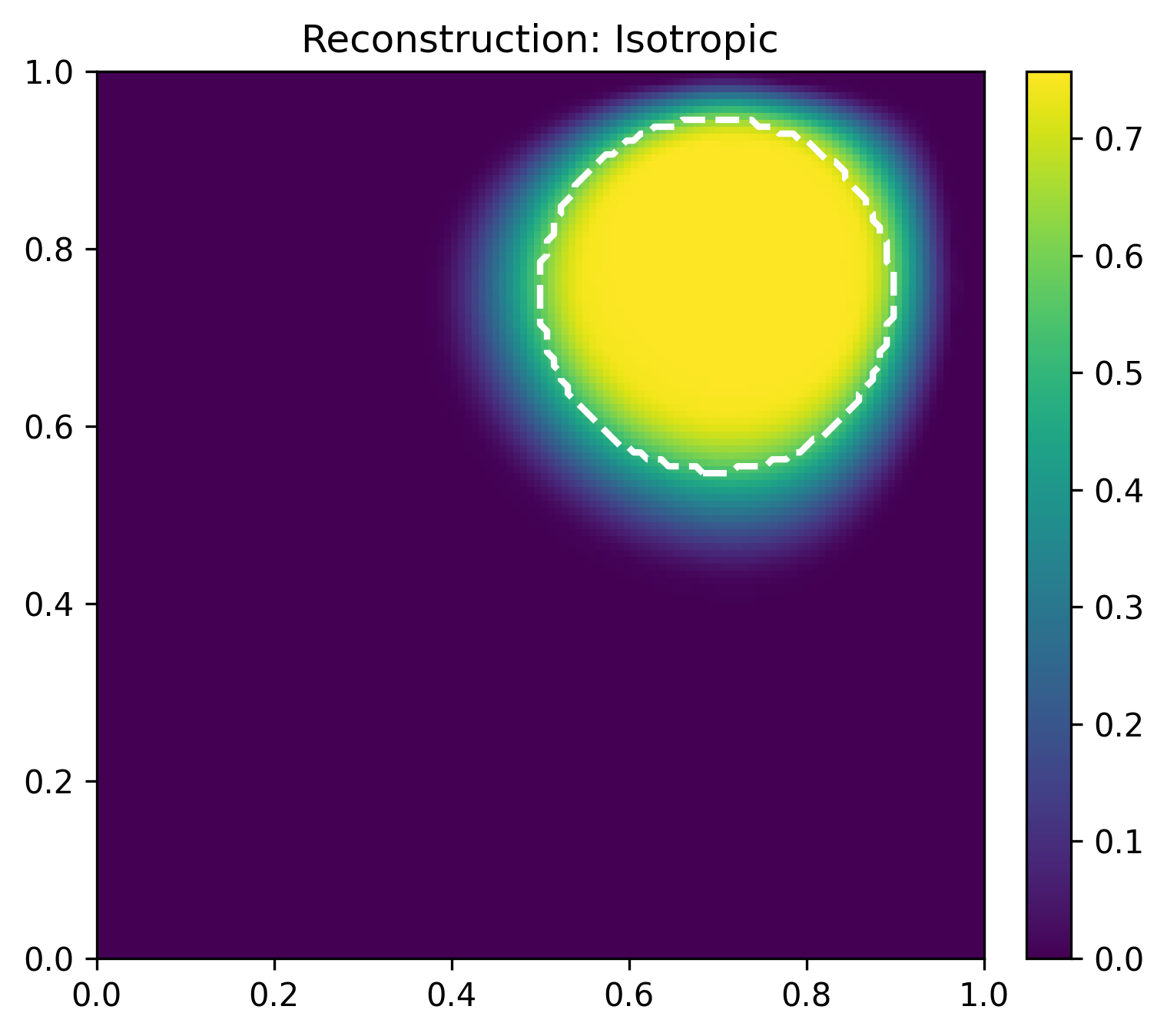}
        \caption{Recovery using $\TVw$}
    \end{subfigure}
    \begin{subfigure}[t]{0.3\linewidth}        
        \centering
        \includegraphics[trim=0 0 0 19, clip, width=\linewidth]{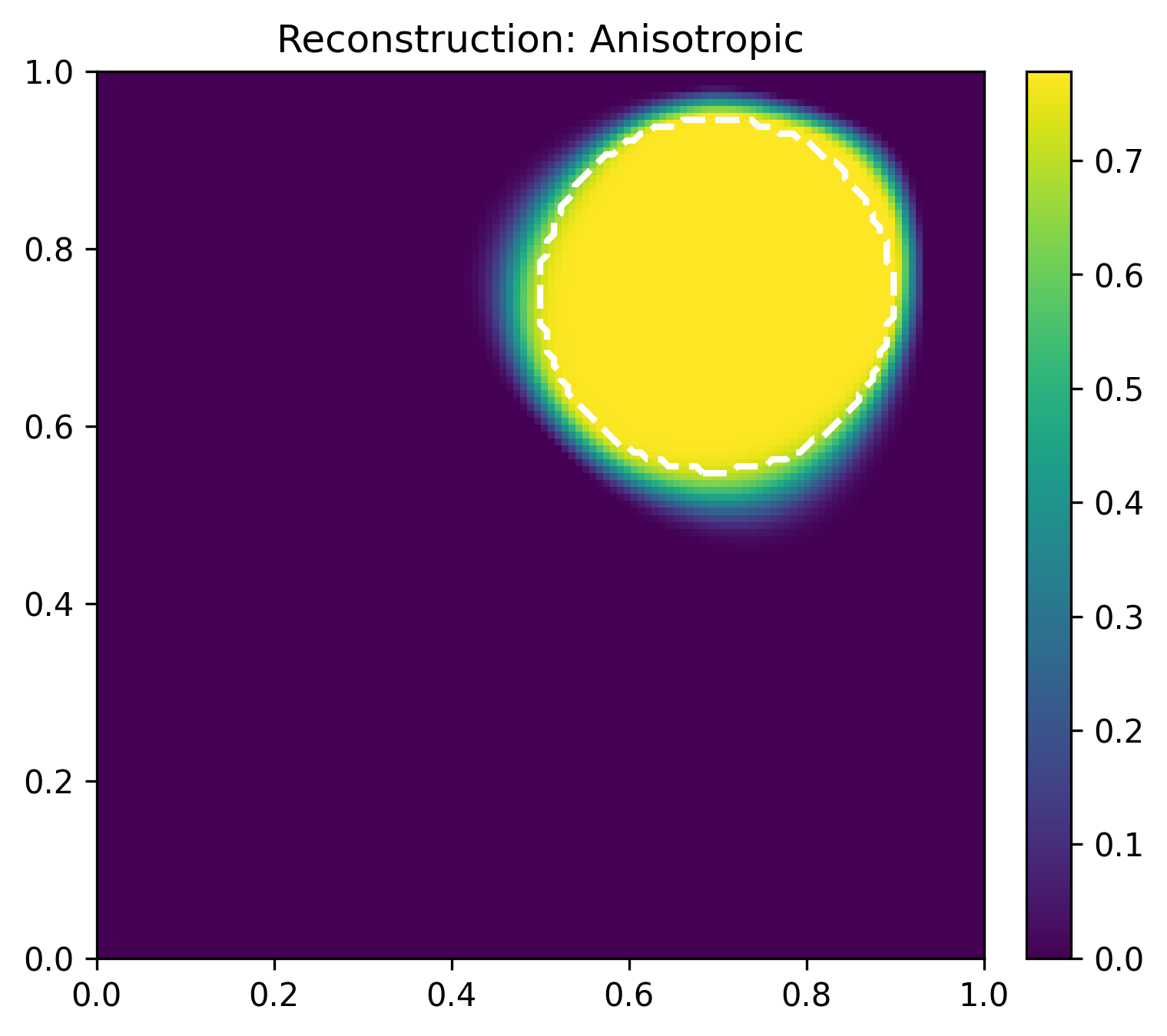}
        \caption{Recovery using $\TVtw$}
    \end{subfigure}
    \begin{subfigure}[t]{0.3\linewidth}        
        \centering
        \includegraphics[trim=0 0 0 19, clip, width=\linewidth]{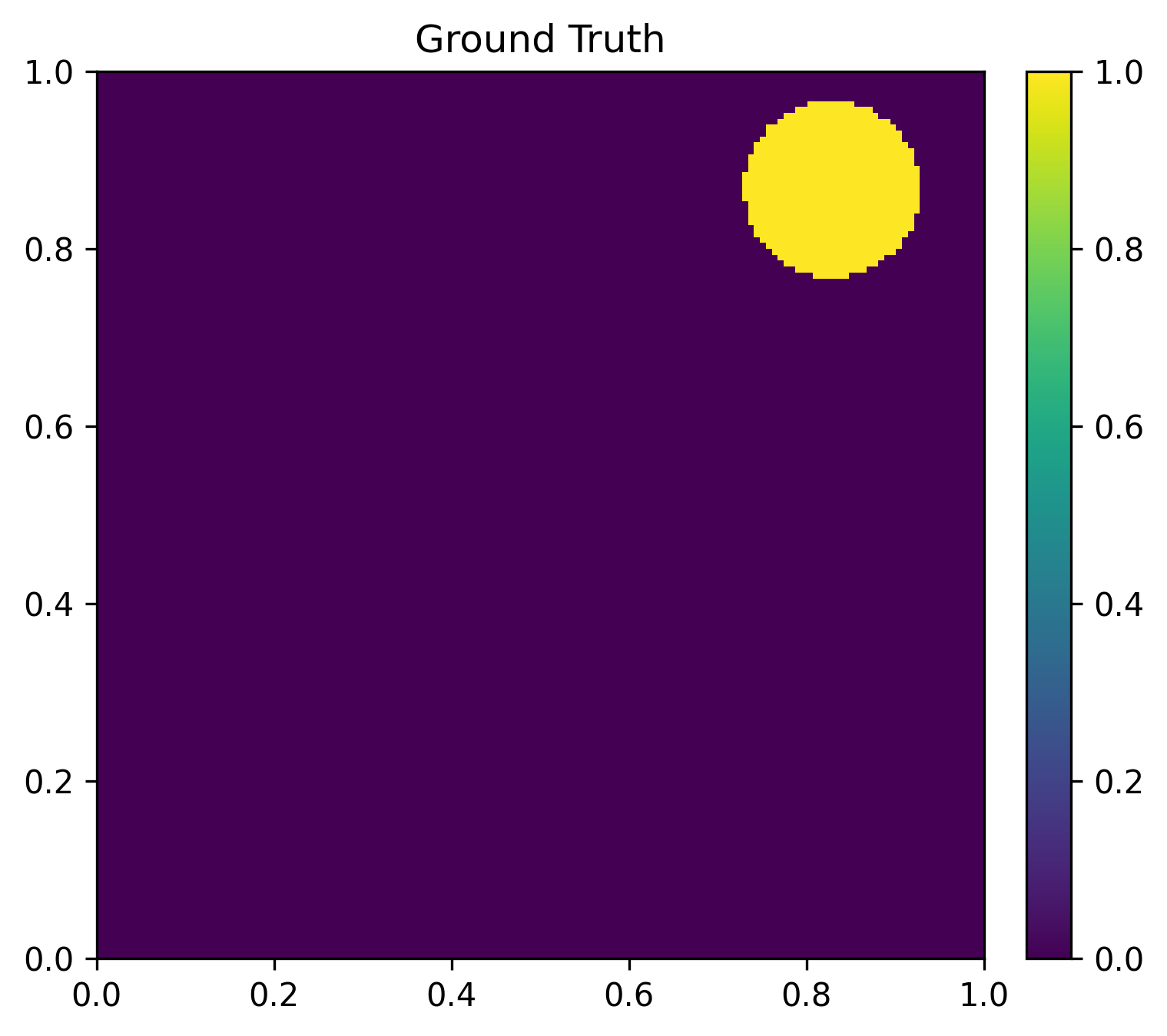}
        \caption{True source}
    \end{subfigure}
    \begin{subfigure}[t]{0.3\linewidth}        
        \centering
        \includegraphics[trim=0 0 0 19, clip, width=\linewidth]{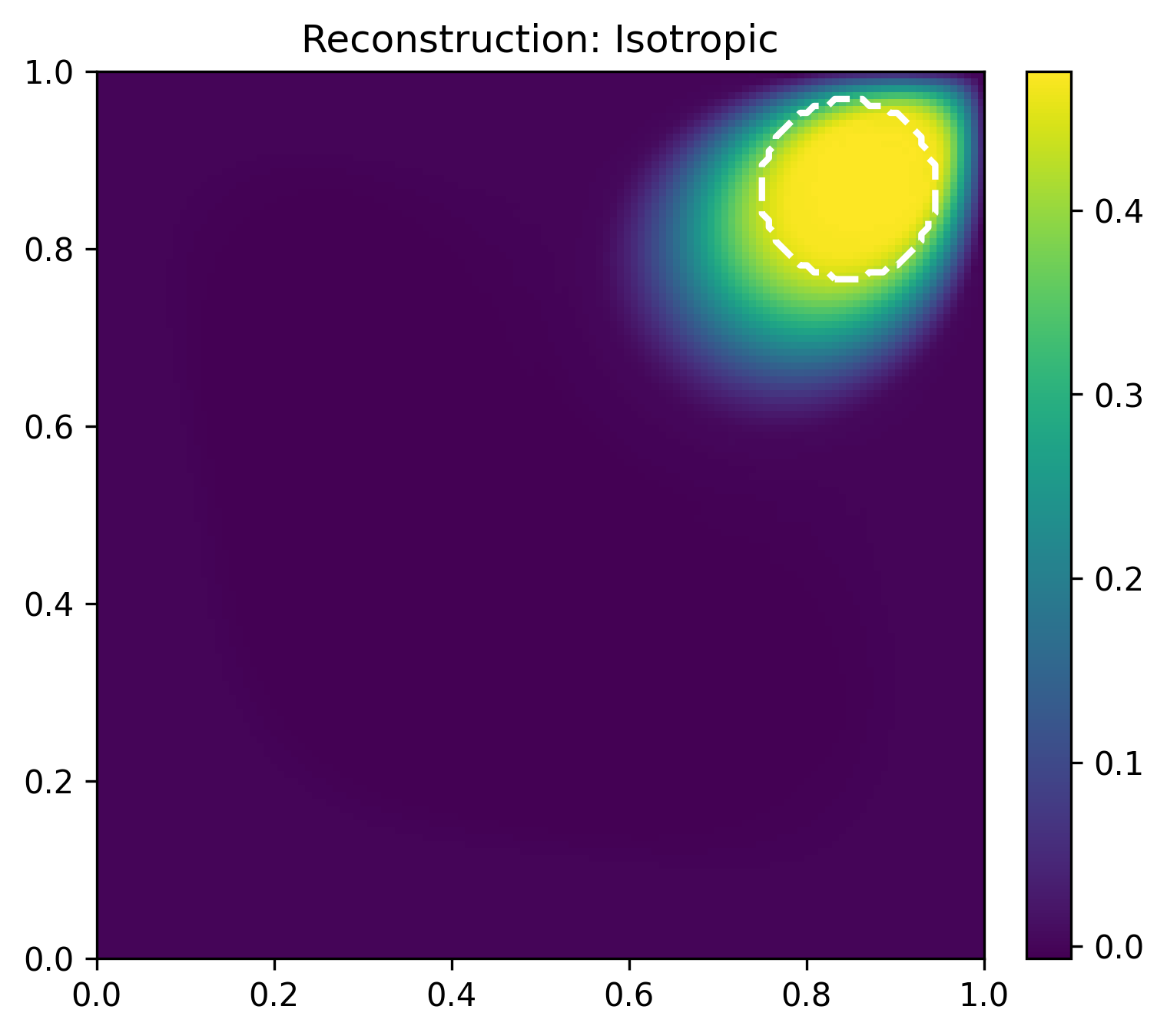}
        \caption{Recovery using $\TVw$}
    \end{subfigure}
    \begin{subfigure}[t]{0.3\linewidth}        
    \centering
    \includegraphics[trim=0 0 0 19, clip, width=\linewidth]{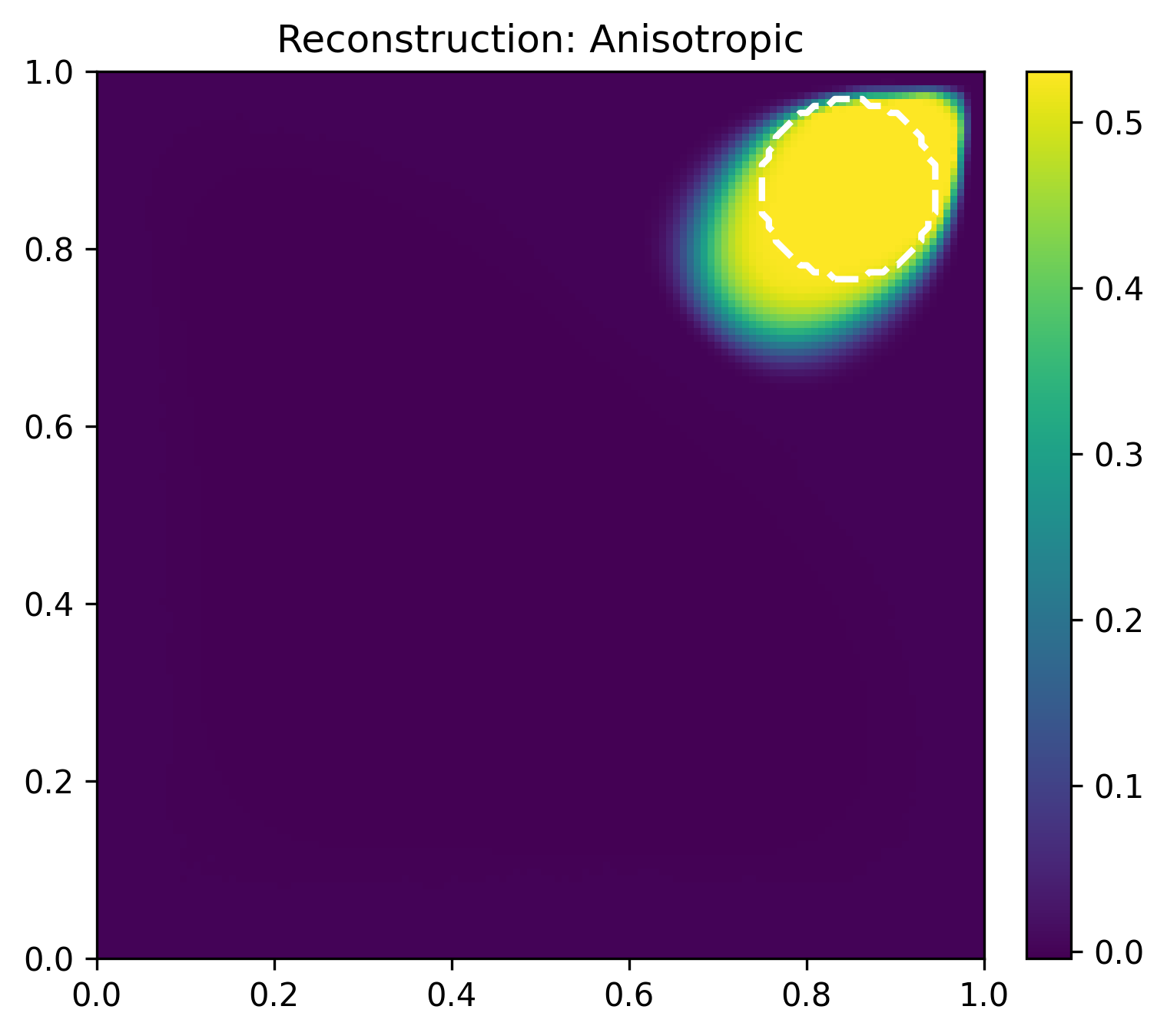}
    \caption{Recovery using $\TVtw$}
    \end{subfigure}
    \caption{Comparison of the true source and inverse recoveries applying the weighted TV ($\TVw$) and directionally weighted TV ($\TVtw$) methods, where the weights are generated with Dirichlet boundary conditions. In all simulations, we set the boundary penalty term $\beta$ equal to the TV-regularization parameter $\alpha$.}
    \label{fig:circles}
\end{figure}

\begin{figure}[H]
    \centering
    \begin{subfigure}[t]{0.8\linewidth}        
        \centering
        \includegraphics[trim=0 0 0 0, clip, width=\linewidth]{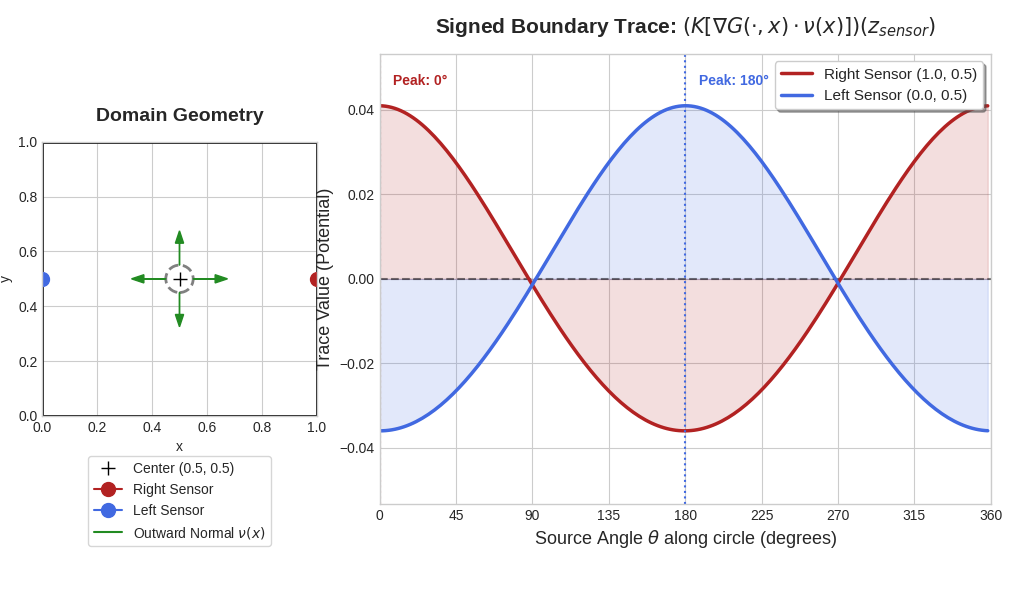}
        \caption{Small circle (dashed curve)}
    \end{subfigure}
    \begin{subfigure}[t]{0.8\linewidth}        
        \centering
        \includegraphics[trim=0 0 0 0, clip, width=\linewidth]{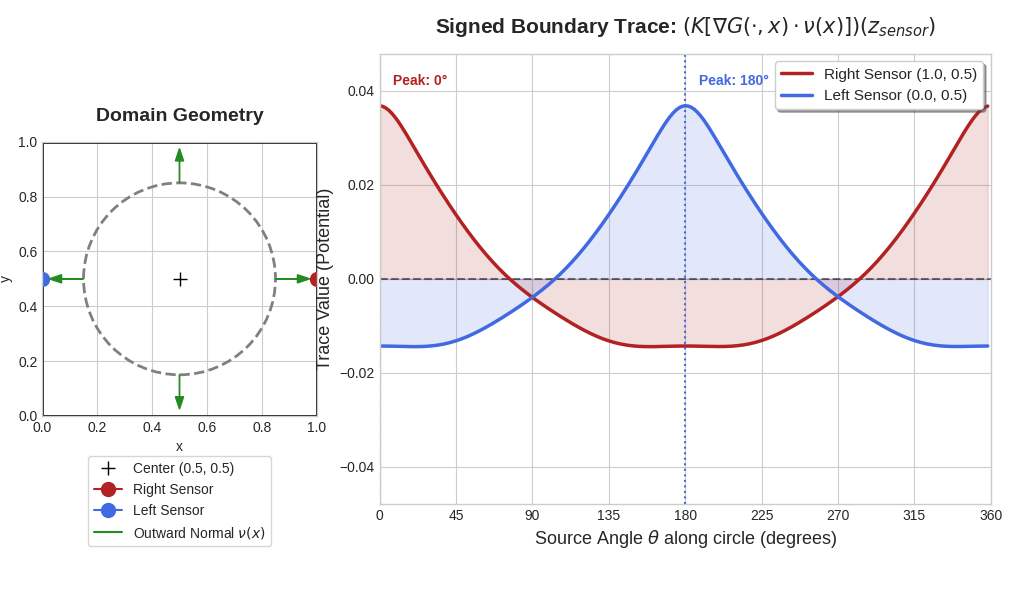}
        \caption{Big circle (dashed curve)}
    \end{subfigure}
    \caption{Point evaluations of $K[\nabla G(\cdot;x) \cdot \vnu^*(x)]$ at two fixed boundary points $z_R, z_L \in \partial\Omega$ (in blue and red, respectively), as functions of the polar angle of the source point $x$ on the source circle.}
    \label{fig:traces}
\end{figure}

In Figure \ref{fig:shapes}, we investigate how well different shapes can be recovered by solving \eqref{def:variational_form} or \eqref{def:variational_form_pre}. For the true convex shapes in panels (a) and (d), we observe that the \methodTwo method (seen in panels (c) and (f)) is successful in reconstructing the shape and position, whereas the isotropic functional (seen in panels (b) and (e)) struggles to preserve the sharp boundary. 

For the non-convex shapes in panels (g) and (j), we observe that both the isotropic and \methodTwo procedures can recover the position and partially the spatial extent of the sources, but they both fail to reconstruct the exact shapes - and the isotropic method still produces smoother reconstructions.

We also investigated how choosing of $p = 1$, instead of $p = 2$, in the definition \eqref{def:tw} of  $\tw$ affected recovery. In Figure \ref{fig:l1norm} we present the recovery of the same sources as in Figure \ref{fig:shapes}. The position and extension of the sources are still quite well preserved, but we might observe a bit more axis-alignment in the inverse solutions.  

\begin{figure}[H]
    \centering
    \begin{subfigure}[t]{0.3\linewidth}        
    \centering
    \includegraphics[trim=0 0 0 19, clip, width=\linewidth]{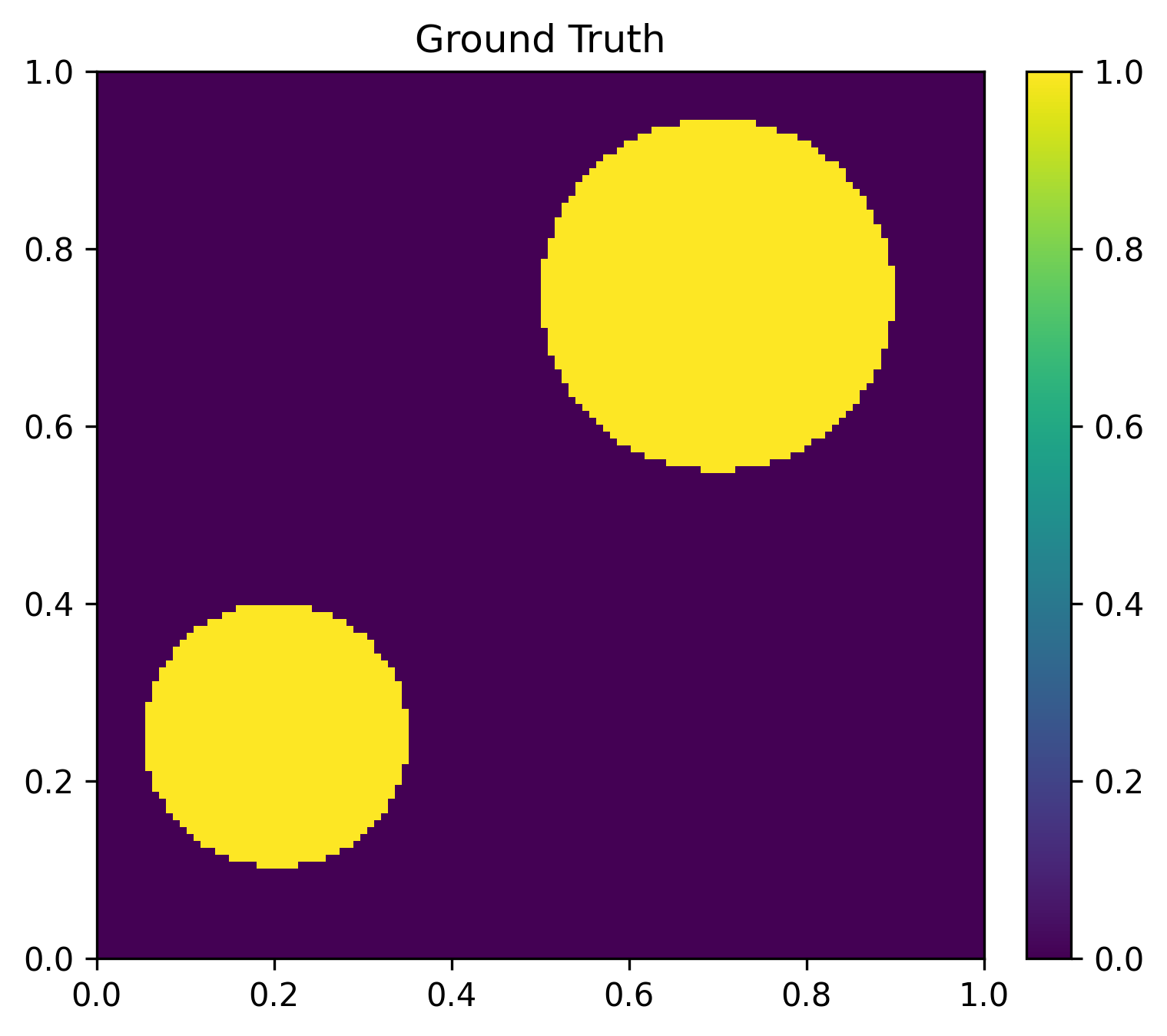}
    \caption{True source}
    \end{subfigure}
    \begin{subfigure}[t]{0.3\linewidth}        
        \centering
        \includegraphics[trim=0 0 0 19, clip, width=\linewidth]{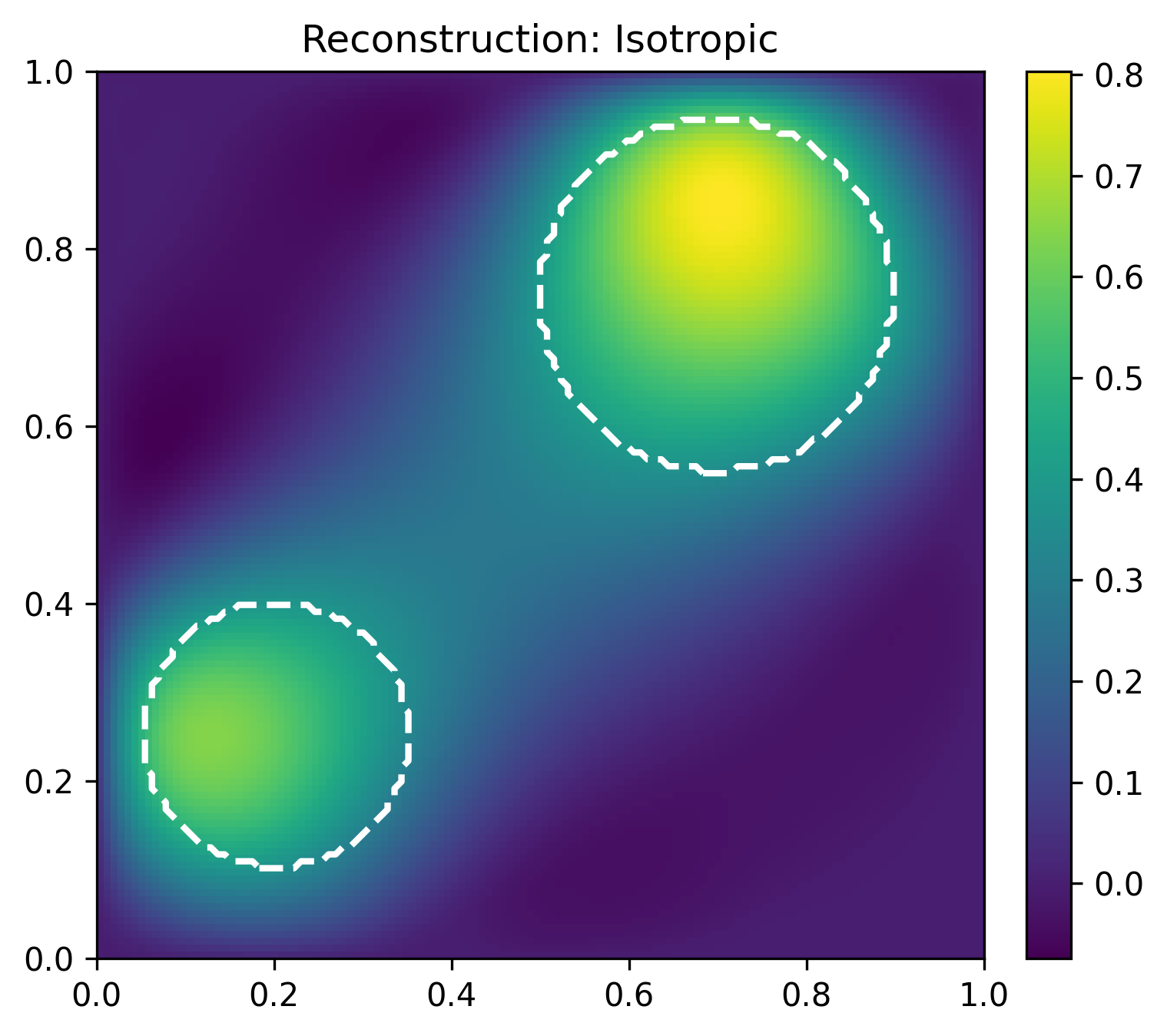}
        \caption{Recovery using $\TVw$}
    \end{subfigure}
    \begin{subfigure}[t]{0.3\linewidth}        
        \centering
        \includegraphics[trim=0 0 0 19, clip, width=\linewidth]{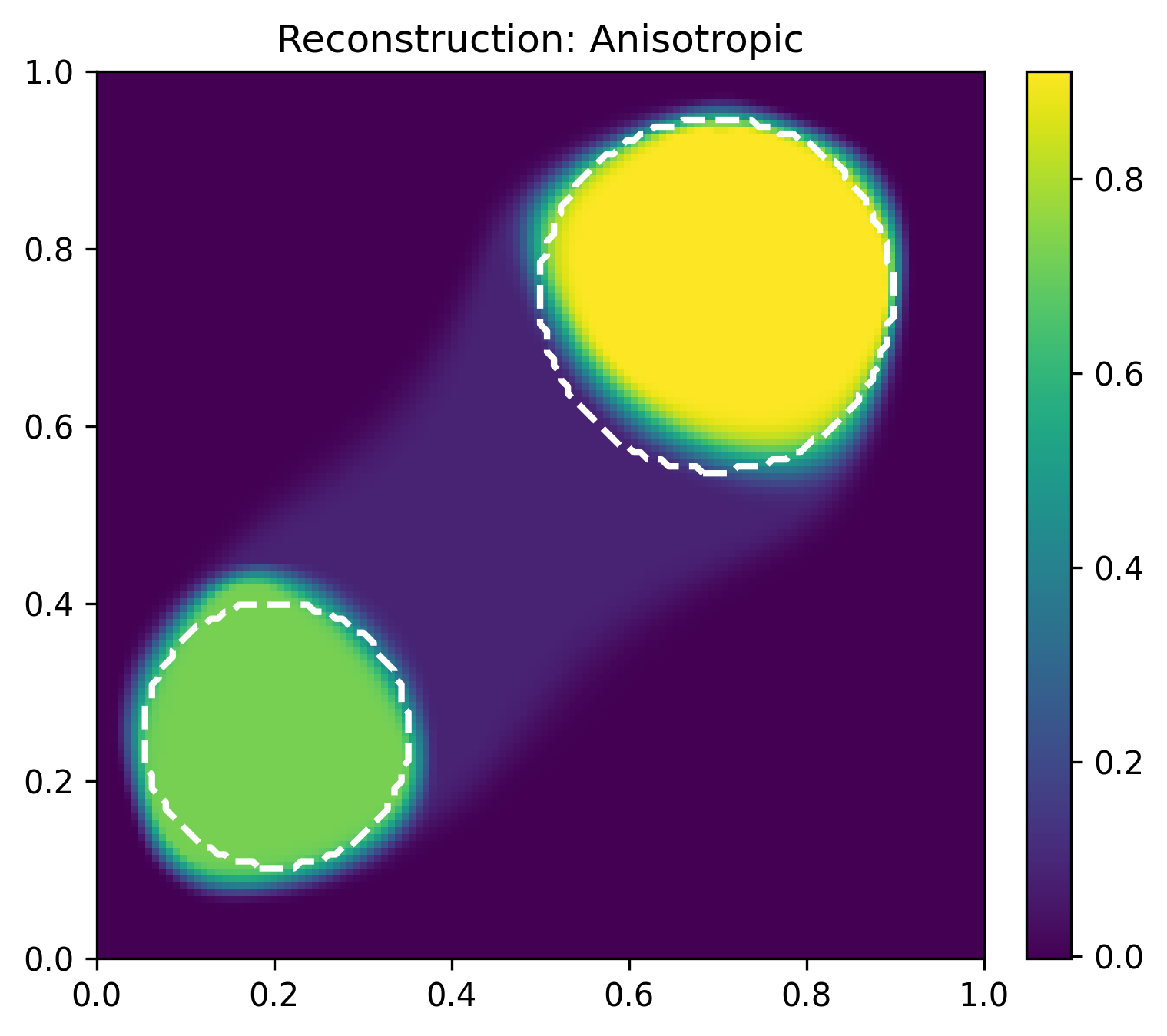}
        \caption{Recovery using $\TVtw$}
    \end{subfigure}
    \begin{subfigure}[t]{0.3\linewidth}        
        \centering
        \includegraphics[trim=0 0 0 19, clip, width=\linewidth]{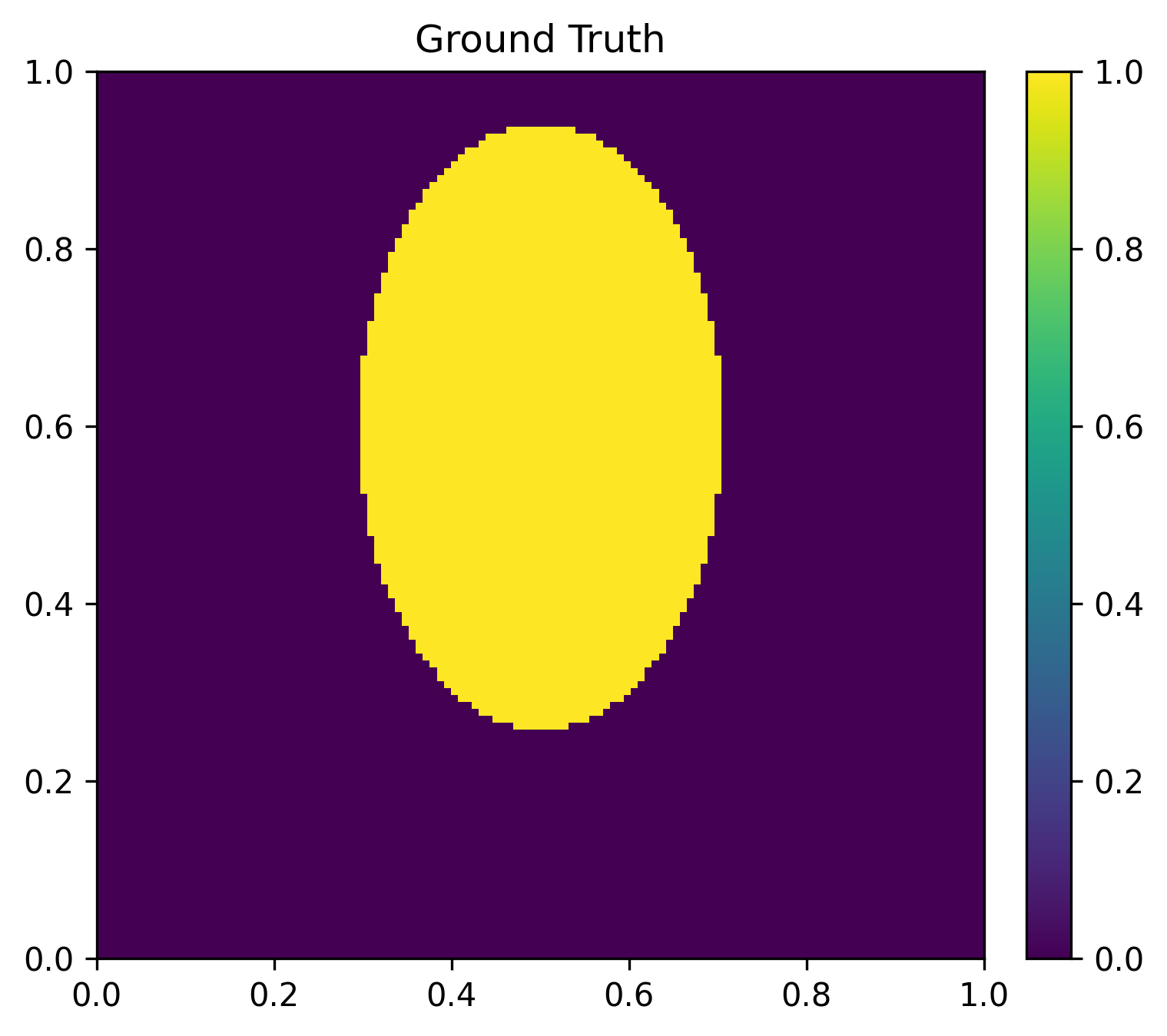}
        \caption{True source}
    \end{subfigure}
    \begin{subfigure}[t]{0.3\linewidth}        
        \centering
        \includegraphics[trim=0 0 0 19, clip, width=\linewidth]{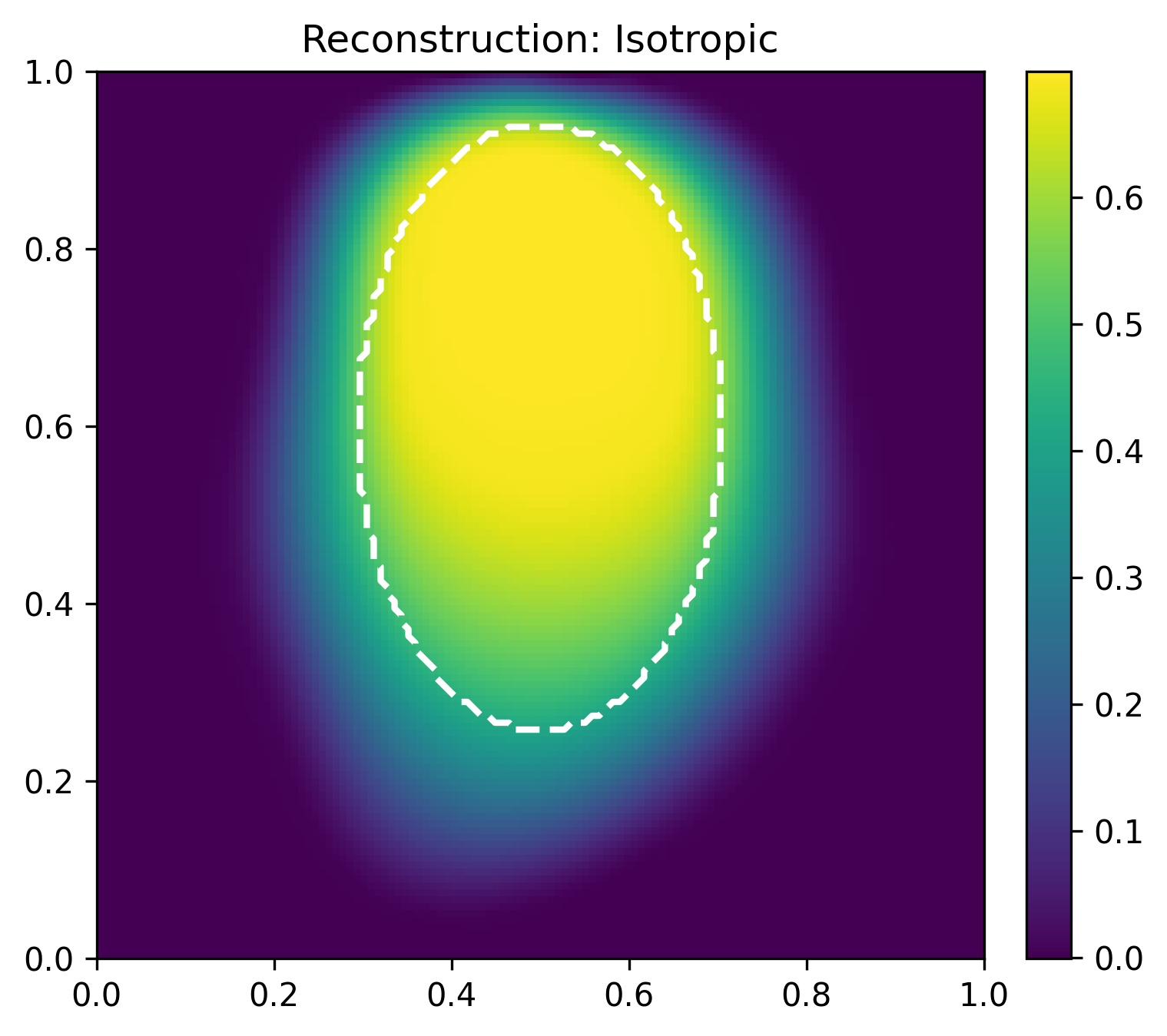}
        \caption{Recovery using $\TVw$}
    \end{subfigure}
    \begin{subfigure}[t]{0.3\linewidth}        
        \centering
        \includegraphics[trim=0 0 0 19, clip, width=\linewidth]{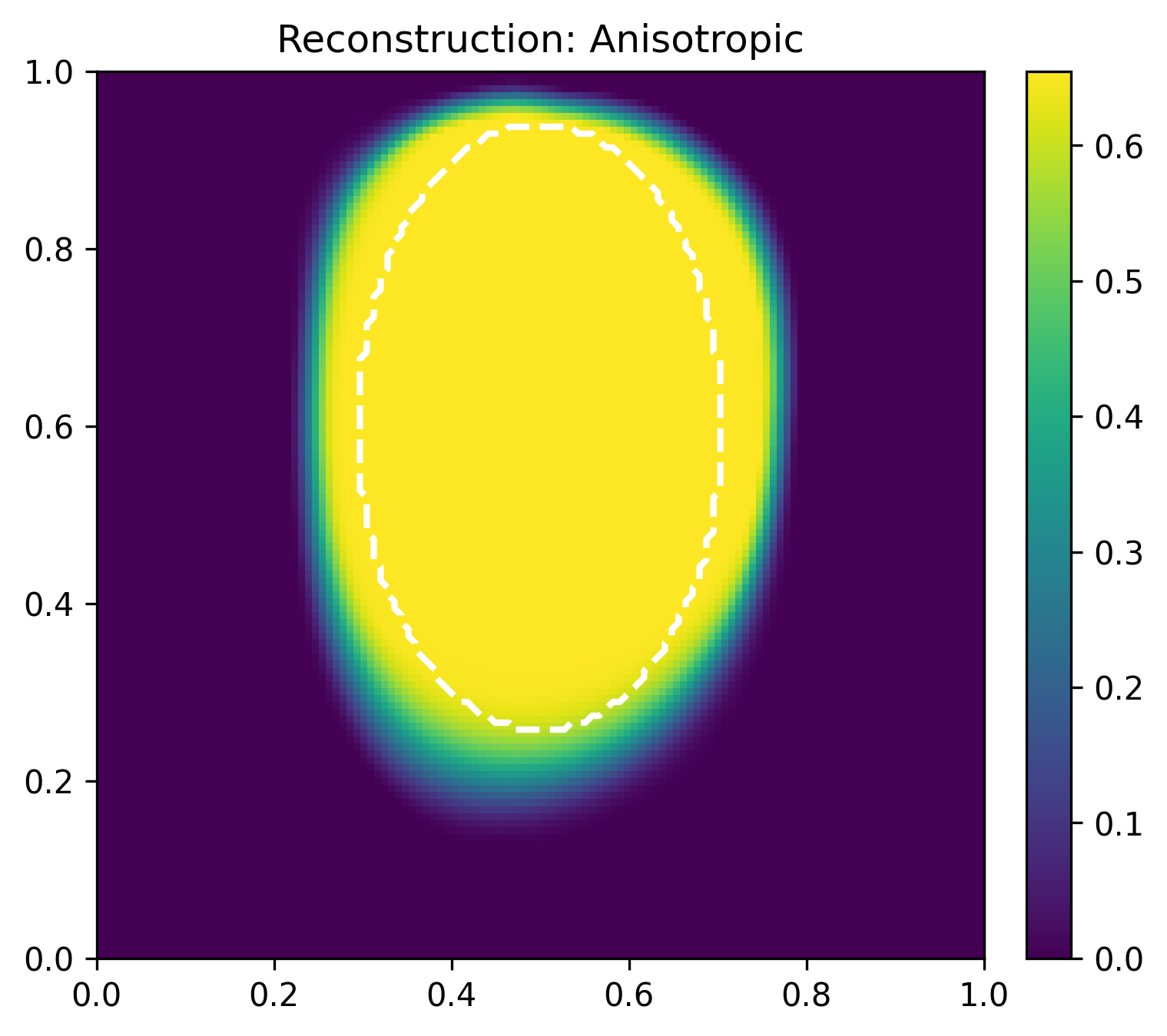}
        \caption{Recovery using $\TVtw$}
    \end{subfigure}
    \begin{subfigure}[t]{0.3\linewidth}        
        \centering
        \includegraphics[trim=0 0 0 19, clip, width=\linewidth]{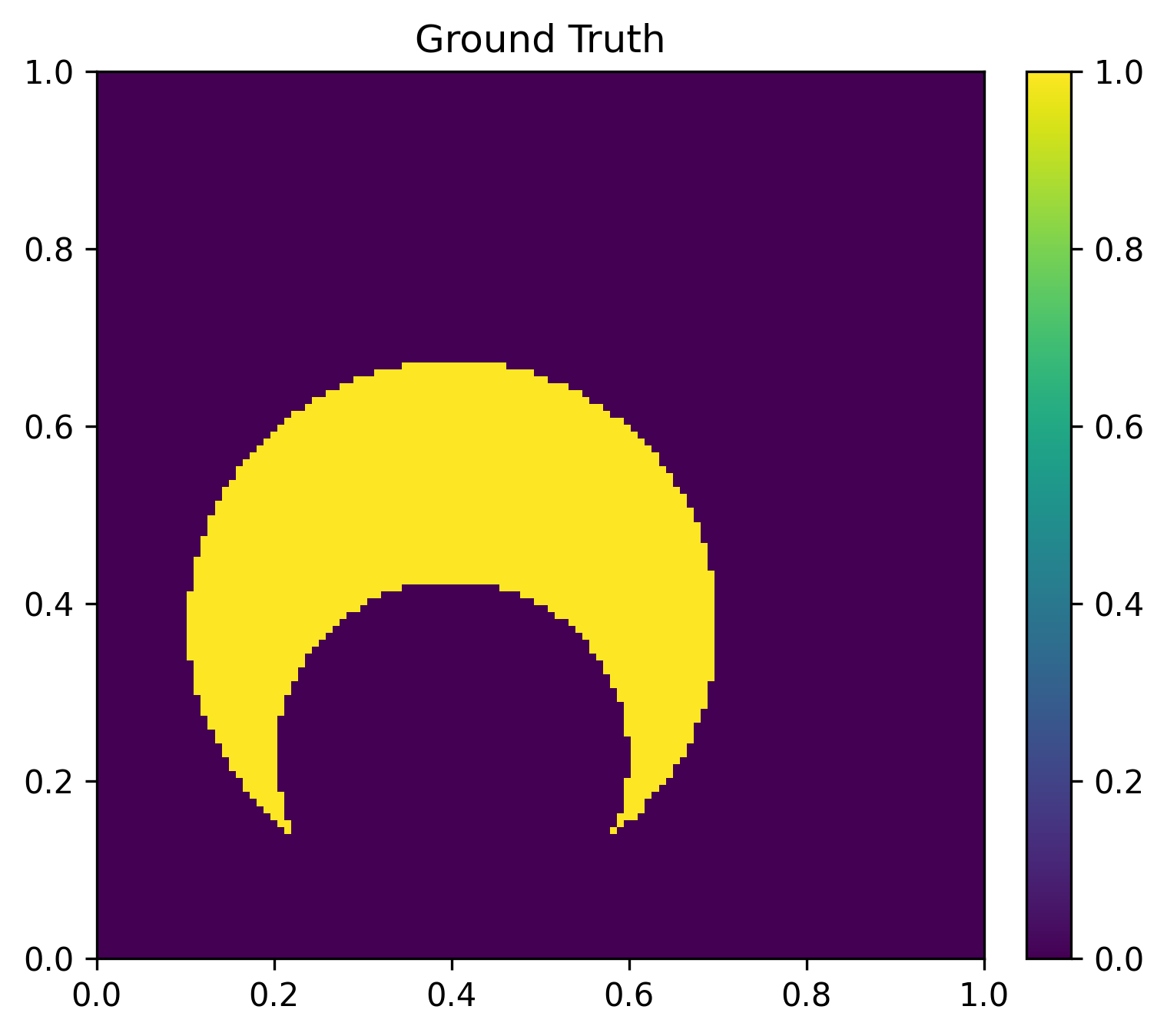}
        \caption{True source}
    \end{subfigure}
        \begin{subfigure}[t]{0.3\linewidth}        
        \centering
        \includegraphics[trim=0 0 0 19, clip, width=\linewidth]{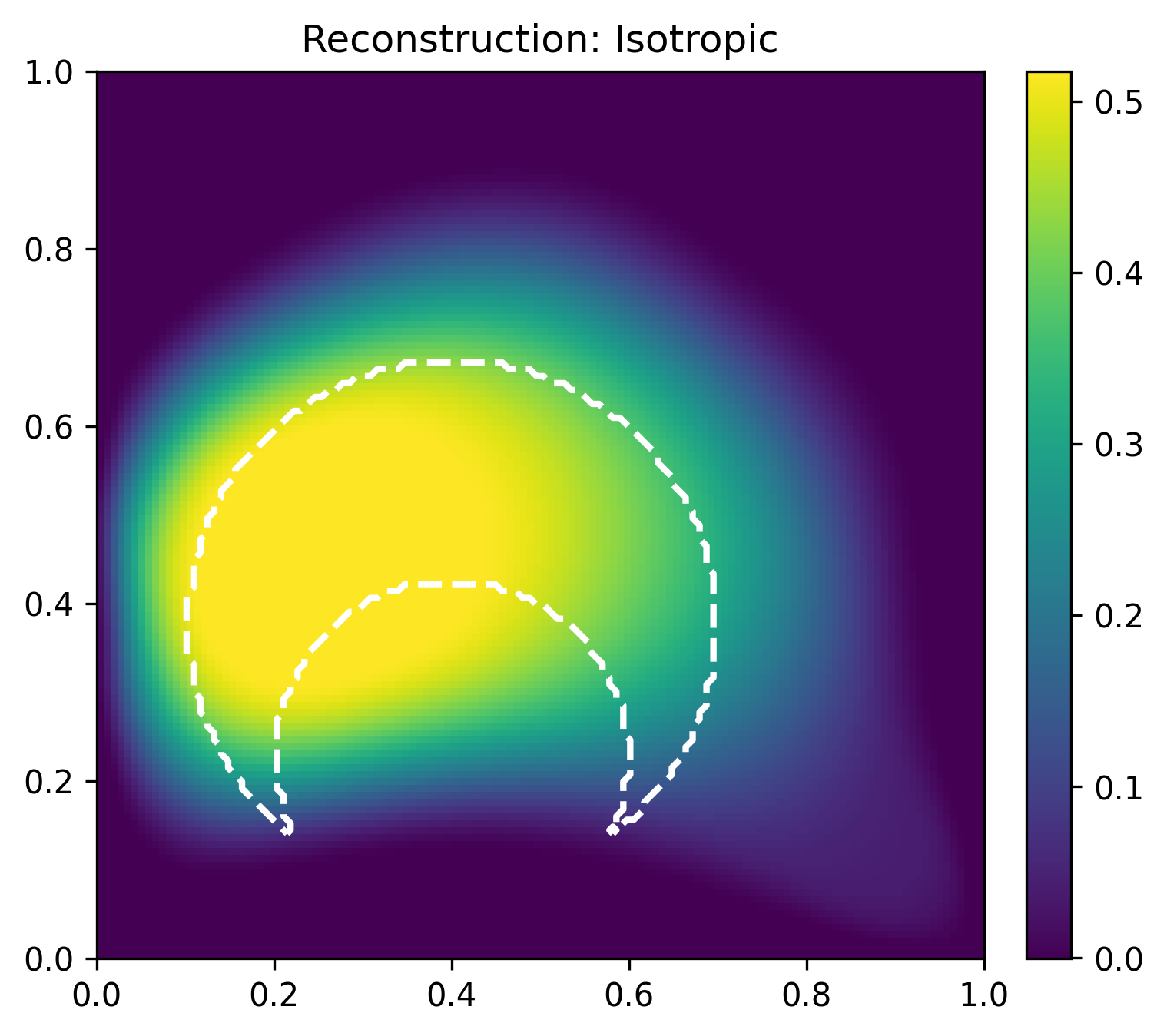}
        \caption{Recovery using $\TVw$}
    \end{subfigure}
    \begin{subfigure}[t]{0.3\linewidth}        
        \centering
        \includegraphics[trim=0 0 0 19, clip, width=\linewidth]{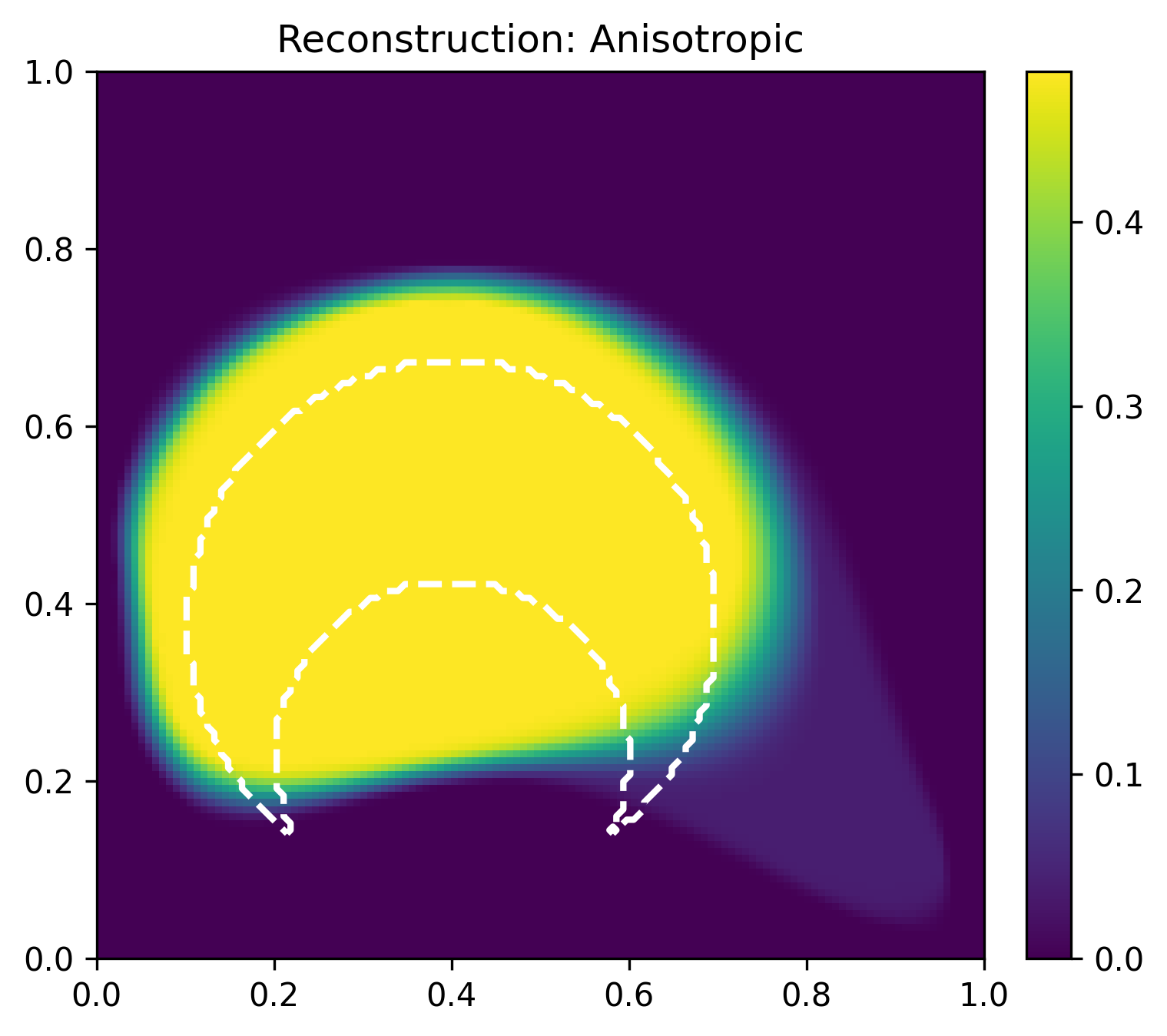}
        \caption{Recovery using $\TVtw$}
    \end{subfigure}
    \begin{subfigure}[t]{0.3\linewidth}        
        \centering
        \includegraphics[trim=0 0 0 19, clip, width=\linewidth]{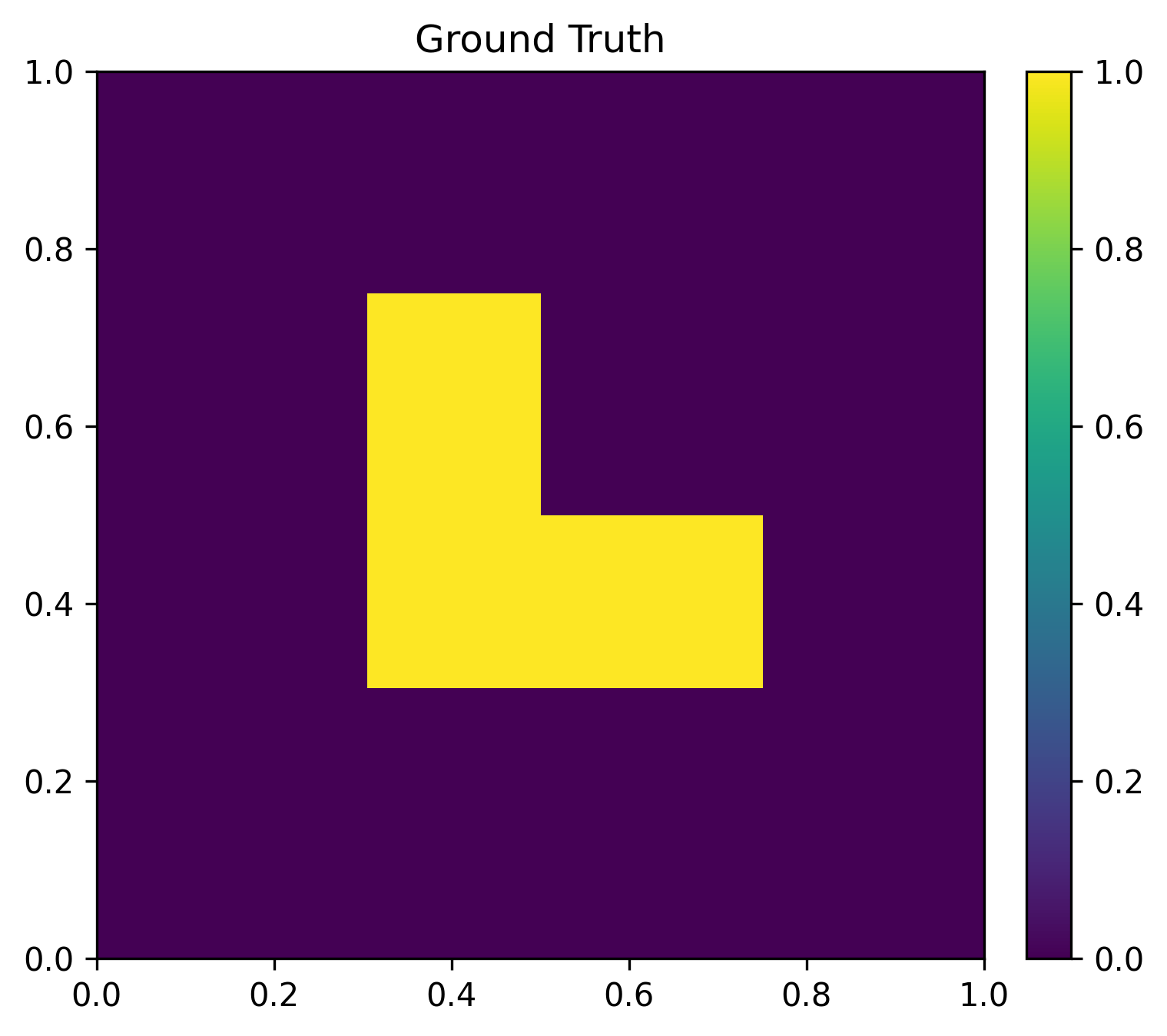}
        \caption{True source}
    \end{subfigure}
    \begin{subfigure}[t]{0.3\linewidth}        
        \centering
        \includegraphics[trim=0 0 0 19, clip, width=\linewidth]{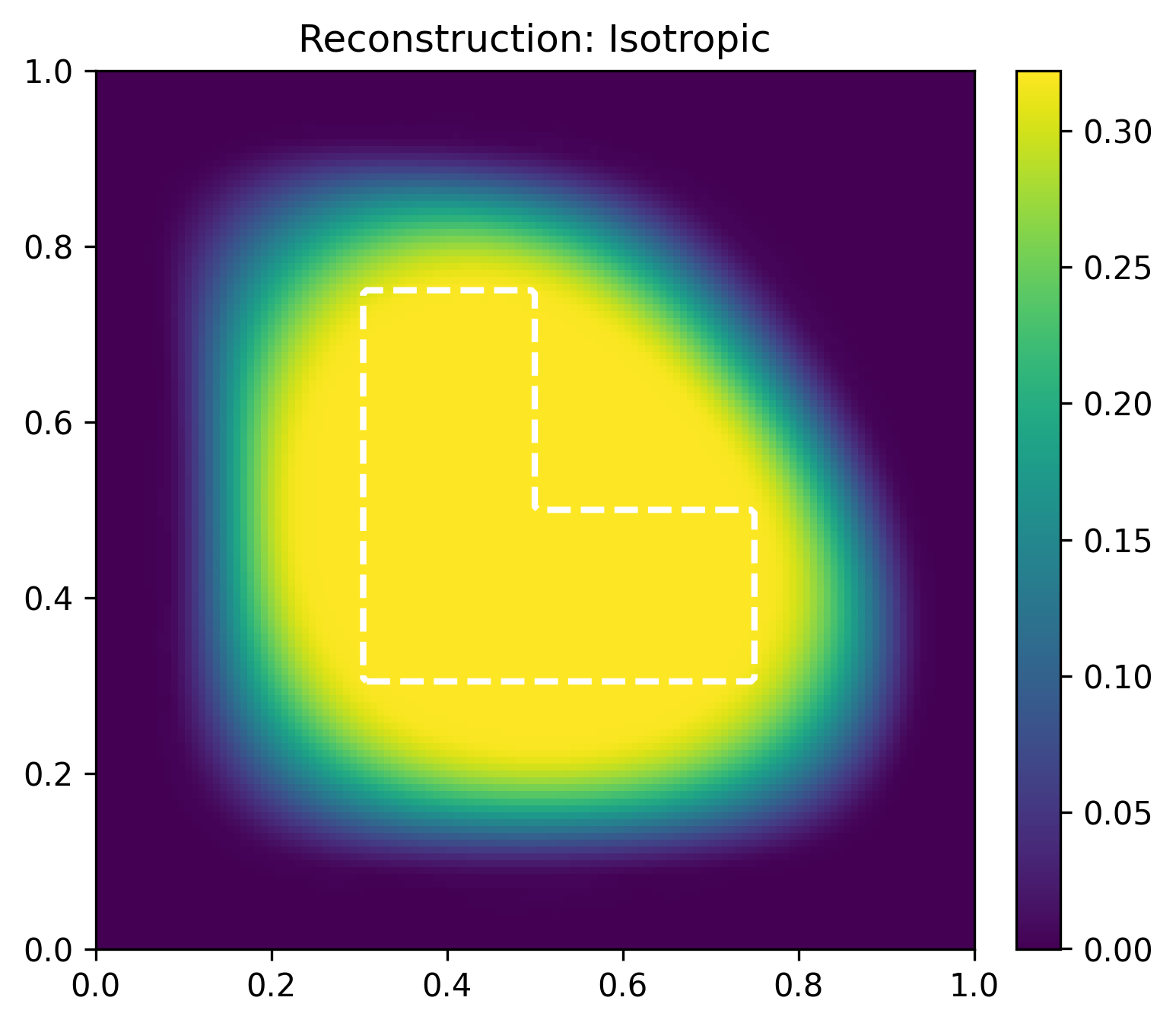}
        \caption{Recovery using $\TVw$}
    \end{subfigure}
    \begin{subfigure}[t]{0.3\linewidth}        
    \centering
    \includegraphics[trim=0 0 0 19, clip, width=\linewidth]{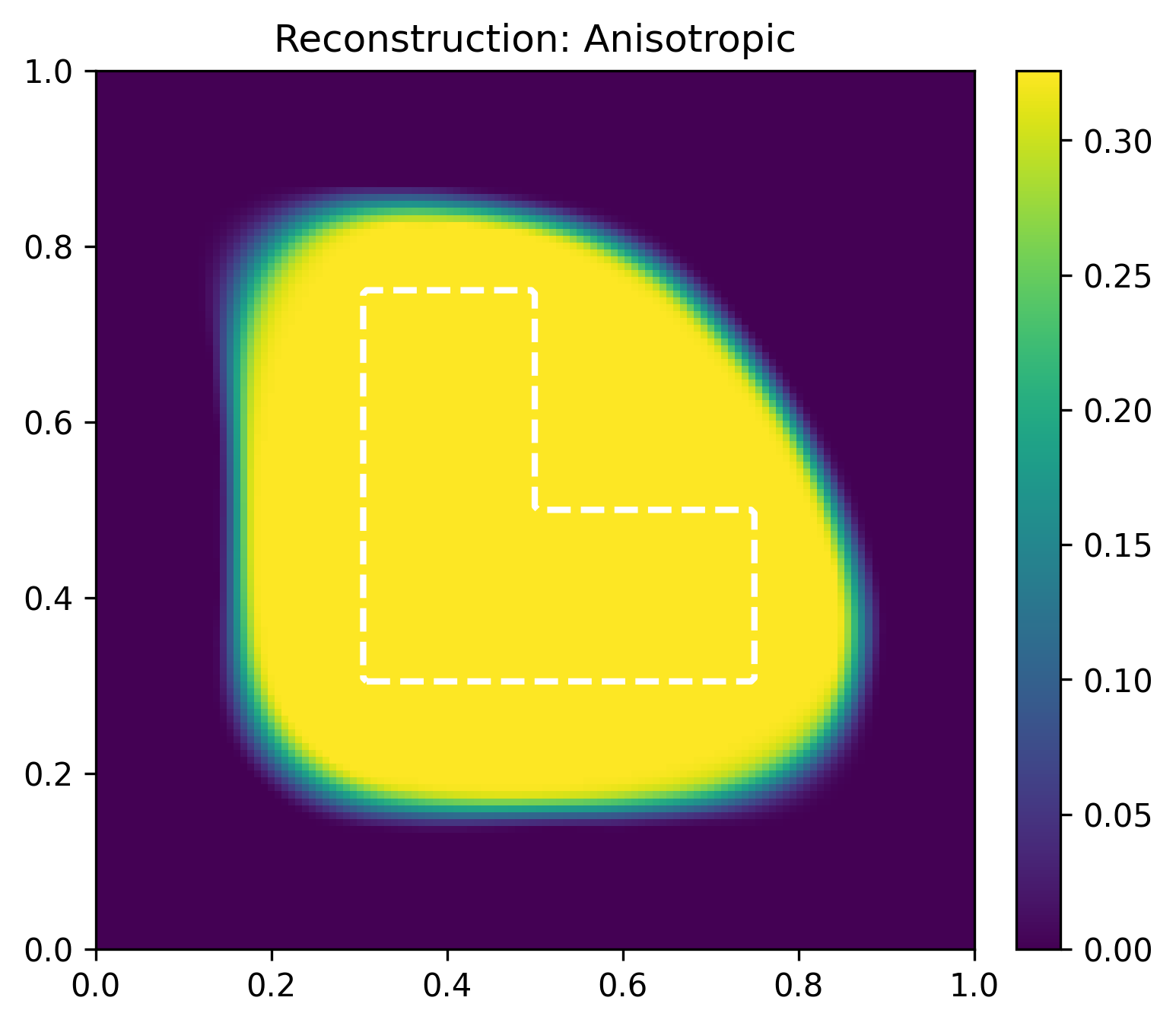}
    \caption{Recovery using $\TVtw$}
    \end{subfigure}
    \caption{Comparison of the true source and inverse recoveries applying the weighted TV ($\TVw$) and directionally weighted TV ($\TVtw$) methods, where the weights are generated with Dirichlet boundary conditions. In all simulations, we set the boundary penalty term $\beta$ equal to the TV-regularization parameter $\alpha$.}
    \label{fig:shapes}
\end{figure}

\begin{figure}[H]
    \centering
    \begin{subfigure}[t]{0.35\linewidth}        
        \centering
        \includegraphics[trim=0 0 0 19, clip, width=\linewidth]{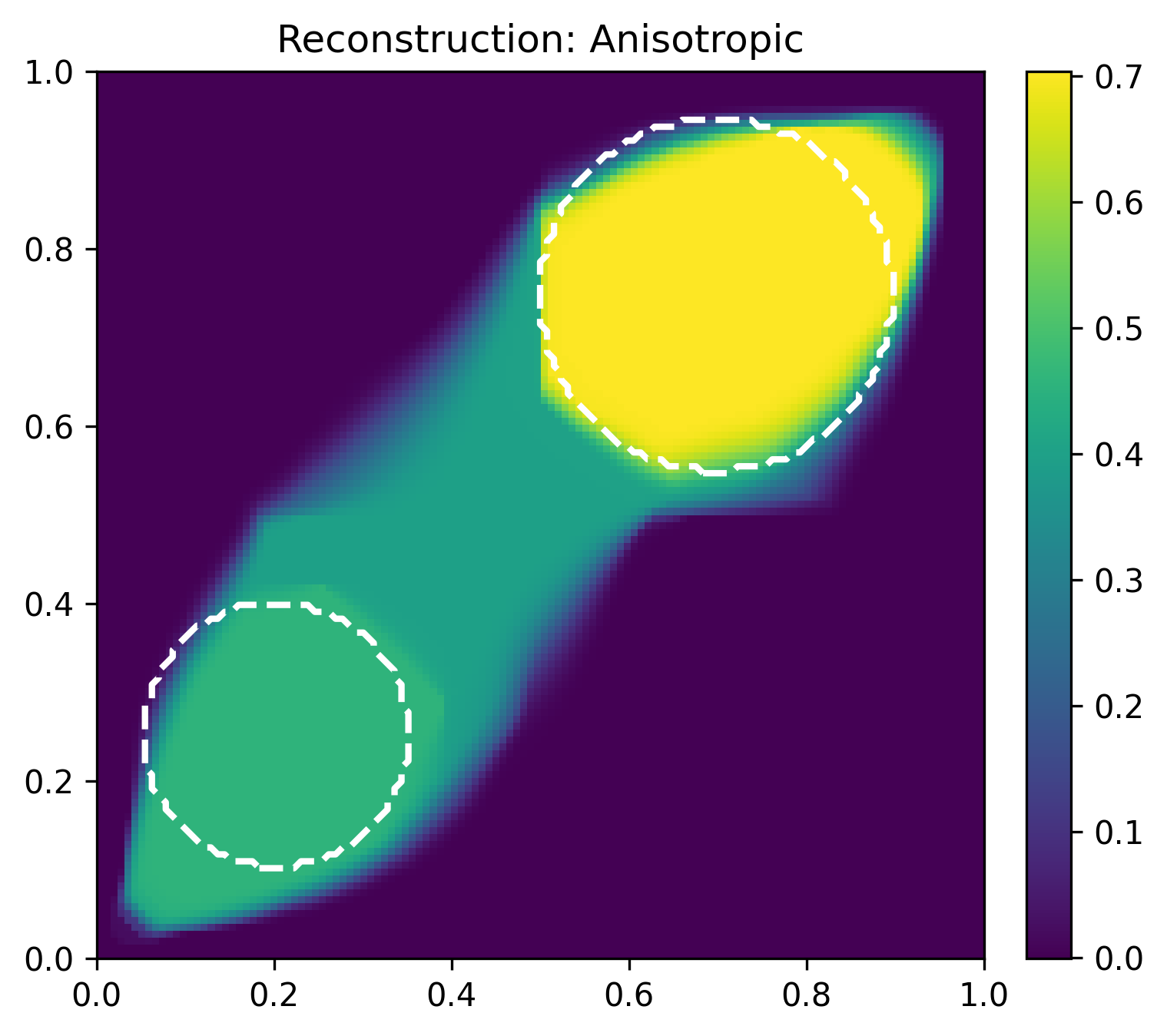}
        \caption{}
    \end{subfigure}
    \begin{subfigure}[t]{0.35\linewidth}        
        \centering
        \includegraphics[trim=0 0 0 19, clip, width=\linewidth]{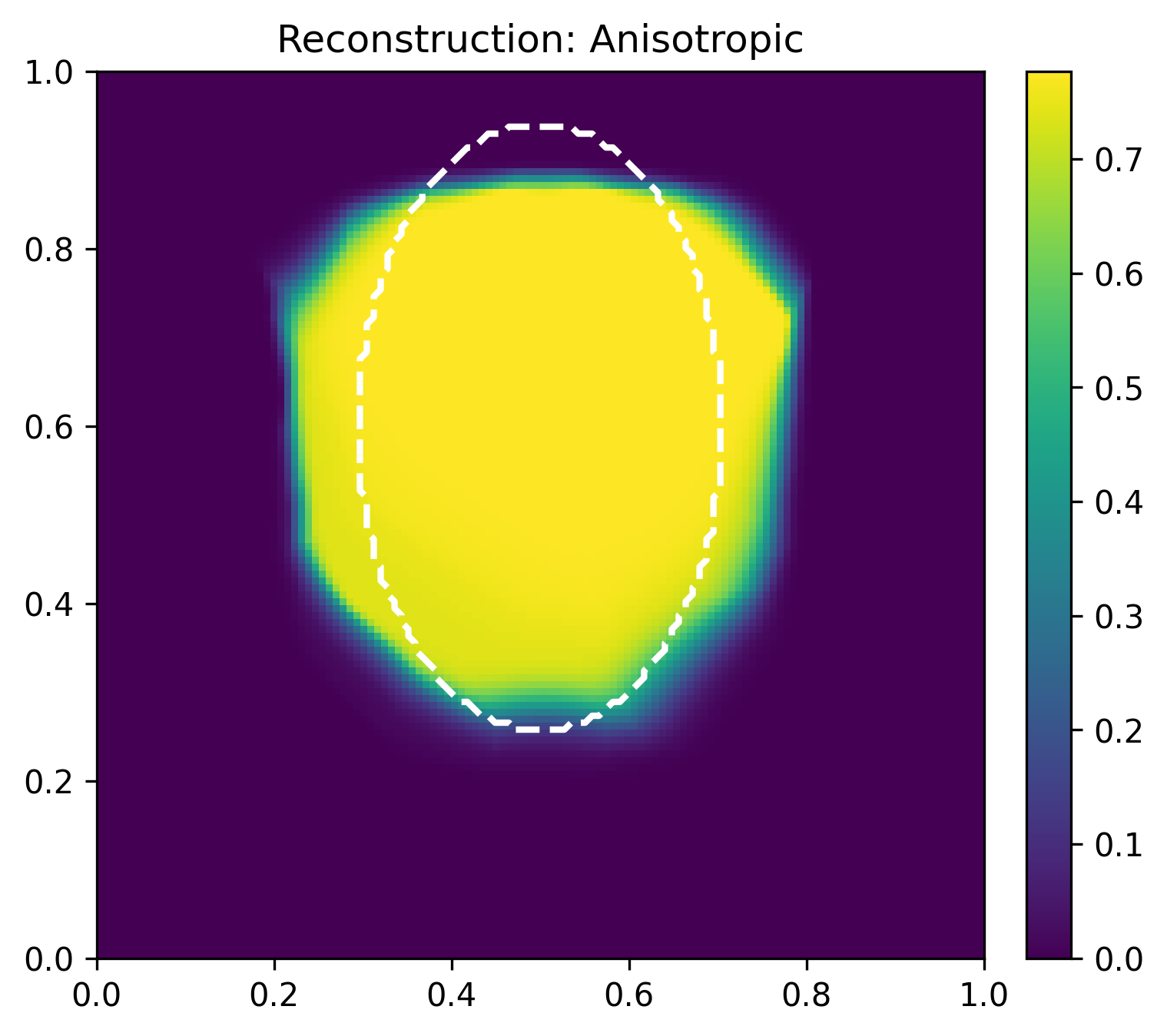}
        \caption{}
    \end{subfigure}\par
    \begin{subfigure}[t]{0.35\linewidth}        
        \centering
        \includegraphics[trim=0 0 0 19, clip, width=\linewidth]{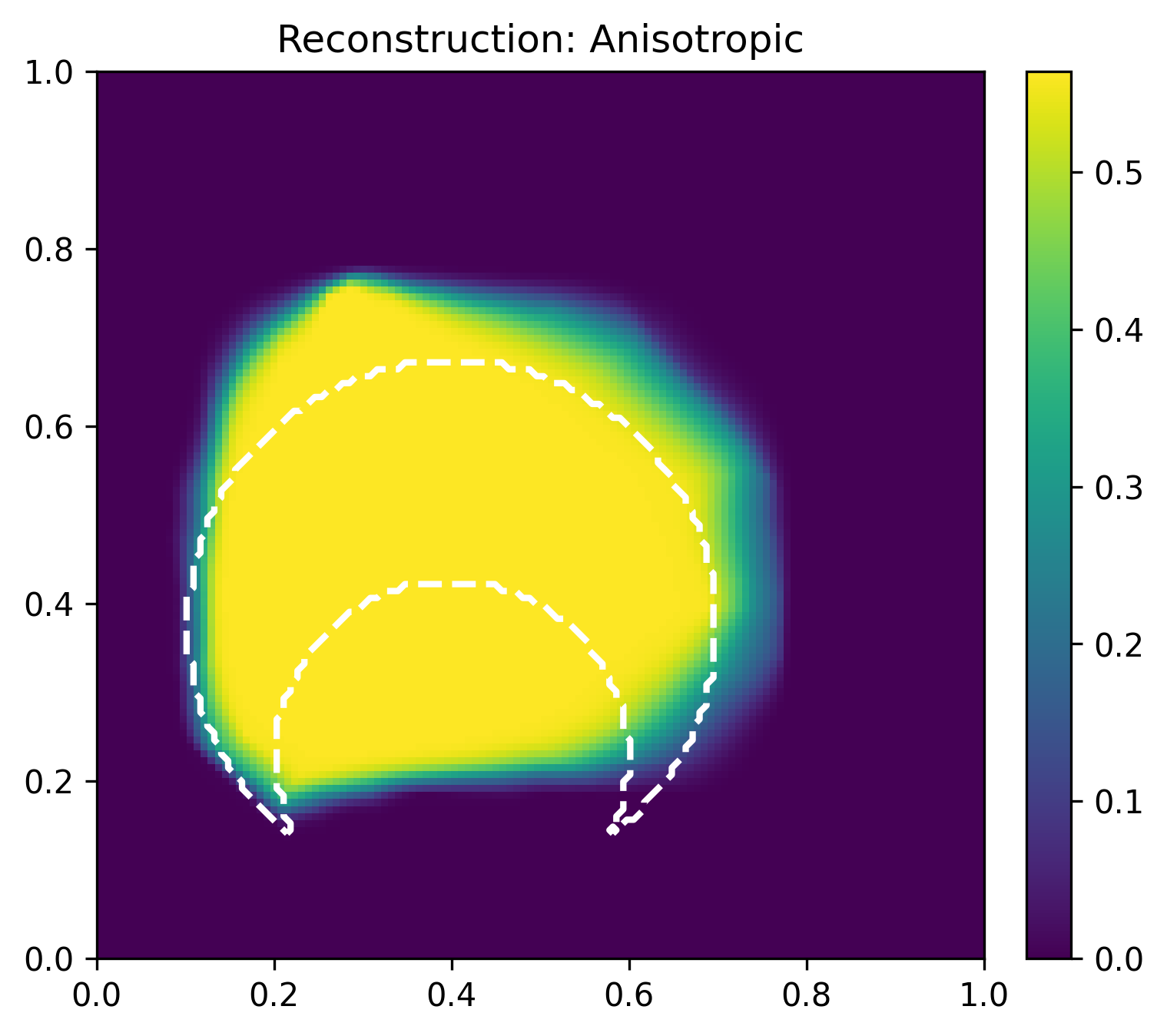}
        \caption{}
    \end{subfigure}
    \begin{subfigure}[t]{0.35\linewidth}        
        \centering
        \includegraphics[trim=0 0 0 19, clip, width=\linewidth]{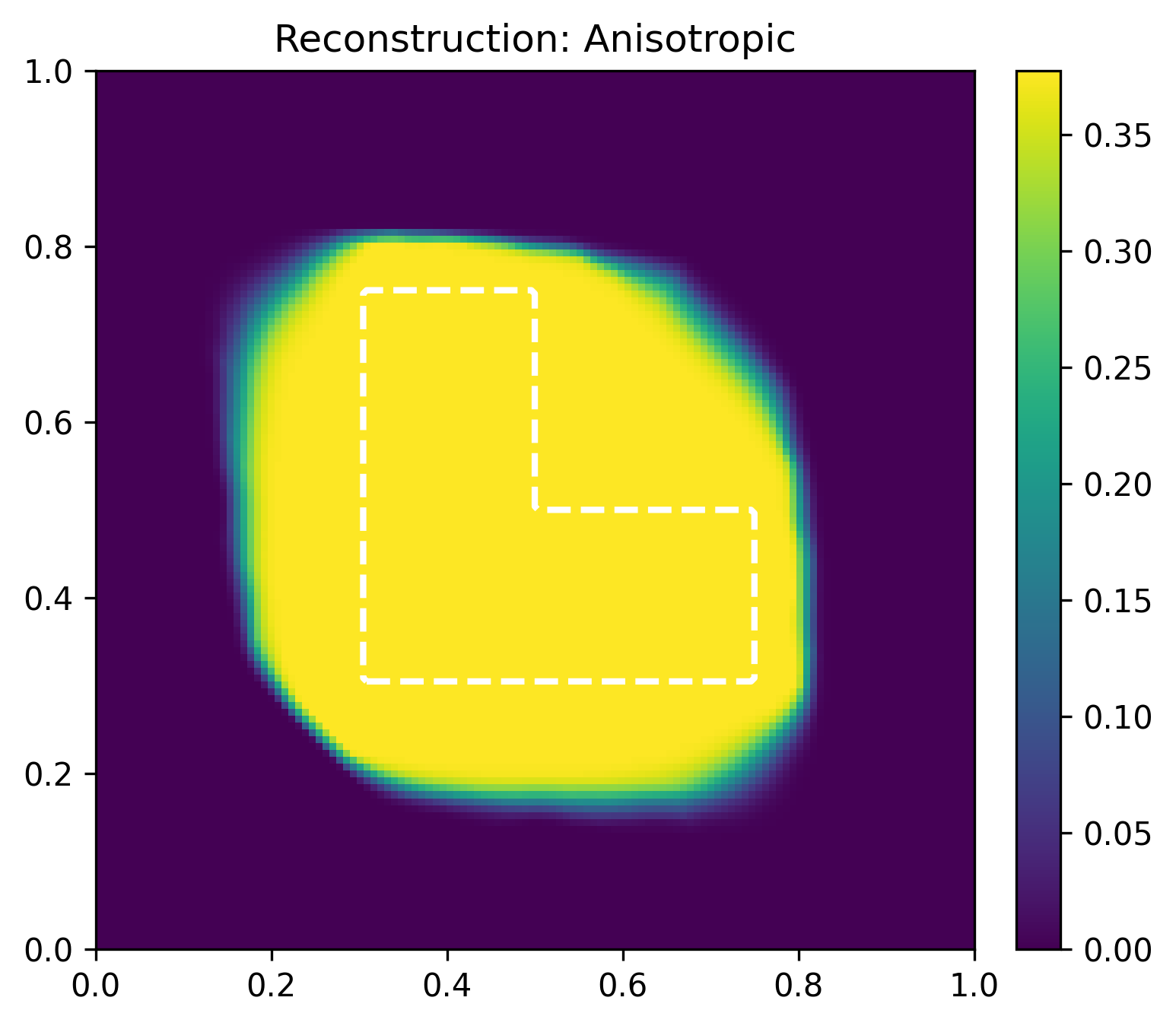}
        \caption{}
    \end{subfigure}
    \caption{Inverse reconstructions using $p = 1$ in the definition of the weights \eqref{def:tw}. The corresponding true sources are presented in Figure \ref{fig:shapes} (and as dashed curves in plots). 
    }
    \label{fig:l1norm}
\end{figure}

\subsection*{Gravimetry inspired problem}
The final experiments still concerns the PDE in \eqref{eq:screenedPoisson}, but we are now searching for a discontinuous source extending throughout the entire domain $\Omega$, placing it into the context of a gravimetry type of problem. 

In Figure~\ref{fig:cracks}, panel~(a) shows the true source. When the solution is penalized at the boundary, i.e., when $\beta > 0$, the reconstruction degrades, as seen in panels~(b) and~(d): since the true source genuinely reaches the boundary, the boundary penalty suppresses a feature that should be present, and this happens whether or not spatial weighting is used. With $\beta = 0$, by contrast, the reconstructions improve, as shown in panels~(c), (e), and~(f). Although, without spatial weighting (panel~(c)), the interface is estimated too close to the boundary. Spatial weighting removes this bias: both the Dirichlet-based weights (panel~(e)) and the Neumann-based weights (panel~(f)) yield good reconstructions. 

However, when we restrict the observation domain to the subset $\Gamma = \{y \in \partial\Omega: y = 1\} \subset \partial\Omega$, we can observe a significant distinction between the reconstructions depending on whether we use Dirichlet or Neumann conditions for generating the Green's functions, with the Neumann approach clearly being superior, cf. Figure \ref{fig:cracks_partial}. 

\begin{figure}[H]
    \centering
    \begin{subfigure}[t]{0.3\linewidth}
        \centering
        \includegraphics[trim=0 0 0 19, clip, width=\linewidth]{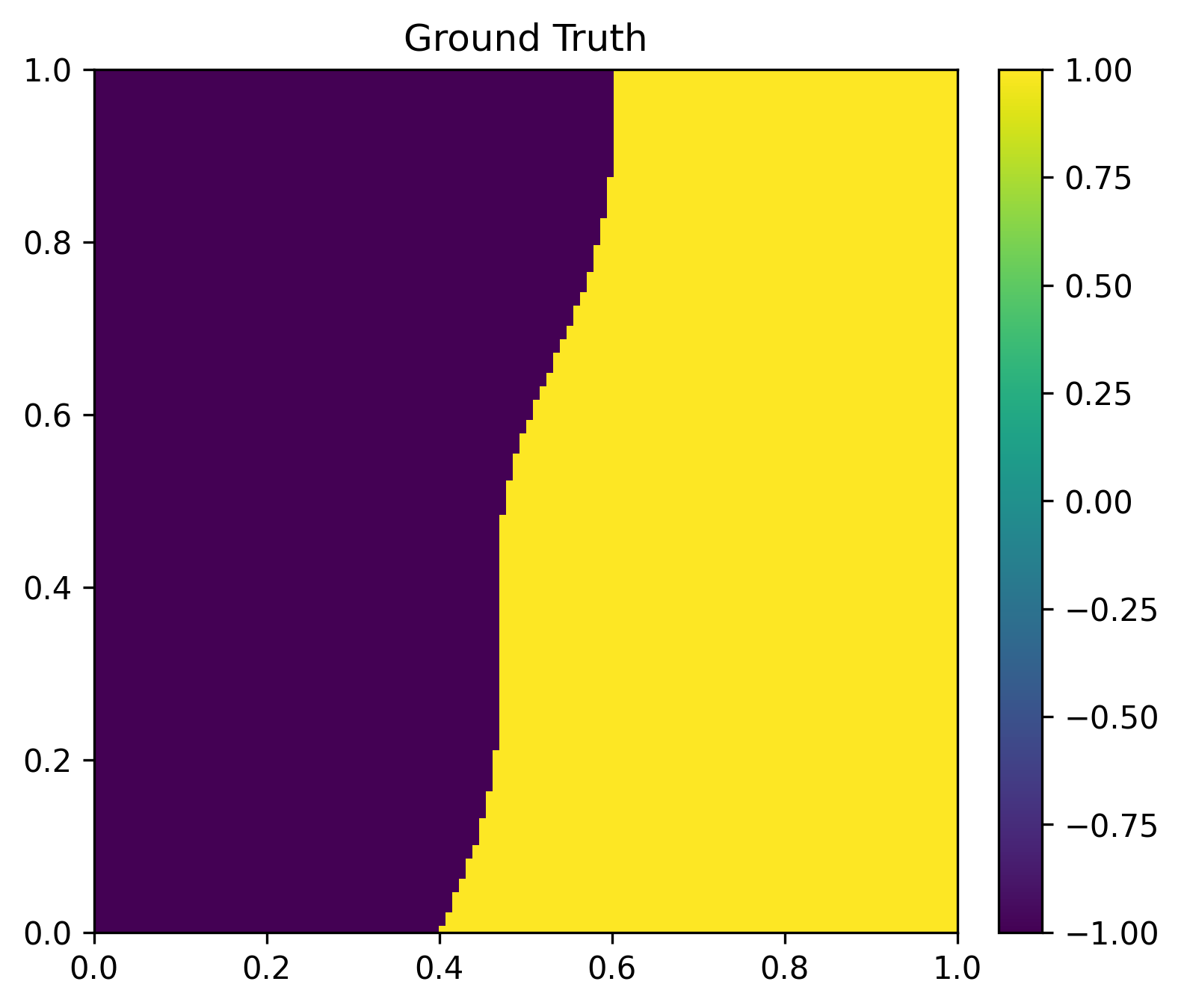}
        \caption{True interface}
    \end{subfigure}
    \begin{subfigure}[t]{0.3\linewidth}
        \centering
        \includegraphics[trim=0 0 0 19, clip, width=\linewidth]{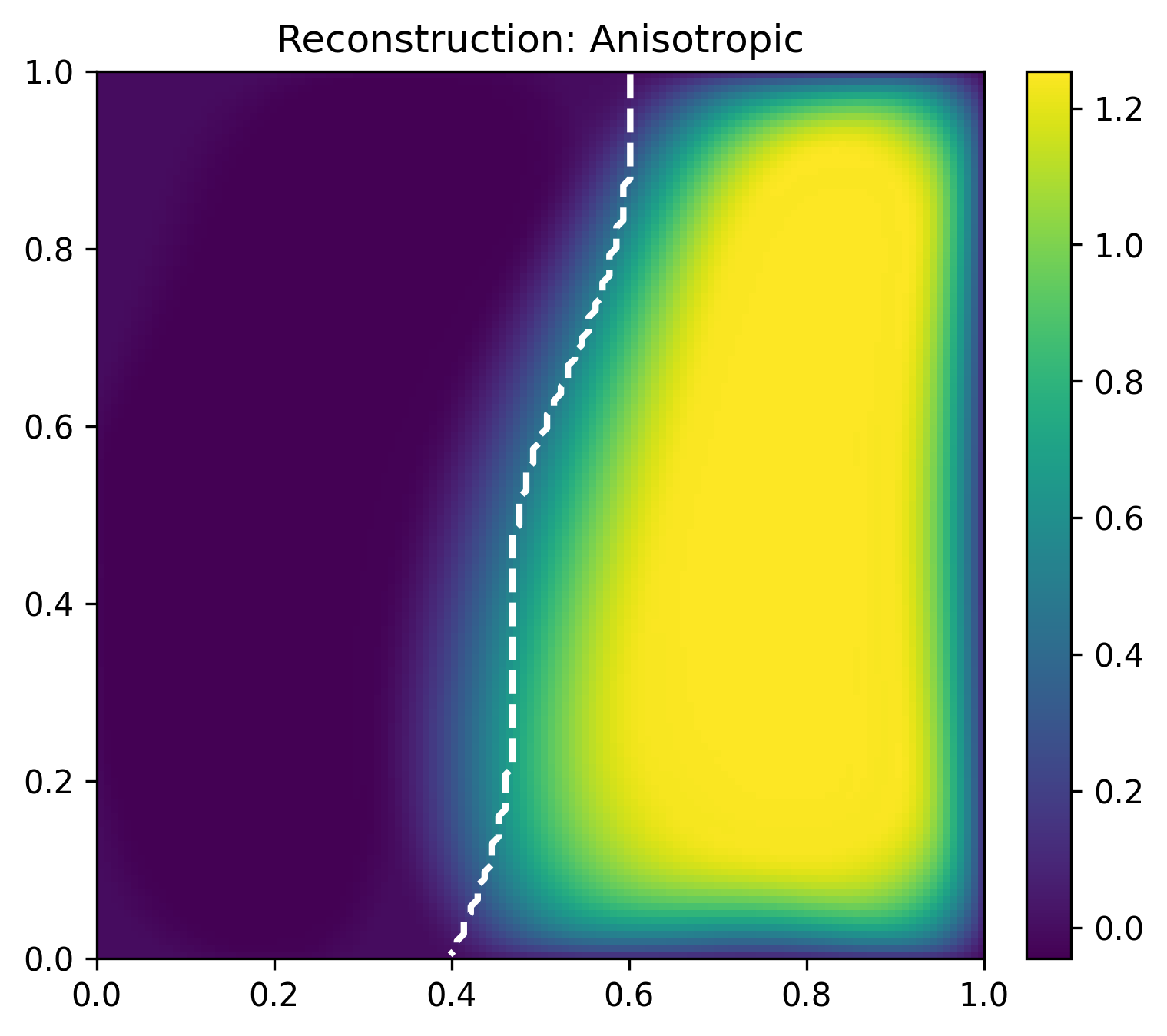}
        \caption{Standard TV, \\ \centering $\beta>0$}
    \end{subfigure}
    \begin{subfigure}[t]{0.3\linewidth}
        \centering
        \includegraphics[trim=0 0 0 19, clip, width=\linewidth]{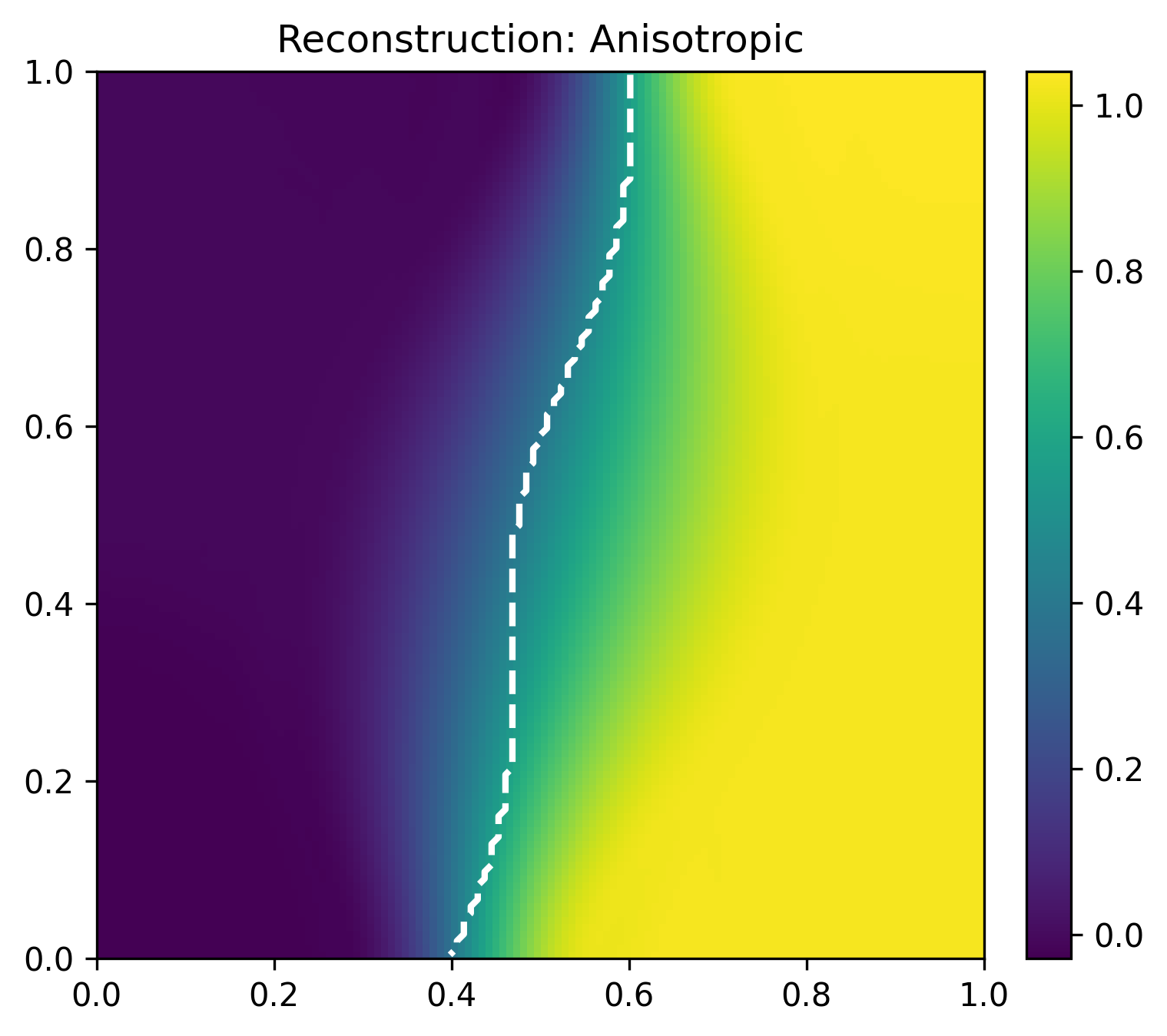}
        \caption{Standard TV, \\ \centering $\beta=0$}
    \end{subfigure}
    \begin{subfigure}[t]{0.3\linewidth}
        \centering
        \includegraphics[trim=0 0 0 19, clip, width=\linewidth]{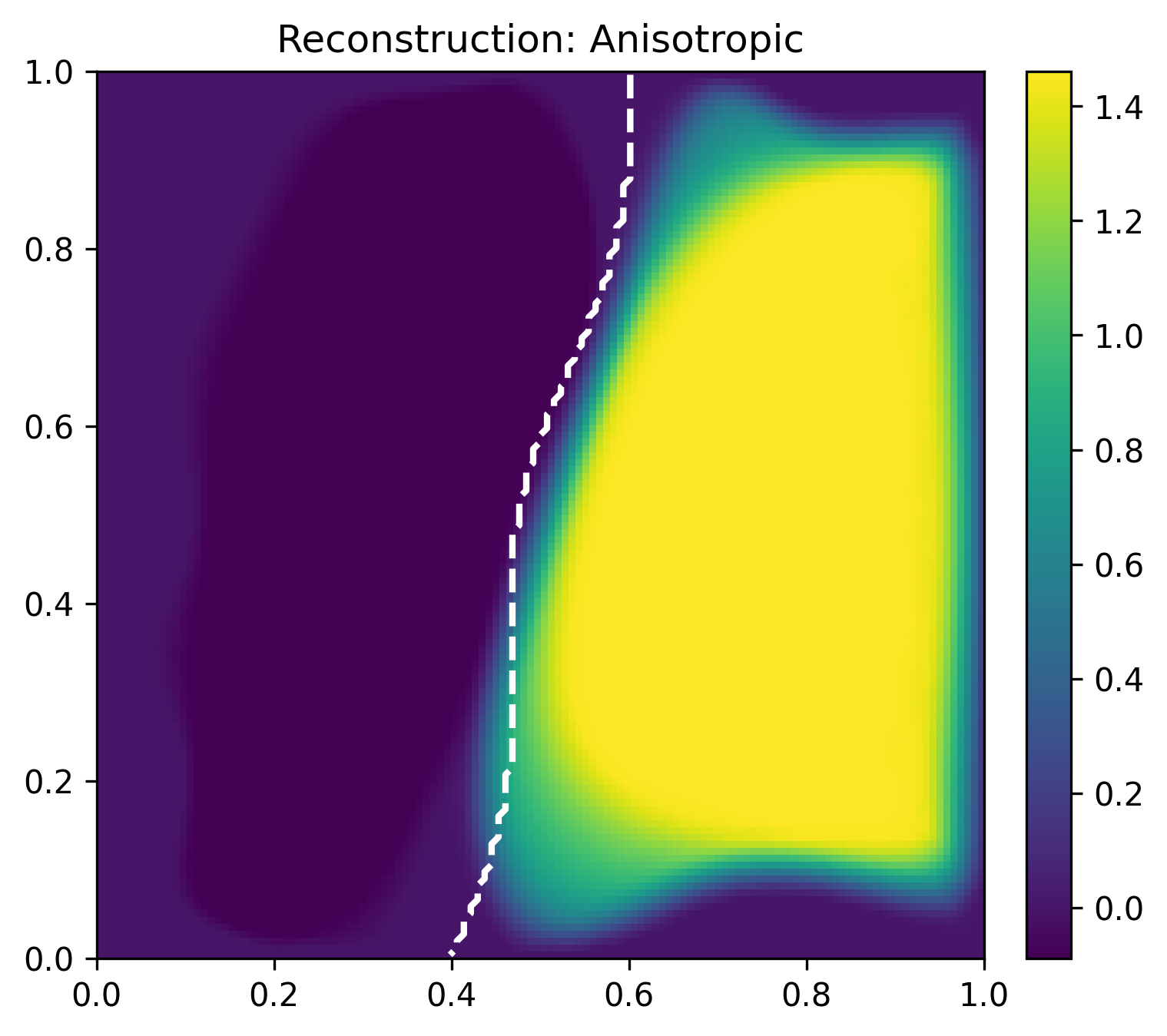}
        \caption{Dirichlet weight, \\ \centering $\beta>0$}
    \end{subfigure}
    \begin{subfigure}[t]{0.3\linewidth}
        \centering
        \includegraphics[trim=0 0 0 19, clip, width=\linewidth]{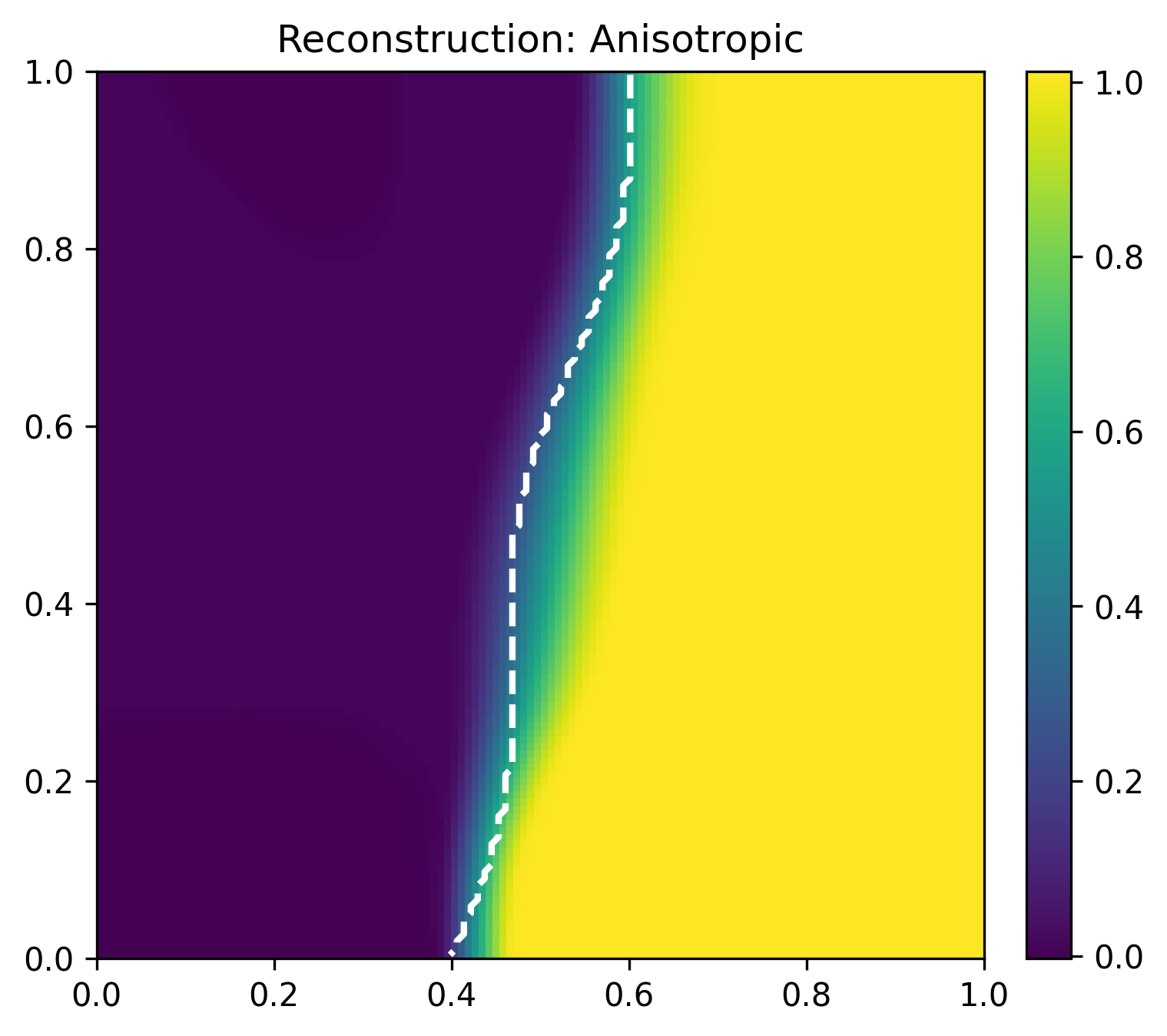}
        \caption{Dirichlet weight,\\ \centering $\beta=0$}
    \end{subfigure}
    \begin{subfigure}[t]{0.3\linewidth}
        \centering
        \includegraphics[trim=0 0 0 19, clip, width=\linewidth]{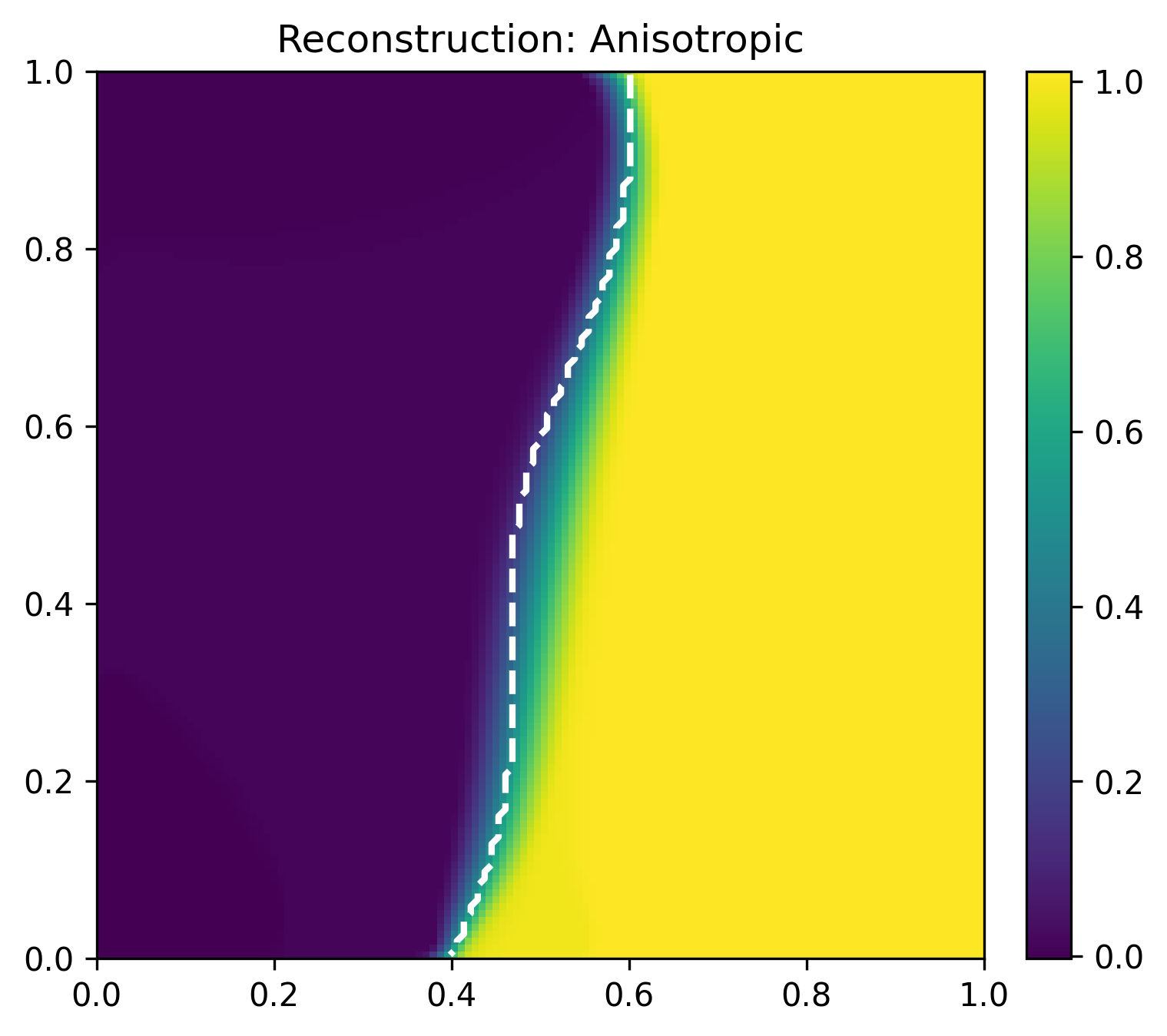}
        \caption{Neumann weight,\\ \centering $\beta=0$}
    \end{subfigure}
    \caption{Reconstruction of the material interface, comparing standard TV with directionally weighted TV, and the presence ($\beta>0$) or absence ($\beta=0$) of the boundary penalty. For the weighted reconstructions, the weight $\tw$ is generated from a Green's function with either Dirichlet or Neumann boundary conditions; the Neumann weight admits no natural boundary term, so its $\beta>0$ case is omitted.}
    \label{fig:cracks}
\end{figure}

\begin{figure}[H]
    \centering
    \begin{subfigure}[t]{0.3\linewidth}        
        \centering
        \includegraphics[trim=0 0 0 19, clip, width=\linewidth]{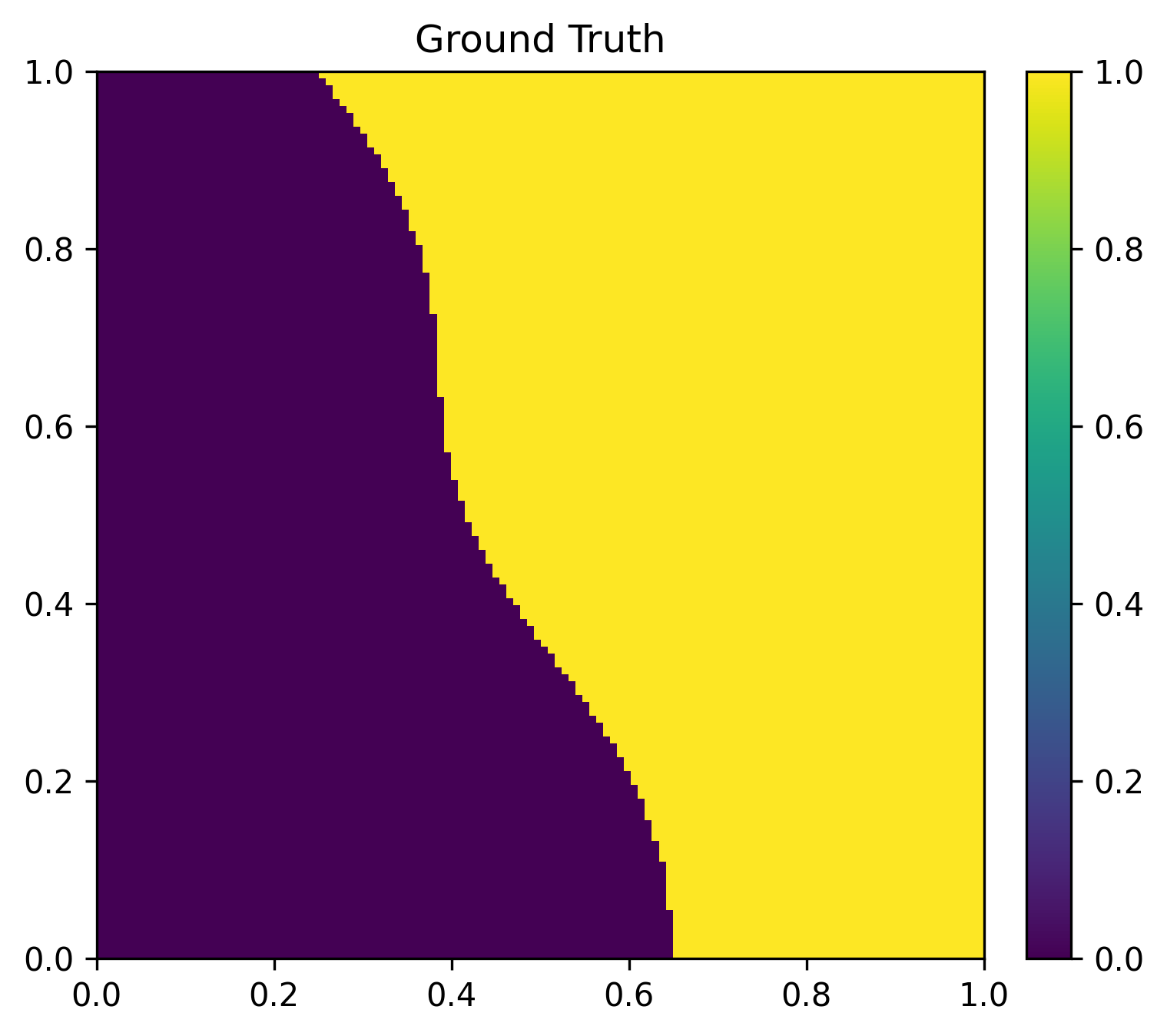}
        \caption{True interface}
    \end{subfigure}
    \begin{subfigure}[t]{0.3\linewidth}        
        \centering
        \includegraphics[trim=0 0 0 19, clip, width=\linewidth]{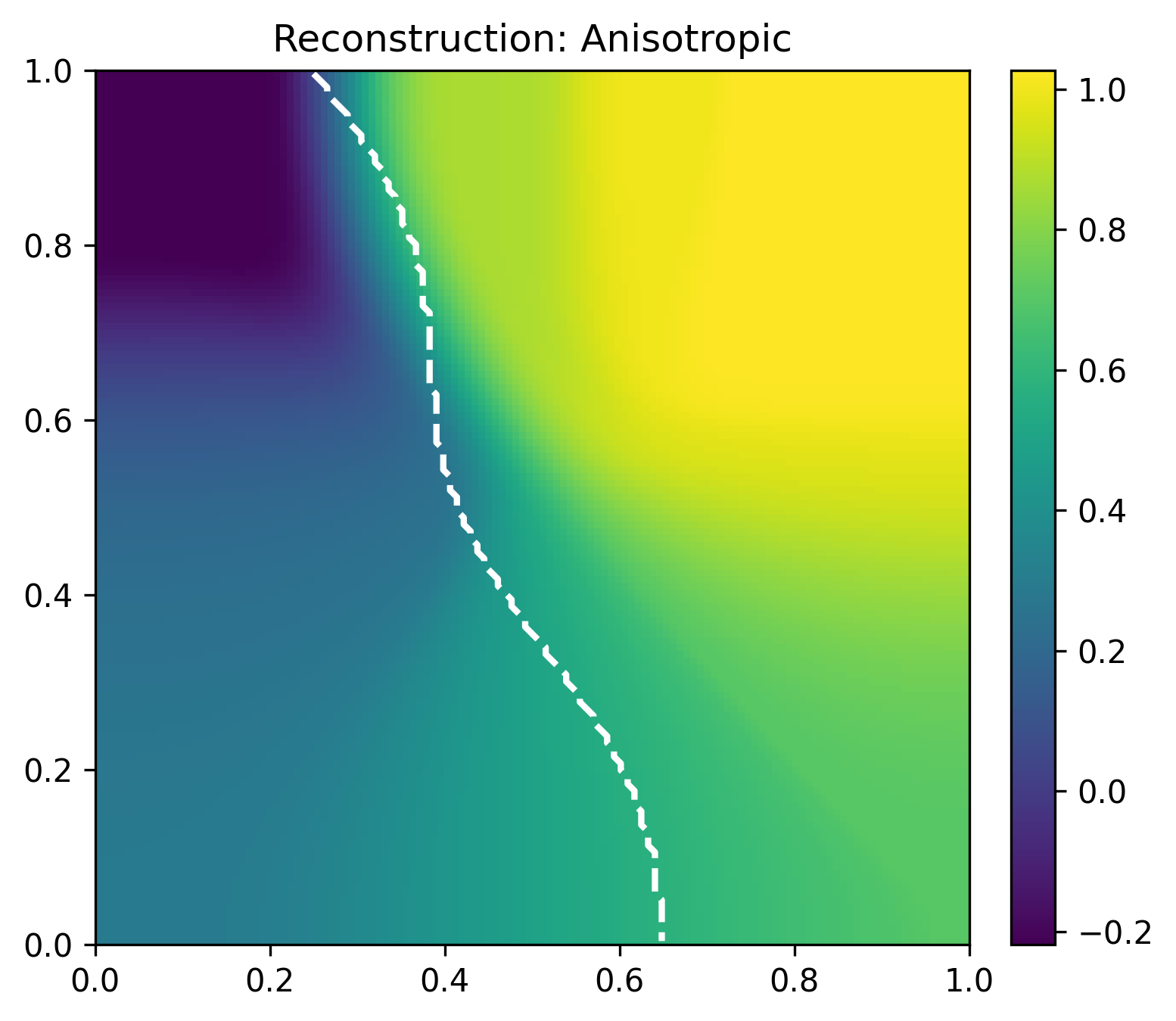}
        \caption{Dirichlet weight;\\ \centering $\beta = 0$}
    \end{subfigure}
    \begin{subfigure}[t]{0.3\linewidth}        
        \centering
        \includegraphics[trim=0 0 0 19, clip, width=\linewidth]{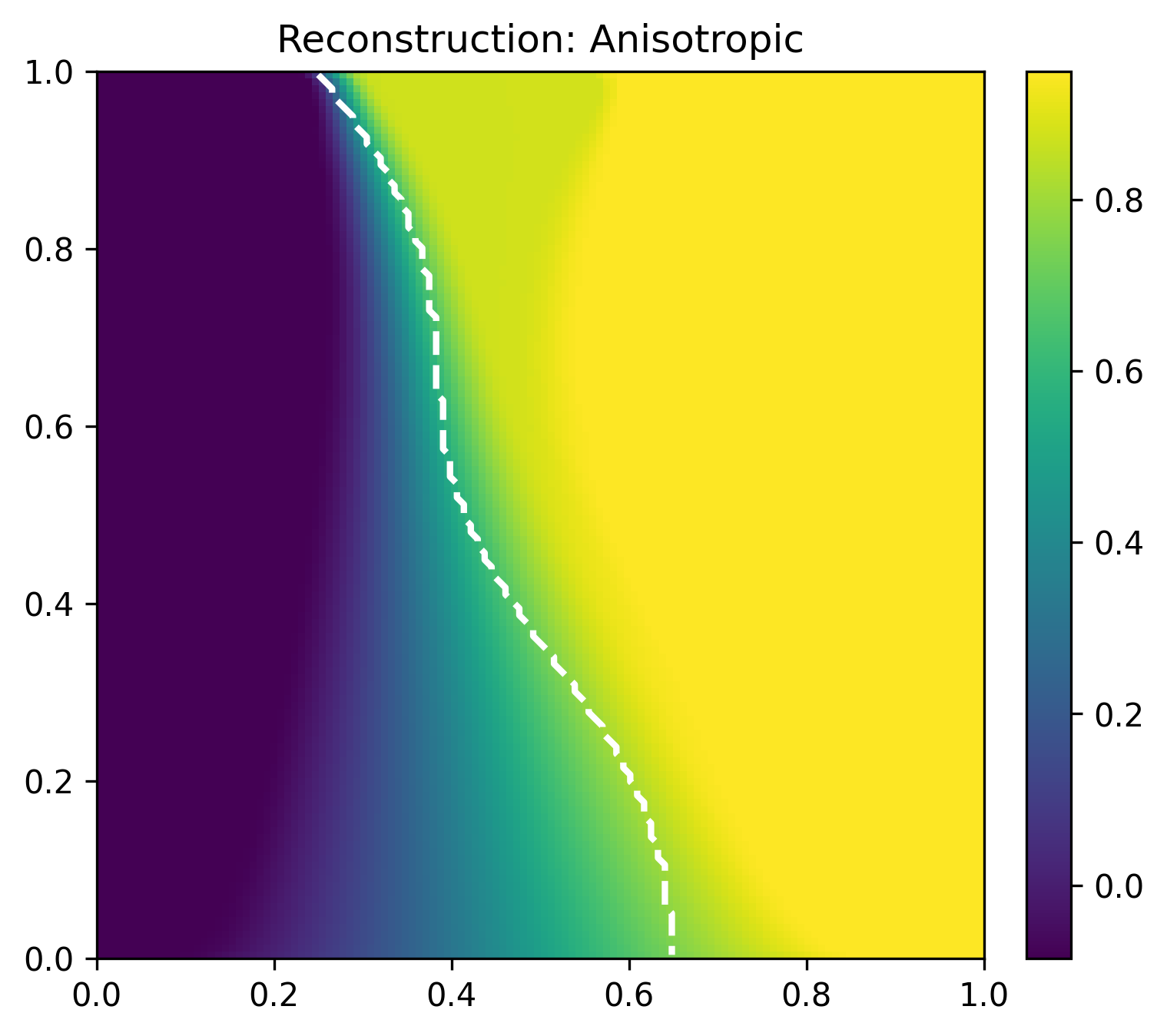}
        \caption{Neumann weight;\\ \centering $\beta = 0$}
    \end{subfigure}
    \begin{subfigure}[t]{0.3\linewidth}        
        \centering
        \includegraphics[trim=0 0 0 19, clip, width=\linewidth]{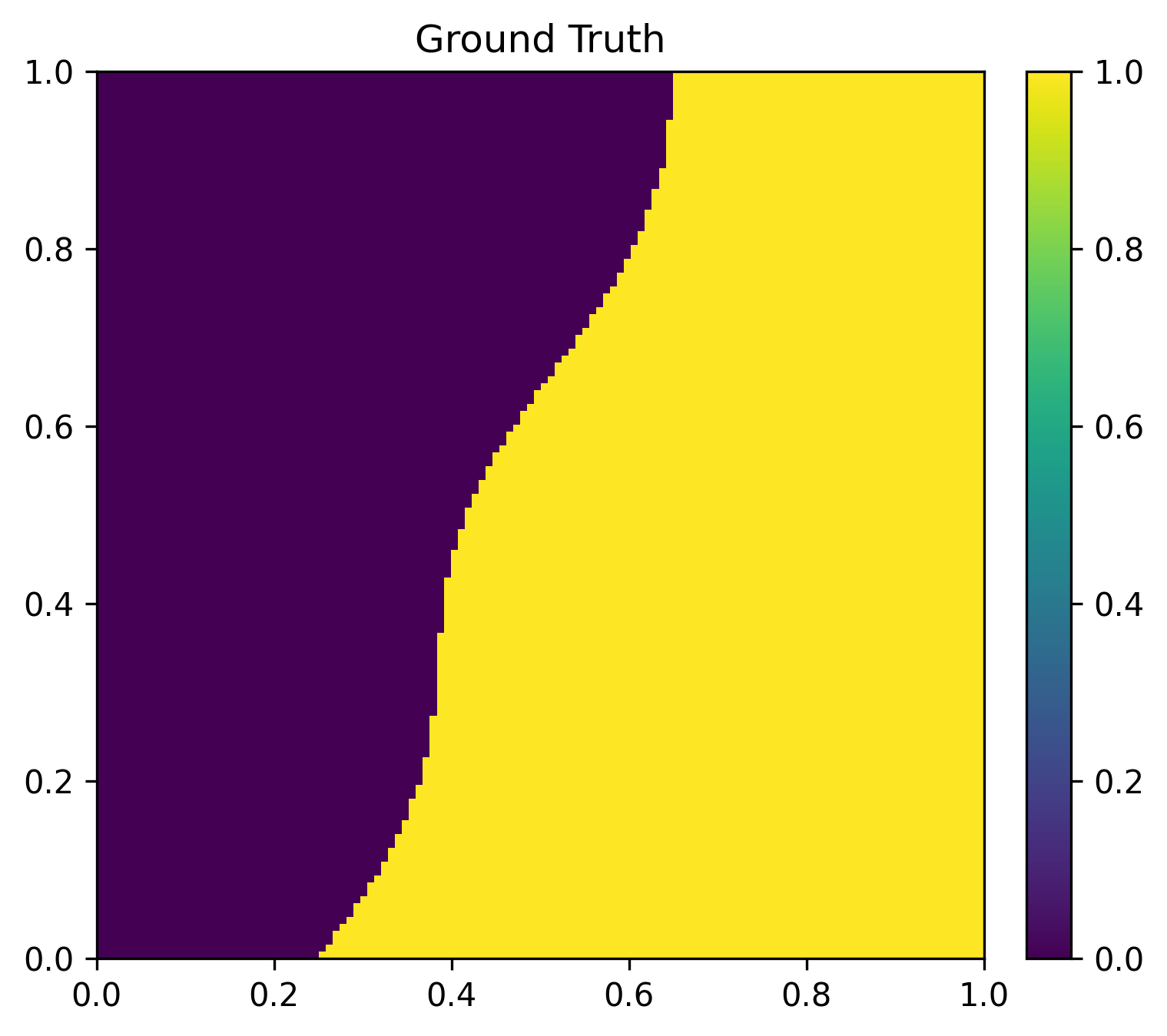}
        \caption{True interface}
    \end{subfigure}
        \begin{subfigure}[t]{0.3\linewidth}        
        \centering
        \includegraphics[trim=0 0 0 19, clip, width=\linewidth]{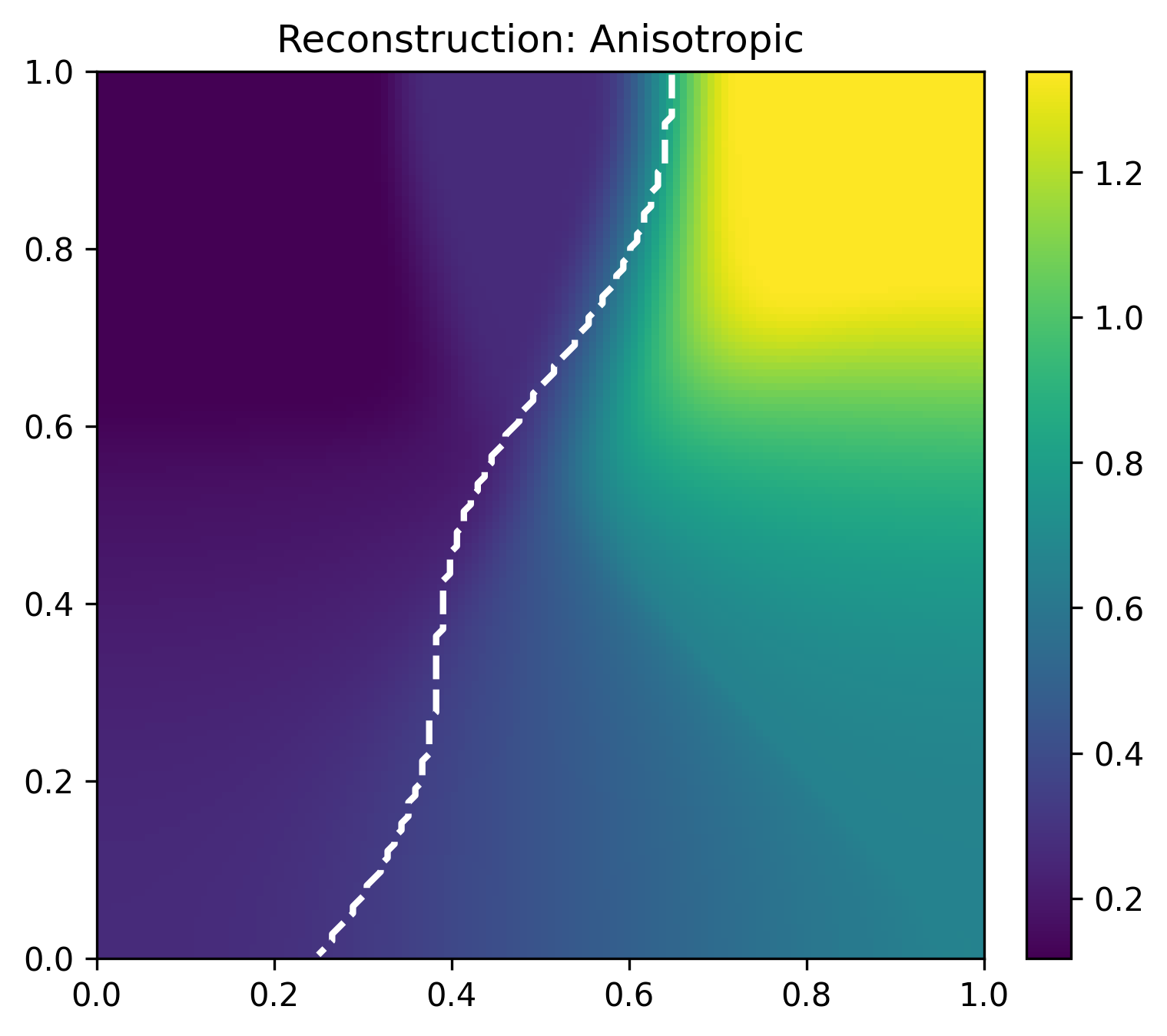}
        \caption{Dirichlet weight;\\ \centering $\beta = 0$}
    \end{subfigure}
    \begin{subfigure}[t]{0.3\linewidth}        
        \centering
        \includegraphics[trim=0 0 0 19, clip, width=\linewidth]{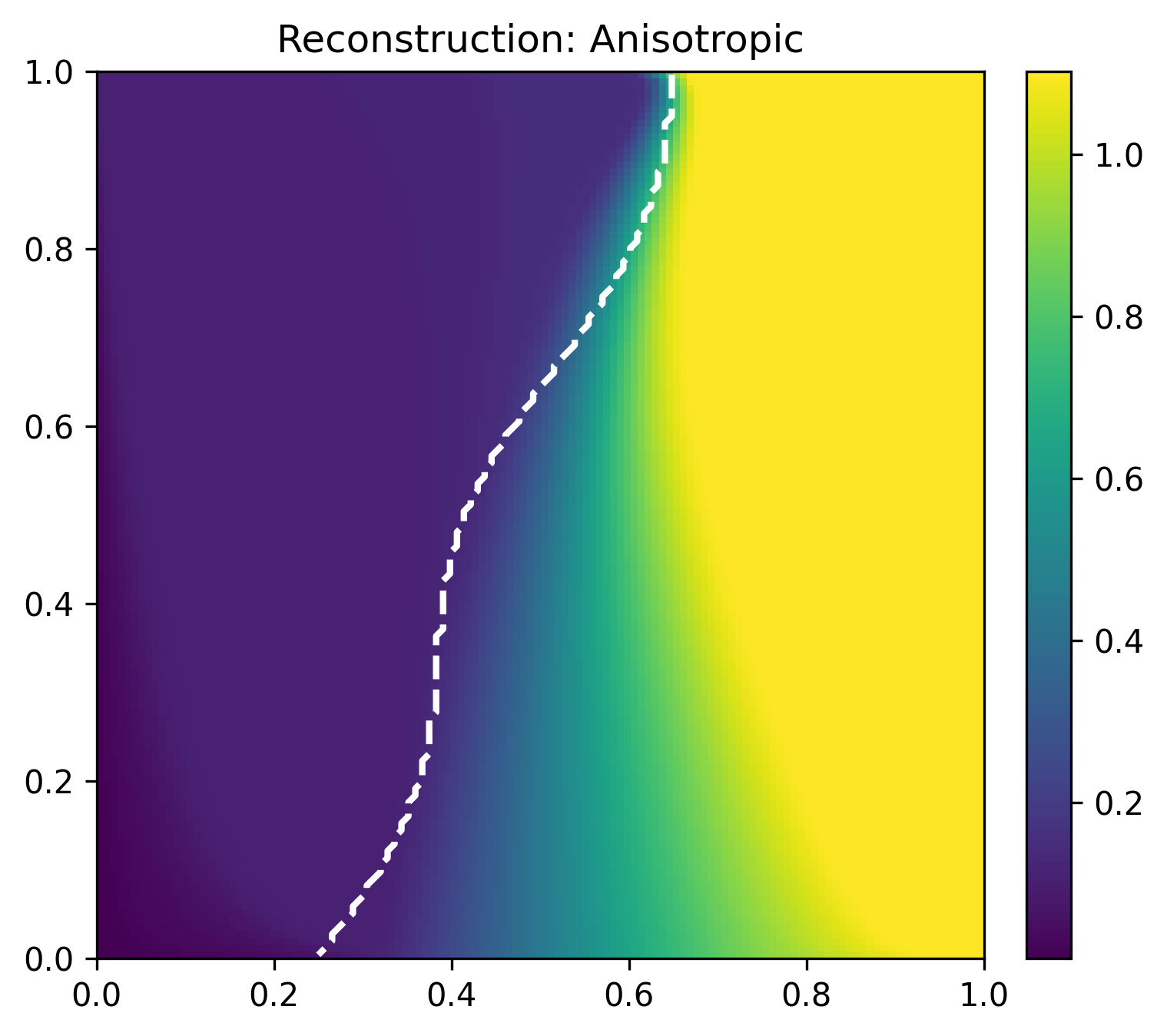}
        \caption{Neumann weight;\\ \centering $\beta = 0$}
    \end{subfigure}
    \caption{Reconstruction of the material interface, without boundary penalty, i.e., $\beta = 0$ in \eqref{def:variational_form_pre}, and with weights $\tw$ generated by Green's functions with either Dirichlet or Neumann boundary conditions. Only observations at the top of the domain, i.e., at $\Gamma = \{y \in \partial\Omega: y = 1\}$.}
    \label{fig:cracks_partial}
\end{figure}

\section{Conclusion}\label{sec:conclusion}
In this paper we have studied weighted total variation regularization for linear inverse problems whose forward operator $K \colon L^2(\Omega) \to L^2(\partial\Omega)$ has a significant null space. We have proposed and analyzed the directionally weighted variant~\eqref{def:variational_form_pre}-\eqref{def:tw}, in which the worst-case supremum over directions defining $w$, see \eqref{def:variational_form}-\eqref{def:w}, is replaced by the actual jump direction of the candidate solution, yielding a sharper but still convex regularizer $\TVtw$. 

In comparison to standard unweighted TV, we observed large improvements in recoverability for both of the weighted regularization techniques - in particular with respect to the center of mass and spatial extension. However, the directional weighting method did seem to produce better results regarding reconstruction of (some of) the shapes, compared with the use of $w$.

Both weights are constructed directly from the operator $K$ together with the Green's function for the Laplace operator on $\Omega$: depending on whether this Green's function is taken to satisfy homogeneous Neumann or homogeneous Dirichlet boundary conditions, the source representation produces either a purely interior term, corresponding to the choice $\beta = 0$ in the variational problem, or a representation with an additional boundary term coming equipped with a natural boundary weight $w_\partial$, which corresponds to the choice $\beta > 0$. Choosing the Green's function according to whether one wants to penalize boundary jumps therefore aligns the regularizer with the (expected) geometry of the unknown source. 

When the support of the true source extends to the boundary, we observed in numerical experiments that using Neumann boundary conditions outperformed the use of Dirichlet boundary conditions. Moreover, for all cases where the true support was strictly interior, we observed that the size and position of all kinds of sources were well recovered, but the new methods admittedly recovered circular/oval shapes much better than more involved shapes. 

Although it appears that both the total mass and the center of mass of the true sources are well recovered in our numerical simulations, it is currently an open problem whether we can prove that these quantities can be recovered for some typical shapes -- or alternatively, that some error estimates can be given.

Also, if more prior information concerning the shape(s) of the potential sources is known, it is currently unclear how this can be incorporated into the regularizer -- potentially through a more sophisticated choice of a PDE for generating the involved  Green's functions. 

\subsection*{AI Declaration}
The software used in this project was developed with the aid of Claude Opus 4.8.

\printbibliography

\end{document}